\documentclass[a4paper,centertags,oneside,12pt]{amsart}
\usepackage{mathrsfs}
\usepackage{appendix}
\usepackage{amssymb}
\usepackage{typearea}
\usepackage{pdfsync}
\usepackage{dsfont}
\usepackage{mathrsfs}
\usepackage[a4paper,top=1in,bottom=1in,left=1in,right=1in]{geometry}
\usepackage{hyperref}

\makeatletter
      \def\@setcopyright{}
      \def\serieslogo@{}
      \makeatother

\newcommand{\Complex}{\mathbb C}
\newcommand{\Real}{\mathbb R}
\newcommand{\ddbar}{\overline\partial}
\newcommand{\pr}{\partial}
\newcommand{\ol}{\overline}
\newcommand{\ov}{\overline}
\newcommand{\Td}{\widetilde}
\newcommand{\norm}[1]{\left\Vert#1\right\Vert}
\newcommand{\abs}[1]{\left\vert#1\right\vert}

\newcommand{\set}[1]{\left\{#1\right\}}
\newcommand{\To}{\rightarrow}
\DeclareMathOperator{\Ker}{Ker}

\DeclareMathOperator{\End}{End}

\newcommand{\field}[1]{\mathbb{#1}}

\newcommand{\C}{\field{C}}
\newcommand{\N}{\field{N}}

\newcommand{\cali}[1]{\mathscr{#1}}
\newcommand{\cC}{\cali{C}} 
\newcommand{\cO}{\cali{O}} \newcommand{\cE}{\cali{E}}
\newcommand{\cH}{\cali{H}}

\newcommand{\calig}[1]{\mathcal{#1}}
\newcommand\mO{\calig{O}}
\theoremstyle{plain}
\newtheorem{thm}{Theorem}[section]
\newtheorem{cor}[thm]{Corollary}

\newtheorem{lem}[thm]{Lemma}

\theoremstyle{definition}
\newtheorem{defn}[thm]{Definition}

\theoremstyle{remark}
\newtheorem{rem}[thm]{Remark}

\numberwithin{equation}{section}

\begin{document}
\title[]{Berezin-Toeplitz quantization for lower energy forms}
\author[]{Chin-Yu Hsiao}
\address{Institute of Mathematics, Academia Sinica  and National Center for Theoretical Sciences, 
6F, Astronomy-Mathematics Building, No.1, Sec.4, Roosevelt Road, Taipei 10617, Taiwan}
\thanks{The first author was supported by Taiwan Ministry of Science of Technology project 
104-2628-M-001-003-MY2,  the Golden-Jade fellowship of Kenda Foundation, 
Academia Sinica Career Development Award, and partially
funded through the Institutional Strategy of 
the University of Cologne within the German Excellence Initiative}
\email{chsiao@math.sinica.edu.tw or chinyu.hsiao@gmail.com} 
\author[]{George Marinescu}
\address{Universit{\"a}t zu K{\"o}ln,  Mathematisches Institut,
    Weyertal 86-90,   50931 K{\"o}ln, Germany
    \newline
    \mbox{\quad}\,Institute of Mathematics `Simion Stoilow', Romanian Academy,
Bucharest, Romania}\thanks{Second author partially supported by DFG funded 
projects SFB/TR 12, MA 2469/2-2, SFB TRR 191}
\email{gmarines@math.uni-koeln.de}
\keywords{Berezin-Toeplitz quantization,  
positive line bundle, Kodaira Laplace operator,
local holomorphic Morse inequalities, spectral asymptotics, Melin-Sj\"ostrand stationary
phase method}

\begin{abstract}
Let $M$ be an arbitrary complex manifold and let $L$ be a Hermitian
holomorphic line bundle over $M$. 
We introduce the Berezin-Toeplitz quantization of the open set of $M$
where the curvature on $L$ is non-degenerate. In particular, we quantize 
any manifold admitting a positive line bundle.
The quantum spaces 
are the spectral spaces corresponding to $\left[0,k^{-N}\right]$, where $N>1$ is fixed, 
of the Kodaira Laplace operator acting on forms with values
in tensor powers $L^k$. We establish the asymptotic expansion of 
associated Toeplitz operators and their composition in the semiclassical limit $k\to\infty$ and we
define the corresponding star-product. 
If the Kodaira Laplace operator has a certain spectral gap this method
yields quantization by means of harmonic forms.
As applications, we obtain the 
Berezin-Toeplitz quantization for semi-positive and big line bundles. 
\end{abstract}

\maketitle \tableofcontents
\section{Introduction and statement of the main results}\label{s-intro}

The aim of the geometric quantization theory of Kostant and Souriau 
is to relate the classical observables 
(smooth functions) on a phase space (a symplectic manifold) to 
the quantum observables (bounded linear operators) on the quantum space 
(sections of a line bundle). 
Berezin-Toeplitz
quantization is a particularly efficient version of the geometric quantization theory
\cite{BFFLS,Berez:74,Fedo:96,Kos:70,Ma10,Sou:70}.
Toeplitz operators and more generally Toeplitz structures were introduced in 
geometric quantization by
Berezin \cite{Berez:74} and Boutet de Monvel-Guillemin \cite{BouGu81}.
We refer to \cite{Ma10,MM11,Schlich:10} for reviews of Berezin-Toeplitz quantization.

The setting of Berezin-Toeplitz quantization
on  K{\"a}hler manifolds is the following. 
Let $(M,\omega, J)$ be a K{\"a}hler manifold of
$\dim_{\C}M=n$ with K{\"a}hler form $\omega$ and complex structure $J$.
Let $(L, h)$ be a holomorphic Hermitian line bundle on $X$, and
let $\nabla ^L$ be the holomorphic Hermitian
connection on $(L,h)$ with curvature
 $R^L=(\nabla^L)^2$.
 We assume that $(L,h,\nabla^L)$ is a prequantum line bundle, i.e.,
%satisfies the {\em prequantization condition}, that is
 \begin{align} \label{toe2.1}
     \omega= \frac{\sqrt{-1}}{2 \pi} R^L.
\end{align}
Let $g^{TM}:=\omega(\cdot,J\cdot)$ be the $J$-Riemannian 
metric on $TM$. The Riemannian volume form of $g^{TM}$ is denoted by $dv_M$.
On the space of smooth sections with compact support $\cC^{\infty}_0(M,L^k)$ 
we introduce the $L^2$-scalar product associated to the metrics $h$ and
the Riemannian volume form $dv_M$ by
\begin{equation}\label{lm2.0}
\big\langle s_1,s_2 \big\rangle =\int_M\big\langle s_1(x),
s_2(x)\big\rangle_{h^k}\,dv_M(x)\,.
\end{equation}
The completion of $\cC^{\infty}_0(M,L^k)$ with respect to \eqref{lm2.0} is denoted as 
usual by $L^{2}(M,L^k)$. We denote by $H^{0}_{(2)}(M,L^k)$ the closed subspace of $L^{2}(M,L^k)$
consisting of holomorphic sections. The Bergman projection is the orthogonal projection
$P_{k}:L^{2}(M,L^k)\to H^{0}_{(2)}(M,L^k)$\,.
For a bounded function $f\in\cC^{\infty}(M)$, set %we denote
%--------------------------------------------------------------------
\begin{equation}\label{toe2.4}
T_{f,\,k}:L^2(M,L^k)\longrightarrow L^2(M,L^k)\,,
\quad T_{f,\,k}=P_k\,f\,P_k\,,
\end{equation}
%----------------------------------------------------------------------------
where the action of $f$ is the pointwise multiplication by $f$.
The map which associates to  $f\in \cC^{\infty}(M)$
the family of bounded operators $\{T_{f,\,k}\}$ on $L^2(M,L^k)$ is called
the  {\em Berezin-Toeplitz quantization}\/.
A {\em Toeplitz operator}\index{Toeplitz operator}
is a sequence $\{T_k\}_{k\in\N}$ of bounded linear endomorphisms of 
$L^2(M,L^k)$ verifying $T_{k}=P_k\,T_k\,P_k$\,,
%----------------------------------------------------------------------------
such that there exist a sequence $g_\ell\in\cC^\infty(M)$ such that
for any $p\geqslant0$, there exists $C_p>0$
%for any $p\in \N$, we have $T_{p}=P_p\,T_p\,P_p$, and
 with
%----------------------------------------------------------------------------
%\begin{equation}\label{toe2.3}
$\|T_k-\sum_{\ell=0}^pT_{g_\ell,\,k}\, k^{-\ell}\|_{op}
\leqslant C_p\,\, k^{-p-1}$ for any $k\in \N$,
%\end{equation}
where $\norm{\,\cdot\,}_{op}$ denotes the operator norm on the space of
bounded operators.

Assume now that $(M,\omega,J)$ is a compact K\"ahler manifold.
Then Bordemann-Meinrenken-Schlichenmaier \cite{BMS94} 
and Schlichenmaier \cite{Schlich:00} (using the analysis of Toeplitz structures 
of Boutet de Monvel-Guillemin \cite{BouGu81}), Charles \cite{Cha03} 
(inspired by semiclassical analysis of Boutet de Monvel-Guillemin \cite{BouGu81}) and 
Ma-Marinescu \cite{MM08b} (using the expansion of the Bergman kernel \cite{DLM04a,MM07})
showed that the composition of two Toeplitz operators is a Toeplitz operator, 
in the sense that for any $f,g\in \cC^\infty(M)$,
 the product $T_{f,\,k}\,T_{g,\,k}$ has an asymptotic expansion 
\begin{equation}\label{0.2}
T_{f,\,k}\,T_{g,\,k}=\sum_{p=0}^\infty T_{C_p(f,\,g), k}\, k^{-p}
+\mO(k^{-\infty})
\end{equation}
where $C_p$ are 
bidifferential operators of order $\leqslant 2r$, 
satisfying $C_0(f,g)=fg$ and 
$C_1(f,g)-C_1(g,f)=\sqrt{-1}\,\{f,g\}$. Here 
$\{ \,\cdot\, , \,\cdot\, \}$ is the Poisson bracket on 
$(M,2\pi \omega)$. We deduce from \eqref{0.2},  
\begin{equation}\label{0.3}
[T_{f,\,k}\,,T_{g,\,k}]=\frac{\sqrt{-1}}{\, k}T_{\{f,g\},\,k}
+\mO(k^{-2})\,.
\end{equation}
In \cite{MM07,MM08b} Ma-Marinescu extended the Berezin-Toeplitz quantization 
to symplectic manifolds and orbifolds
by using as quantum space the kernel of the Dirac operator 
acting on powers of the prequantum line bundle twisted with 
an arbitrary vector bundle with arbitrary metric on manifolds.
Recently, Charles \cite{Cha14} introduced a semiclassical approach for symplectic 
manifolds inspired from the Boutet de Monvel-Guillemin theory \cite{BouGu81}.

In this paper we extend the Berezin-Toeplitz quantization in several directions.
Firstly, we consider an arbitrary Hermitian manifold $(M,\Theta,J)$ endowed with
arbitrary Hermitian holomorphic line bundle $(L,h)$ and we
quantize the open set $M(0)$ where the curvature of $(L,h)$ is positive.
Since there are no holomorphic $L^2$ sections in general, we use as
quantum spaces the spectral spaces of the Kodaira Laplacian $\Box^{(0)}_{k}$ on
$L^2(M,L^k)$, corresponding to energy less than $k^{-N}$, $N>1$ fixed, 
decaying to $0$ polynomially in $k$, as $k\to\infty$.  
Secondly, we consider the same construction for the Kodaira Laplacian $\Box^{(q)}_{k}$
acting on $(0,q)$-forms. In this case we quantize the open set $M(q)$ 
where the curvature of $(L,h)$ is non-degenerate and has exactly $q$ negative eigenvalues
(and hence $n-q$ positive ones). 
Quantization using $(0,q)$-forms was introduced in \cite[\S 8.2]{MM07} for bundles
with mixed curvature of signature $(q,n-q)$ everywhere on a compact manifold. It was based on the
asymptotic of Bergman kernel developed in \cite{MM06}.

The idea underlying the construction used in this paper comes from the 
local holomorphic Morse inequalities \cite{Bismut87,De:85,HM12,MM07}. Roughly speaking, the harmonic
$(0,q)$-forms with values in $L^{k}$ tend to concentrate on $M(q)$ as $k\to\infty$.
More precisely, the semiclassical limit of the kernel of the spectral projectors
considered above was determined in \cite[Theorem\,1.1]{HM12}, see also \cite[Theorems\,1.6\,--1.10]{HM12}
for important particular cases. This is the main technical ingredient
used in this paper, which is in turn based on techniques of microlocal and semiclassical analysis,
see e.\,g.\ \cite{DS99,MS74}, especially the stationary phase method of Melin-Sj\"ostrand \cite{MS74}.

We now formulate the main results. We refer to Section~\ref{s:prelim} for some 
standard notations and terminology used here.  
We are working in the following general setting:
\medskip

\begin{itemize}
\item[(A)] $(M,\Theta,J)$ is a Hermitian manifold of
complex dimension $n$, where $\Theta$ is a smooth positive $(1,1)$-form 
and $J$ is the complex structure.
Moreover,
$(L,h)$ is a holomorphic Hermitian line bundle over $M$, where
$h$ is the Hermitian fiber metric on $L$, and
$q\in\{0,1,\ldots,n\}$.
\smallskip
\item[(B)]  $f,g\in\cC^\infty(M)$ are smooth bounded functions.
\end{itemize}
\medskip

Let $g^{TM}_\Theta(\cdot,\cdot)=\Theta(\cdot,J\cdot)$ be the Riemannian 
metric on $TM$ induced by $\Theta$ and $J$ and let $\langle\,\cdot\,,\cdot\,\rangle$ be the Hermitian
metric on $\Complex TM:=TM\otimes_\Real\Complex$
induced by $g^{TM}_\Theta$. The Riemannian volume
form $dv_M$ of $(M,\Theta)$ satisfies $dv_M= \Theta^n/n!$\,. For every $q=0,1,\ldots,n$, 
the Hermitian metric $\langle\,\cdot\,,\cdot\,\rangle$ on $TM\otimes_\Real\Complex$ 
induces a Hermitian metric $\langle\,\cdot\,,\cdot\,\rangle$ on $\Lambda^{0,q}(T^*M)$ 
the bundle of $(0,q)$ forms of $M$.

We will denote by
$\phi$ the local weights of the Hermitian metric $h$ on $L$ (see \eqref{s1-e1}). 
Let $\nabla ^L$ be the holomorphic Hermitian connection on $(L,h)$  with curvature $R^L=(\nabla^L)^2$.
We will identify the curvature form $R^L$ with the Hermitian matrix
$\dot{R}^L \in\cC^\infty(M,\End(T^{1,0}M))$ satisfying
for every $U, V\in T^{1,0}_xM$, $x\in M$,
\begin{equation}\label{s0-e0}
\langle\,R^L(x)\,,\,U\wedge\ov{V}\,\rangle = \langle\,\dot{R}^L(x)U,V\,\rangle.
\end{equation} 
Let $\det\dot{R}^L(x):=\mu_1(x)\ldots\mu_n(x)$, where $\{\mu_j(x)\}_{j=1}^n$, are the eigenvalues of $\dot R^L$ with respect to 
$\langle\,\cdot\,,\cdot\,\rangle$.
For $j\in\{0,1,\dots,n\}$, let 
\begin{equation} \label{s0-e2}
\begin{split}
M(j)=\big\{x\in M;\, &\mbox{$\dot{R}^L(x)$ is non--degenerate and has exactly $j$ negative eigenvalues}\big\}.
\end{split}
\end{equation}
%\begin{equation} \label{s0-e2}
%\begin{split}
%M(q)=\big\{x\in M;\, &\mbox{$\dot{R}^L(x)\in \End(T^{1,0}_x M)$ is non--degenerate}\\
%&\quad\mbox{and has exactly $j$ negative eigenvalues}\big\}.
%\end{split}
%\end{equation}
We denote by $W$ the subbundle of rank $j$ of $T^{1,0}M|_{M(j)}$ generated 
by the eigenvectors corresponding to negative eigenvalues of $\dot{R}^L$. 
Then $\det\ov{W}^{\,*}:=\Lambda^j\ov{W}^{\,*}\subset \Lambda^{0,j}(T^*M)|_{M(j)}$ 
is a rank one sub-bundle. Here $\ov{W}^{\,*}$ is the dual bundle of the complex 
conjugate bundle of $W$ and $\Lambda^j\ov{W}^{\,*}$ is the vector space of all 
finite sums of $v_1\wedge\cdots\wedge v_j$, 
$v_1,\ldots,v_j\in\ov{W}^{\,*}$. We denote by $I_{\det\ov{W}^{\,*}}\in\End(\Lambda^{0,j}(T^*M))$ 
the orthogonal projection from $\Lambda^{0,j}(T^*M)$ onto $\det\ov{W}^{\,*}$.

For $k>0$, let $(L^k,h^k)$ be the $k$-th tensor power of the line bundle $(L,h)$.  
Let $(\,\cdot\,,\cdot\,)_{k}$ and  $(\,\cdot\,,\cdot\,)$ denote the global $L^2$ 
inner products on $\Omega^{0,q}_0(M, L^k)$ and $\Omega^{0,q}_0(M)$ induced by 
$\langle\,\cdot\,,\cdot\,\rangle$ and $h^k$ respectively (see \eqref{toe2.2}). 
We denote by $L^2_{(0,q)}(M,L^k)$ and $L^2_{(0,q)}(M)$ the completions of 
$\Omega^{0,q}_0(M, L^k)$ and $\Omega^{0,q}_0(M)$ with respect to 
$(\,\cdot\,,\cdot\,)_{k}$ and $(\,\cdot\,,\cdot\,)$ respectively.  

Let $\Box_{k}^{(q)}$ be the Kodaira Laplacian acting on $(0,q)$--forms 
with values in $L^k$, cf.\ \eqref{e:Kod_Lap}. We denote by the same symbol 
$\Box_{k}^{(q)}$ the Gaffney extension of the Kodaira Laplacian, cf.\ \eqref{Gaf1}. 
It is well-known that $\Box^{(q)}_k$ is self-adjoint and the spectrum 
of $\Box^{(q)}_k$ is contained in $\ol\Real_+$ (see \cite[Proposition\,3.1.2]{MM07}). 
For a Borel set $B\subset\Real$ let $E(B)$ be the spectral projection of $\Box^{(q)}_k$ 
corresponding to the set $B$, where $E$ is the spectral measure of $\Box^{(q)}_k$ 
(see Davies~\cite[Section\,2]{Dav95}) and for $\lambda\in\Real$ 
we set $E_\lambda=E\big((-\infty,\lambda]\big)$ and
\begin{equation} \label{s1-specsp}
\cE^q_\lambda(M, L^k)=\operatorname{Range}E_\lambda\subset L^2_{(0,q)}(M,L^k)\,.
\end{equation}
If $\lambda=0$, then $\cE^q_0(M, L^k)=\Ker\Box_{k}^{(q)}=:\cH^q(M,L^k)$
is the space of global harmonic sections. The \emph{spectral projection} of 
$\Box_{k}^{(q)}$ is the orthogonal projection
\begin{equation}\label{e:orto_proj}
P^{(q)}_{k,\lambda}:L^2_{(0,q)}(M,L^k)\to\cE^q_\lambda(M, L^k)\,.
\end{equation}
Fix $f\in\cC^\infty(M)$ be a bounded function. Let $\lambda\geq0$. 
The \emph{Berezin-Toeplitz quantization for $\cE^q_\lambda(M, L^k)$} is the operator 
\begin{equation}\label{e-gue140419}
T^{(q),f}_{k,\lambda}:=P^{(q)}_{k,\lambda}\circ f\circ P^{(q)}_{k,\lambda}: 
L^2_{(0,q)}(M,L^k)\to\cE^q_\lambda(M, L^k).
\end{equation}

Let $T^{(q),f}_{k,\lambda}(\,\cdot\,,\cdot)$ be the Schwartz kernel of 
$T^{(q),f}_{k,\lambda}$, see \eqref{e:kern0}, \eqref{e:kern1}. 
Since $\Box^{(q)}_k$ is elliptic, we have 
$T^{(q),f}_{k,\lambda}(\,\cdot\,,\,\cdot\,)\in 
\cC^\infty\big(M\times M,(L^k\otimes\Lambda^{0,q}(T^*M))
\boxtimes(L^k\otimes\Lambda^{0,q}(T^*M))^*\big)$.

Let $A_k:L^2_{(0,q)}(M,L^k)\To L^2_{(0,q)}(M,L^k)$ be a 
$k$-dependent continuous operator with smooth kernel $A_k(x,y)$ 
and let $D_0, D_1\Subset M$ be open trivializations with trivializing sections 
$s$ and $\widehat s$ respectively. In this paper, 
we will identify $A_k$ and $A_k(x,y)$ on $D_0\times D_1$ with the localized operators 
$A_{k,s,\widehat s}$ and $A_{k,s,\widehat s}(x,y)$ respectively (see \eqref{e-gue141118}).

The first main result of this work is the following. 

\begin{thm}\label{t-gue140521}
Under the assumptions (A) and (B) let $j\in\set{0,1,\ldots,n}$
and $D_{0},D_{1}\Subset M$ on which $L$ is trivial.
%
%Let $(M,\Theta,J)$ be a Hermitian manifold of
%complex dimension $n$ and $(L,h)$ be a holomorphic Hermitian line bundle over $M$.
%Fix $q\in\set{0,1,\ldots,n}$ and consider open trivializations $D_{0},D_{1}\Subset M$. 
Suppose that
one of the following conditions is fulfilled:
\\[2pt]
(i)   $D_0\Subset M(j)$ and $j\neq q$,
\\[2pt]
(ii) $D_0\Subset M(q)$ and $\ol D_0\bigcap\ol D_1=\emptyset$. 
\\[2pt]
Then, for every $N>1$, $m\in\mathbb N$, there exists $C_{N,m}>0$ independent of $k$ 
such that 
\begin{equation}\label{e1.1}
\abs{T^{(q),f}_{k,k^{-N}}(x,y)}_{\cC^m(D_0\times D_1)}\leq C_{N,m}k^{2n-\frac{N}{2}+2m}.
\end{equation}
If $D_0\Subset M(q)$ there exists 
a symbol 
\[b_f(x,y,k)\in S^{n}(1;D_0\times D_0, \Lambda^{0,q}(T^*M)
\boxtimes (\Lambda^{0,q}(T^*M))^*)
\] and a phase function
$\Psi\in\cC^\infty(D_0\times D_0)$ such that for every $N>1$, $m\in\mathbb N$, there exists
$\Td C_{N,m}>0$ independent of $k$ such that 
\begin{equation}\label{e1.2}
\abs{T^{(q),f}_{k,k^{-N}}(x,y)-e^{ik\Psi(x,y)}b_f(x,y,k)}_{\cC^m(D_0\times D_0)}\leq
\Td C_{N,m}k^{2n-\frac{N}{2}+2m},
\end{equation}
where $b_f(x,y,k)\sim\sum^\infty_{j=0}b_{f,j}(x,y)k^{n-j}$ in the sense of Definition \ref{s1-d1} and  
\begin{equation}  \label{e-gue140521I}
\begin{split}
&b_{f,0}(x,x)=(2\pi)^{-n}f(x)\,\big|\!\det\dot{R}^L(x)\big|I_{\det\ov{W}^{\,*}}(x),\ \ x\in D_0,
\end{split}
\end{equation}
and 
\begin{equation}\label{prop_psi}
\begin{split}
&\Psi(x,y)\in\cC^\infty(D_0\times D_0),\ \ \Psi(x,y)=-\ol\Psi(y,x)\,,\\
& \exists\, c>0:\ {\rm Im\,}\Psi\geq c\abs{x-y}^2\,,\ \Psi(x,y)=0\Leftrightarrow x=y \,.
\end{split}
\end{equation}
\end{thm}
%---------------------
We collect more properties for the phase $\Psi$ in Theorem \ref{t-gue140523}. 
%---------------------
The results says that, roughly speaking, the Toeplitz kernel $T^{(q),f}_{k,k^{-N}}(\cdot,\cdot)$
acting on $(0,q)$-forms, decays rapidly as $k\to\infty$ outside $M(q)$ and off-diagonal, and
admits an asymptotic expansion on the set $M(q)$.

Let $\ell,m\in\N$ be fixed and choose $N\geq2(n+\ell+2m+1)$. Then we deduce from \eqref{e1.2} that
\begin{equation}\label{e1.31}
T^{(q),f}_{k,k^{-N}}(x,x)=\sum_{r=0}^\ell b_{f,r}(x,x)k^{n-r}
+O(k^{n-\ell-1})\:\:\text{in $\cC^m(D_0)$, $D_0\Subset M(q)$}.
\end{equation}
Note that if $M$ is compact complex manifold endowed with a positive line  
bundle $L$ (i.\,e.\ $M(0)=M$) we have by \cite[Theorem\,0.1]{MM12} for any $\ell,m\in\N$, 
\begin{equation}\label{e1.32}
T^{(0),f}_{k,0}(x,x)=\sum_{r=0}^\ell b_{f,r}(x,x)k^{n-r}
+O(k^{n-\ell-1})\:\:\text{in $\cC^m(M)$}.
\end{equation}
Actually, in this case, due to the spectral gap of the Kodaira Laplacian \cite[Theorem\,1.5.5]{MM07}
we have $T^{(0)}_{f,k,k^{-N}}=T^{(0)}_{f,k,0}$ for $k$ large enough, so \eqref{e1.31} follows from
\eqref{e1.32}.
The expansion \eqref{e1.31} bears resemblance to the expansion of the Toeplitz kernels for functions
$f\in\cC^p(M)$ (see \cite[(3.19)]{BMMP}), for arbitrary $p\in\N$. In \eqref{e1.31} the upper bound for
the order of expansion $\ell$ is due to the size $k^{-N}$ of the spectral parameter, while in case of 
symbols of class $\cC^p(M)$ is due to the order of differentiability $p$.

It is interesting to note that Theorem \ref{t-gue140521} and the following results provide
a generalization of various expansions for Toeplitz operators in the case of an arbitrary 
complex manifold endowed with a positive line bundle. In this case we have simply $M=M(0)$.
Of course, in such generality, the quantum spaces have to be spectral spaces $\cE^q_{k^{-N}}(M, L^k)$.

The first three coefficients of the kernel expansions of Toeplitz operators and of
their composition for the quantization of a compact K\"ahler manifold with positive line bundle
were calculated by Ma-Marinescu \cite{MM12} in the presence of a twisting
vector bundle $E$ and later by Hsiao \cite{Hsiao09} for $E=\C$. Both
\cite{Hsiao09,MM12} work with a general not necessarily K\"ahler base metric $\Theta$
which might not be polarized, that is, $\Theta\neq\frac{\sqrt{-1}}{2\pi}R^L$
in general.  We will calculate the top coefficients $b_{f,1}(x,x)$ and $b_{f,2}(x,x)$ of the expansion
\eqref{e1.2} in Section \ref{s-gue140729}. 
The coefficients $b_{f,0}(x,x)$ and $b_{f,1}(x,x)$ were given in \cite{Cha03} for
$E=\C$ and $\Theta=\frac{\sqrt{-1}}{2\pi}R^L$.
It is a remarkable manifestation of universality, that
the coefficients for the quantization with holomorphic sections \cite{Hsiao09,MM12} 
and for the quantization with spectral spaces used in this paper are given by the same formulas.  
We refer to \cite{Xu12}
for an interpretation in graph-theoretic terms of the Toeplitz kernel expansion. 
The formulas from \cite{MM12} play an essential role in the quantization 
of the Mabuchi energy \cite{Fine:12} and Laplace operator \cite{KMS}. 
%-------------------------------------
On the set where the curvature of $L$ is degenerate we have the following behavior.
%-------------------------------------
\begin{thm} \label{s1-maindege}
Under the general assumptions (A) and (B),
%Let $(M,\Theta,J)$, $(L, h)$ and $f\in\cC^\infty(M)$ be as in Theorem~\ref{t-gue140521}. 
set 
\[M_{\mathrm{deg}}=\set{x\in M;\, \mbox{$\dot{R^L}$ is degenerate at $x\in M$}}.\]
Then for every $x_0\in M_{\mathrm{deg}}$, $\varepsilon>0$, $N>1$ and every $j\in\{0,1,\dots,n\}$ , there 
exist a neighborhood $U$ of $x_0$ and $k_0>0$, such that for all $k\geq k_0$ we have
\begin{equation} \label{s1-e3main}
\abs{T^{(j),f}_{k,k^{-N}}(x,x)}\leq \varepsilon k^n,\ \ x\in U.
\end{equation}
\end{thm}

We consider next the composition of two Berezin-Toeplitz quantizations. 
We have first the following expansion of the kernels
of Toeplitz operators.
%Let $f, g\in C^\infty(M)$ be bounded functions. 
%For $\lambda\geq0$, put 
%\begin{equation}\label{e-gue140729I}
%T^{(q),f,g}_{k,\lambda}:=T^{(q),f}_{k,\lambda}\circ T^{(q),g}_{k,\lambda}: L^2_{(0,q)}(M,L^k)\to\cE^q_\lambda(M, L^k).
%\end{equation}
%Let $s$, $\widehat s$ be local trivializing sections of $L$ on $D_0\Subset M$, $D_1\Subset M$ 
%respectively, $\abs{s}^2_{h}=e^{-2\phi}$, $\abs{\widehat s}^2_{h}=e^{-2\widehat\phi}$. 
%As \eqref{e-gue140521}, let 
%\begin{equation}\label{e-gue140729II}
%\begin{split}
%T^{(q),f,g}_{k,\lambda,s,\widehat s}:\Omega^{0,q}_0(D_1)&\To\Omega^{0,q}(D_0),\\
%u&\To s^{-k}e^{-k\phi}(T^{(q),f,g}_{k,\lambda}\widehat s^ke^{k\widehat\phi}u),
%\end{split}
%\end{equation}
%and let 
%$T^{(q),f,g}_{k,\lambda,s,\widehat s}(x,y)\in C^\infty(D_0\times D_1,T^{*0,q}_yM\otimes T^{*0,q}_xM)$ 
%be the distribution kernel of $T^{(q),f,g}_{k,\lambda,s,\widehat s}$. We formally write 
%\[T^{(q),f}_{k,\lambda,s,\widehat s}= s^{-k}e^{-k\phi}T^{(q),f}_{k,\lambda}\widehat s^ke^{k\widehat\phi}.\]
%For $s=\widehat s$, $D_0=D_1$, we write $T^{(q),f,g}_{k,\lambda,s}:=T^{(q),f,g}_{k,\lambda,s,s}$, 
%$T^{(q),f,g}_{k,\lambda,s}(x,y):=T^{(q),f,g}_{k,\lambda,s,s}(x,y)$. For $\lambda=0$, we write 
%$T^{(q),f,g}_{k,s,\widehat s}=T^{(q),f,g}_{k,0,s,\widehat s}$, 
%$T^{(q),f,g}_{k,s,\widehat s}(x,y)=T^{(q),f,g}_{k,0,s,\widehat s}(x,y)$, 
%$T^{(q),f,g}_{k,s}=T^{(q),f,g}_{k,0,s}$, $T^{(q),f,g}_{k,s}(x,y)=T^{(q),f,g}_{k,0,s}(x,y)$.

\begin{thm}\label{t-gue140729II}
Under the assumptions (A) and (B) let $j\in\set{0,1,\ldots,n}$
and $D_{0},D_{1}\Subset M$ on which $L$ is trivial.
Suppose that
one of the following conditions is fulfilled:
\\[2pt]
(i)   $D_0\Subset M(j)$ and $j\neq q$,
\\[2pt]
(ii) $D_0\Subset M(q)$ and $\ol D_0\bigcap\ol D_1=\emptyset$. 
\\[2pt]
Then, for every $N>1$, $m\in\mathbb N$, there exists $C_{N,m}>0$ independent of $k$ such that 
\begin{equation}
\abs{\big(T^{(q),f}_{k,k^{-N}}\circ T^{(q),g}_{k,k^{-N}}\big)(x,y)}_{\cC^m(D_0\times D_1)}\leq
 C_{N,m}k^{3n-\frac{N}{2}+2m}.
\end{equation}
%\[\abs{T^{(q),f,g}_{k,k^{-N},s,\widehat s}(x,y)}_{C^m(D_0\times D_1)}\leq C_{N,m}k^{3n-\frac{N}{2}+2m}.\]
If $D_0\Subset M(q)$ there exists 
a symbol 
\[
b_{f,g}(x,y,k)\in S^{n}(1;D_0\times D_0,\Lambda^{0,q}(T^*M)
\boxtimes(\Lambda^{0,q}(T^*M))^*)
\]
%and a phase function $\Psi\in\cC^\infty(D_0\times D_0)$ 
such that for every $N>1$, $m\in\mathbb N$, there exists
$\Td C_{N,m}>0$ independent of $k$ such that 
%\
%Assume that $D_0\Subset M(q)$. Then, for every $N>1$, $m\in\mathbb N$, there exists $\Td C_{N,m}>0$ 
%independent of $k$ such that 
\begin{equation}\label{e-gue141118I}
\abs{\big(T^{(q),f}_{k,k^{-N}}\circ T^{(q),g}_{k,k^{-N}}\big)(x,y)
-e^{ik\Psi(x,y)}b_{f,g}(x,y,k)}_{\cC^m(D_0\times D_0)}\leq\Td C_{N,m}k^{3n-\frac{N}{2}+2m},\end{equation}
where $b_{f,g}(x,y,k)\sim\sum^\infty_{j=0}b_{f,g,j}(x,y)k^{n-j}$ in the sense of Definition 
\ref{s1-d1} and 
\begin{equation}  \label{e-gue140521Iab}
\begin{split}
%&b_{f,g}(x,y,k)\in S^{n}(1;D_0\times D_0,T^{*0,q}_yM\otimes T^{*0,q}_xM), \\
%&b_{f,g}(x,y,k)\sim\sum^\infty_{j=0}b_{j,f,g}(x,y)k^{n-j}\text{ in }
%S^{n}(1;D_0\times D_0,T^{*0,q}_yM\otimes T^{*0,q}_xM), \\
%&b_{j,f,g}(x,y)\in C^\infty(D_0\times D_0,T^{*0,q}_yM\otimes T^{*0,q}_xM),\ \ j=0,1,2,\ldots,\\
&b_{f,g,0}(x,x)=(2\pi)^{-n}f(x)g(x)\big|\det\dot{R}^L(x)\big|I_{\det\ov{W}^{\,*}}(x),\ \ x\in D_0,
\end{split}
\end{equation}
and $\Psi(x,y)\in\cC^\infty(D_0\times D_0)$ is as in Theorem~\ref{t-gue140521}. 
\end{thm}

It should be noticed that Theorem~\ref{t-gue140729II} holds for any Hermitian manifold $M$, 
not necessarily compact.  
Note that the estimates in Theorem~\ref{t-gue140729II} involve the power $k^{3n-\frac{N}{2}+2m}$
compared to  $k^{2n-\frac{N}{2}+2m}$ in Theorem \ref{t-gue140521}.
We will explain why there are different exponents $3n$ and $2n$ in 
the proof of Theorem \ref{t-gue140729}.

We will calculate the top coefficients $b_{f,1}(x,x)$, $b_{f,2}(x,x)$ and 
$b_{f,g,1}(x,x)$, $b_{f,g,2}(x,x)$ of the expansions \eqref{e1.2} and 
\eqref{e-gue141118I} in Section \ref{s-gue140729} 
(see Theorem~\ref{t-gue140523I} and Theorem~\ref{s1-tmaincomp}).

We come now to the asymptotic expansion of the composition of two Toeplitz operators
in the operator norm.
Let $A_k:L^2(M,L^k)\To L^2(M,L^k)$ be $k$-dependent continuous operator. 
We say that 
$A_k=\mO(k^m+k^{m_1})$ as $k\to\infty$, locally in the $L^2$ operator norm if for any 
$\chi, \chi_1\in\cC^\infty_0(M)$, there exists $C>0$ independent of $k$ such that 
$\norm{\chi A_k\chi_1}_{op}\leq C(k^m+k^{m_1})$, for $k$ large, 
where $\norm{\cdot}_{op}$ denotes the $L^2$ operator norm.
We also denote by $\langle\,\cdot\,,|\,\cdot\,\rangle_\omega$ the Hermitian metric on 
$T^*M\otimes_\Real\Complex$ induced by $\omega:=\frac{\sqrt{-1}}{2\pi}R^L$.
%----------------
\begin{thm}\label{t-comp}
Under the assumptions (A) and (B) suppose moreover that
%Let $(M,\Theta,J)$ and $(L, h)$ be as in Theorem~\ref{t-gue140521}.
$f,g\in \cC^\infty(M)$ have compact support in $M(0)$.
Then for every $N>1$, 
there exist functions $C_p(f,g)\in\cC^\infty_0(M(0))$, $p\in\N$, such that
for  any $\ell\in\N$
the product $T^{(0),f}_{k,k^{-N}}\,T^{(0),g}_{k,k^{-N}}$ has the asymptotic expansion 
%----------------
\begin{equation}\label{0.2.1}
T^{(0),f}_{k,k^{-N}}\circ T^{(0),g}_{k,k^{-N}}=
\sum_{p=0}^\ell T^{(0),C_p(f,\,g)}_{k,k^{-N}}\, k^{-p}
+\mO(k^{-\ell-1}+k^{3n-\frac{N}{2}}),\ \ k\to\infty,
\end{equation}
%----------------
locally in the $L^2$ operator norm. Moreover, 
%----------------
\begin{equation}\label{0.2.1a}
C_0(f,g)=fg\,,\quad C_1(f,g)=-\frac{1}{2\pi}\langle\,\pr f\,,\,\pr\ol g\,\rangle_\omega\,,
\end{equation}
%----------------
and therefore the commutator of two Toeplitz operators satisfies
%----------------
\begin{equation}\label{e-gue160717}
\Big[T^{(0),f}_{k,k^{-N}},T^{(0),g}_{k,k^{-N}}\Big]=\frac{\sqrt{-1}}{k}T^{(0),\{f,g\}}_{k,k^{-N}}
+\mO(k^{-2}+k^{3n-\frac{N}{2}}) ,\ \ k\to\infty,
\end{equation}
%----------------
where $\{f,g\}$ is the Poisson bracket on $(M(0),2\pi\omega)$.
%$\{f,g\}:=\frac{\sqrt{-1}}{2\pi}\langle\,\pr f\,|\,\pr\ol g\,\rangle_\omega$, .  
%Note that $\{f,g\}$ in \eqref{e-gue160717} and in \eqref{0.3} are the same. 
%Indeed,  $\{f,g\}$ is the Poisson bracket on $(M(0),\omega)$. 
%, are the universal coefficients from \eqref{0.2}.
\end{thm}
%----------------
We will give formulas for the coefficients $C_j(f,g)$, $j=0,1,3,$ in Corollary \ref{cor:C}.
They have the same form as those
in the expansion of the Toeplitz operators acting on spaces of holomorphic sections, see
\cite[(1.29)]{Hsiao09}, \cite[(0.20)]{MM12}.
Formula \eqref{e-gue160717} represents the semiclassical correspondence principle between 
classical and quantum observables. Theorem \ref{t-comp} allows us to
introduce a star-product on the set where a line positive is positive, see Remark 
\ref{rem_sp}.

As an application of Theorem~\ref{t-gue140521} and Theorem~\ref{s1-maindege}, we obtain:  

\begin{thm}\label{localmorse}
Assume (A) and (B) are fulfilled and
%Let $(M,\Theta)$, $(L, h)$ and $f\in\cC^\infty(M)$ be as above.
let $N>2n$. Then 
\begin{equation} \label{s0-e4mainmore}
T^{(q),f}_{k,k^{-N}}(x,x)
=k^n(2\pi)^{-n}\abs{{\rm det\,}\dot{R}^L(x)}f(x)I_{\det\ov{W}^{\,*}}(x)+O(k^{n-1})\,,\:k\to\infty,
\end{equation} 
locally uniformly on $M(q)$, for every $D\Subset M$, there exists $C_D>0$ independent of $k$ such that
\begin{equation} \label{s0-e4mainmoreI}
\abs{T^{(q),f}_{k,k^{-N}}(x,x)}\leq C_Dk^n,\ \ \forall x\in D,
\end{equation} 
and if $\mathds{1}_{M(q)}$ denotes the characteristic function of $M(q)$, we have the pointwise convergence:
\begin{equation} \label{s0-e4main}
\lim_{k\To\infty}k^{-n}T^{(q),f}_{k,k^{-N}}(x,x)
=(2\pi)^{-n}f(x)\abs{{\rm det\,}\dot{R}^L(x)}\mathds{1}_{M(q)}(x)I_{\det\ov{W}^{\,*}}(x),\ \ x\in M.
\end{equation}
\end{thm} 

Since $L^k_x\boxtimes(L^k_x)^*\cong\Complex$ we can identify $T^{(q),f}_{k,\lambda}(x,x)$ 
to an element of $\End(\Lambda^{0,q}_x(T^*M)$. Then
\begin{equation}\label{e:state}
M\ni x\longmapsto T^{(q),f}_{k,\lambda}(x,x)\in\End(\Lambda^{0,q}_x(T^*M))
\end{equation}
is a smooth section of $\End(\Lambda^{0,q}(T^*M))$. Let ${\rm Tr\,}T^{(q),f}_{k,\lambda}(x,x)$ 
denote the trace of $T^{(q),f}_{k,\lambda}(x,x)$ with respect to $\langle\,\cdot\,,\cdot\,\rangle$.
%The trace of $T^{(q),f}_{k,\lambda}(x,x)$ is given by
%\[{\rm Tr\,}T^{(q),f}_{k,\lambda}(x,x):=
%\sum^d_{j=1}\big\langle\,T^{(q),f}_k(x,x)\,e_{J_j}(x)\,|\,e_{J_j}(x)\,\big\rangle,\]
%where $e_{J_1},\ldots,e_{J_d}$ is a local orthonormal basis of 
%$T^{*0,q}M$ with respect to $\langle\,\cdot\,,\cdot\,\rangle$. 
When $M$ is compact, we define
\begin{equation}\label{e-gue140519I}
{\rm Tr\,}T^{(q),f}_{k,\lambda}:=\int_M{\rm Tr\,}T^{(q),f}_{k,\lambda}(x,x)dv_M(x).
\end{equation}
For $\lambda=0$, we set $T^{(q),f}_{k}:=T^{(q),f}_{k,0}$, $T^{(q),f}_{k}(x,y):=T^{(q),f}_{k,0}(x,y)$, 
${\rm Tr\,}T^{(q),f}_{k}(x,x):={\rm Tr\,}T^{(q),f}_{k,0}(x,x)$, 
${\rm Tr\,}T^{(q),f}_{k}:={\rm Tr\,}T^{(q),f}_{k,0}$.

From \eqref{s0-e4mainmore}, \eqref{s0-e4mainmoreI} and \eqref{s0-e4main}, 
we get Weyl's formula for Berezin-Toeplitz quantization.

\begin{thm}\label{t-gue140523II}
%Let $(M,\Theta)$, $(L, h)$ and $f\in\cC^\infty(M)$ be as above and let $N>2n$. 
Assume (A) and (B) are fulfilled and
let $N>2n$. 
If $M$ is compact, then 
\begin{equation} \label{s0-e4mainI}
{\rm Tr\,}T^{(q),f}_{k,k^{-N}}=
k^n(2\pi)^{-n}\int_{M(q)}f(x)\big|\!\det\dot{R}^L(x)\big|\,dv_M(x)+o(k^n)\,,
\:\:k\to\infty.
\end{equation}
\end{thm}
%------------------------------
From Theorem~\ref{t-gue140523II} we deduce the following (see Section~\ref{s-gue140802}).
%------------------------------
\begin{thm}\label{t-gue140523III}
Under assumptions (A) and (B) suppose
that $M$ is compact and $M(q-1)=\emptyset$, $M(q+1)=\emptyset$.
%Let $(M,\Theta)$, $(L, h)$ and $f\in\cC^\infty(M)$ be as above and assume that $M$ is compact. 
%Fix $q\in\set{0,1,\ldots,n}$ and suppose that $M(q-1)=\emptyset$, $M(q+1)=\emptyset$. 
Then 
%------------------------------
\begin{equation}\label{e-gue140802}
\lim_{k\To\infty}\abs{k^{-n}T^{(q),f}_{k}(x,x)-(2\pi)^{-n}f(x)\big|\!\det\dot{R}^L(x)\big|\,
\mathds{1}_{M(q)}(x)I_{\det\ov{W}^{\,*}}(x)}=0\:\:\: \text{in $L^1_{(0,q)}(M)$}. 
\end{equation}
%------------------------------
In particular,  %we get Grauert-Riemenschneider criterion for Berezin-Toeplitz quantization: 
\begin{equation} \label{s0-e4mainII}
{\rm Tr\,}T^{(q),f}_{k}=k^n(2\pi)^{-n}\int_{M(q)}f(x)\big|\!\det\dot{R}^L(x)\big|\,dv_M(x)+o(k^n)
\quad\text{as $k\to\infty$\,.}
\end{equation}
\end{thm}
%------------------------------
Let's consider $q=0$ and $f\equiv1$ in \eqref{s0-e4mainII}. If $M(1)=\emptyset$ we obtain
$\dim H^0(M,L^k)=k^n(2\pi)^{-n}\int_{M(0)}\big|\!\det\dot{R}^L(x)\big|\,dv_M(x)+o(k^n)$
as $k\to\infty$. 
Therefore, $\dim H^0(M,L^k)\sim k^n$ as $k\to\infty$, provided $M(0)\neq\emptyset$ and $M(1)=\emptyset$.
This is a form of Demailly's criterion for a line bundle to be big, which answers
the Grauert-Riemenschneider conjecture, see \cite{De:85}, \cite[Theorem\,2.2.27]{MM07}. 
%------------------------------

We wish now to link the quantization scheme we proposed above by using spectral
spaces $\cE^q_{k^{-N}}(M, L^k)$ to the more traditional quantization using
holomorphic sections (or, more generally, harmonic forms).
For this purpose we need the notion of  
$O(k^{-N})$ small spectral gap property introduced in~\cite[Definition\,1.5]{HM12}:
%------------------------------
\begin{defn} \label{s1-d2bis}
Let $D\subset M$. We say that $\Box^{(q)}_k$ has \emph{$O(k^{-N})$ small spectral 
gap on $D$} if there exist constants $C_D>0$,  $N\in\mathbb N$, $k_0\in\mathbb N$, 
such that for all $k\geq k_0$ and $u\in\Omega^{0,q}_0(D,L^k)$, we have  
\[\norm{(I-P^{(q)}_{k,0})u}_{k}\leq C_D\,k^{N}\norm{\Box^{(q)}_ku}_{k}.\]
\end{defn}
%------------------------------
Let $D_0,D_1\subset M$ be open sets and 
$A_k, C_k:\Omega^{0,q}_0(D_1)\To\Omega^{0,q}(D_0)$ 
be $k$-dependent continuous operators with smooth kernels $A_k(x,y), C_k(x,y)\in
\cC^\infty(D_0\times D_1,\Lambda^{0,q}(T^*M))\boxtimes(\Lambda^{0,q}(T^*M))^*)$. We write 
$A_k\equiv C_k\mod O(k^{-\infty})$ locally uniformly on $D_0\times D_1$ or 
$A_k(x,y)\equiv C_k(x,y)\mod O(k^{-\infty})$ locally uniformly on $D_0\times D_1$
if 
\[\abs{\pr^\alpha_x\pr^\beta_y(A_k(x, y)-C_k(x,y))}=O(k^{-N})\] 
uniformly on every compact set in $D_0\times D_1$, for all $\alpha, \beta\in\mathbb N^{2n}_0$ and every $N>1$.  

The following result describes the asymptotics of the kernels of Toeplitz operators
corresponding to \emph{harmonic forms}
in the case of small spectral gap.
%-----
\begin{thm}\label{t-gue140523IV}
Under the assumptions (A) and (B) let $j\in\set{0,1,\ldots,n}$
and $D_{0},D_{1}\Subset M$ on which $L$ is trivial.
Suppose that
one of the following conditions is fulfilled:
\\[2pt]
(i)   $D_0\Subset M(j)$ and $j\neq q$,
\\[2pt]
(ii) $D_0\Subset M(q)$, $\Box^{(q)}_k$ has an $O(k^{-N})$ small spectral gap on $D_0$
 and $\ol D_0\bigcap\ol D_1=\emptyset$. 
\\[2pt]
%Let $(M,\Theta)$, $(L, h)$ and $f, g\in\cC^\infty(M)$ be bounded functions and let 
%$D_0\Subset M$, $D_1\Subset M$ be open sets on which $L$ is trivial.
%Fix $q\in\set{0,1,\ldots,n}$. Assume that $D_0\Subset M(j)$, 
%for some $j\in\set{0,1,\ldots,n}$ with $j\neq q$ or $D_0\Subset M(q)$ 
%and $\Box^{(q)}_k$ has $O(k^{-N})$ small spectral gap on $D_0$ and $\ol D_0\bigcap\ol D_1=\emptyset$. 
Then
\[\begin{split}
&\mbox{$T^{(q),f}_{k}(x,y)\equiv0\mod O(k^{-\infty})$ locally uniformly 
on $D_0\times D_1$},\\
&\mbox{$(T^{(q),f}_{k}\circ T^{(q),g}_{k})(x,y)\equiv0\mod O(k^{-\infty})$ locally uniformly 
on $D_0\times D_1$}.\end{split}\]
%---------------------
Assume that $D_0\Subset M(q)$ and $\Box^{(q)}_k$ has an $O(k^{-N})$ small 
spectral gap on $D_0$. Then, 
\[\begin{split}
&\mbox{$T^{(q),f}_{k}(x,y)\equiv e^{ik\Psi(x,y)}b_f(x,y,k)\mod O(k^{-\infty})$ 
locally uniformly on $D_0\times D_0$},\\
&\mbox{$(T^{(q),f}_{k}\circ T^{(q),g}_{k})(x,y)\equiv e^{ik\Psi(x,y)}b_{f,g}(x,y,k)\mod O(k^{-\infty})$
locally uniformly on $D_0\times D_0$},\end{split}\]
where $b_f(x,y,k), b_{f,g}(x,y,k)\in S^{n}(1;D_0\times D_0,\Lambda^{0,q}(T^*M))\boxtimes(\Lambda^{0,q}(T^*M))^*\big)$ 
are as in \eqref{e1.2} and \eqref{e-gue141118I}, respectively, 
and $\Psi(x,y)\in\cC^\infty(D_0\times D_0)$ is as in Theorem~\ref{t-gue140521}. 
\end{thm}
%-----------------
There are several geometric situations when there exists a spectral gap.
For example, if $L$ is a positive line bundle on a compact manifold $M$, or more generally, if $L$ is uniformly 
positive on a complete manifold $(M,\Theta)$ with $\sqrt{-1}R^{K^*_M}$ and $\partial\Theta$ 
bounded below, then the Kodaira Laplacian $\Box^{(0)}_k$ has a ``large" spectral gap on $M$, that is, 
there exists a constant $C>0$ such that for all $k$
we have $\inf\{\lambda\in\operatorname{Spec}(\Box^{(0)}_k);\, \lambda\neq0\}\geq Ck\,,$
(see \cite[Theorem\,1.5.5]{MM07}, \cite[Theorem\,6.1.1, (6.1.8)]{MM07}).
Therefore, we can recover from Theorem \ref{t-gue140523IV} results about quantization of
non-compact manifolds, such as \cite[Theorem\,7.5.1]{MM07}, \cite[Theorem\,5.3]{MM08b},
\cite[Theorem\,2.30]{MM11}.

In this paper, as applications of Theorem~\ref{t-gue140523IV}, we establish 
Berezin-Toeplitz quantization for semi-positive and big line bundles. 
We assume now that $(M,\Theta)$ is compact and we set
\[
\operatorname{Herm}(L)=\big\{\text{singular Hermitian metrics on $L$} \big\}\,,
\]
\[\begin{split}
\mathcal{M}(L)=\big\{h\in\operatorname{Herm}(L);&\,\text{$h$ is smooth outside a proper analytic set}\\
&\text{and the curvature current of $h$ is strictly positive}\big\}\,.
\end{split}\] 
Note that by Siu's criterion \cite[Theorem\,2.2.27]{MM07}, $L$ is big under the hypotheses of 
Theorem \ref{s1-sing-semi-main} below. By \cite[Lemma\,2.3.6]{MM07}, $\mathcal{M}(L)\neq\emptyset$.
Set
\begin{equation}\label{s1-sing-semi-set2}
M':=\set{p\in M\,;\, \mbox{$\exists $ $h\in\mathcal{M}(L)$ with $h$ smooth near $p$}}.
\end{equation}
%-------------------------------------------
\begin{thm} \label{s1-sing-semi-main}
Let $(M,\Theta)$ be a compact Hermitian manifold.
Let $(L,h)\to M$ be a Hermitian holomorphic line bundle with smooth Hermitian metric $h$ having 
semi-positive curvature and with $M(0)\neq\emptyset$. Let $f, g\in\cC^\infty(M)$ and let 
$D_0\Subset M(0)\bigcap M'$ be an open set on which $L$ is trivial. Then 
\[\begin{split}
&\mbox{$T^{(0),f}_{k}(x,y)\equiv e^{ik\Psi(x,y)}b_f(x,y,k)\mod O(k^{-\infty})$ 
locally uniformly on $D_0\times D_0$},\\
&\mbox{$(T^{(0),f}_{k}\circ T^{(0),g}_{k})(x,y)\equiv e^{ik\Psi(x,y)}b_{f,g}(x,y,k)\mod O(k^{-\infty})$ 
locally uniformly on $D_0\times D_0$},
\end{split}\]
where $b_f(x,y,k), b_{f,g}(x,y,k)\in S^{n}(1;D_0\times D_0)$ are as in \eqref{e1.2} and 
\eqref{e-gue141118I} respectively and $\Psi(x,y)\in\cC^\infty(D_0\times D_0)$ is as in Theorem~\ref{t-gue140521}. 
\end{thm}

Let us consider a singular Hermitian holomorphic line bundle $(L,h)\to M$ 
(see e.\,g.\ \cite[Definition\,2.3.1]{MM07}). We assume that $h$ is smooth outside a 
proper analytic set $\Sigma$ and the curvature current of $h$ is strictly positive.
The metric $h$ induces singular Hermitian
metrics $h^k$ on $L^k$. We denote by $\cali{I}(h^k)$ the Nadel multiplier ideal sheaf 
associated to $h^k$ and by $H^0(M,L^k\otimes\cali{I}(h^k))\subset H^0(M,L^k)$ the 
space of global sections of the sheaf $\cali{O}(L^k)\otimes\cali{I}(h^k)$ (see \eqref{l2:mult}), 
where $H^0(M,L^k):=\set{u\in\cC^\infty(M,L^k);\, \ddbar_ku=0}$.
We denote by $(\,\cdot\,,\cdot\,)_{k}$ the natural inner products on 
$\cC^\infty(M,L^k\otimes\cali{I}(h^k))$ induced by $h$ and the volume form $dv_M$ 
on $M$ (see \eqref{l2sing} and see also \eqref{pcsing} for the precise meaning of 
$\cC^\infty(M,L^k\otimes\cali{I}(h^k))$\,).
The (multiplier ideal) Bergman kernel of $H^0(M,L^k)\cali{I}(h^k))$ is the orthogonal projection
\begin{equation} \label{sing-e0-2}
P^{(0)}_{k,\cali{I}}:L^2(M,L^k)\To H^0(M,L^k\otimes\cali{I}(h^k)).
\end{equation}
Let $f\in\cC^\infty(M)$. The multiplier ideal Berezin-Toeplitz operator is
\begin{equation} \label{e-gue140523f}
T^{(0),f}_{k,\cali{I}}:=P^{(0)}_{k,\cali{I}}\circ f\circ 
P^{(0)}_{k,\cali{I}}:L^2(M,L^k)\To H^0(M,L^k\otimes\cali{I}(h^k))
\end{equation}
where we denote by $f$ the multiplication operator on $L^2(M,L^k)$ by $f$.
Let $T^{(0),f}_{k,\cali{I}}(x,y)$ be the distribution kernel of 
$T^{(0),f}_{k,\cali{I}}$. Note that 
$T^{(0),f}_{k,\cali{I}}(x,y)\in\cC^\infty((M\setminus\Sigma)\times(M\setminus\Sigma),(L^k)^*\boxtimes L^k)$. 

\begin{thm} \label{sing-main}
Let $(L,h)$ be a singular Hermitian holomorphic line bundle over a 
compact Hermitian manifold $(M,\Theta)$. We assume that $h$ is 
smooth outside a proper analytic set $\Sigma$ and the curvature current 
of $h$ is strictly positive. Let $f, g\in \cC^\infty(M)$. Let $D_0\subset M\setminus\Sigma$ 
be an open set on which $L$ is trivial. Then
\[\begin{split}
&\mbox{$T^{(0),f}_{k,\cali{I}}(x,y)\equiv e^{ik\Psi(x,y)}b_f(x,y,k)
\mod O(k^{-\infty})$ locally uniformly on $D_0\times D_0$},\\
&\mbox{$(T^{(0),f}_{k,\cali{I}}\circ T^{(0),g}_{k,\cali{I}})(x,y)\equiv e^{ik\Psi(x,y)}b_{f,g}(x,y,k)
\mod O(k^{-\infty})$ locally uniformly on $D_0\times D_0$},
\end{split}\]
where $b_f(x,y,k), b_{f,g}(x,y,k)\in S^{n}(1;D_0\times D_0)$ are as in 
\eqref{e1.2} and \eqref{e-gue141118I} respectively and $\Psi(x,y)\in\cC^\infty(D_0\times D_0)$ 
is as in Theorem~\ref{t-gue140521}.  
\end{thm}
%-------------------
The paper is organized as follows. In Section \ref{s:prelim} we collect terminology,
definitions and statements we use throughout.
In Sections \ref{s-gue140503} and \ref{s:toediag} prove the off-diagonal decay for the kernels
$P^{(q)}_{k,k^{-N}}(\cdot,\cdot)$ and $T^{(q),f}_{k,k^{-N}}(\cdot,\cdot)$.
In Section \ref{s-gue140508} we establish the full asymptotic of the Berezin-Toeplitz kernels
$T^{(q),f}_{k,k^{-N}}(\cdot,\cdot)$
and prove Theorem \ref{t-gue140521}. Section \ref{s:ctoe} is devoted to
the expansion of the composition of two Toeplitz operators and contains the proof of Theorems
\ref{t-gue140729II}, \ref{t-comp},
\ref{t-gue140523IV}, \ref{s1-sing-semi-main} and \ref{sing-main}.
In Section \ref{s-gue140729} we calculate the leading coefficients of the various expansions
we established.
Finally, in Section \ref{s-gue140802} we prove 
Theorem~\ref{s1-maindege} and Theorem~\ref{t-gue140523III}.

\section{Preliminaries}\label{s:prelim}

\noindent
\textbf{Some standard notations.}
We denote by $\mathbb N=\set{0,1,2,\ldots}$ the set of 
natural numbers and by $\Real$ the set of real numbers. 
%We set $\mathbb N=\mathbb N\cup\set{0}$.
We use the standard notations $w^\alpha$, $\partial_x^\alpha$ 
for multi-indices $\alpha=(\alpha_1,\ldots,\alpha_m)\in\mathbb N^m$, $w\in\mathbb C^m$, 
$\partial_x=(\partial_{x_1},\ldots,\partial_{x_m})$.

Let $\Omega$ be a $\cC^\infty$ paracompact manifold equipped with a smooth density of integration.
We let $T\Omega$ and $T^*\Omega$ denote the tangent bundle of $\Omega$ 
and the cotangent bundle of $\Omega$ respectively.
The complexified tangent bundle of $\Omega$ and the complexified cotangent bundle of 
$\Omega$ will be denoted by $\Complex T\Omega:=T\Omega\otimes_\Real\Complex$
and $\Complex T^*\Omega:=T^*\Omega\otimes_\Real\Complex$ respectively. We write $\langle\,\cdot\,,\cdot\,\rangle$ 
to denote the pointwise duality between $T\Omega$ and $T^*\Omega$.
We extend $\langle\,\cdot\,,\cdot\,\rangle$ bilinearly to 
$(T\Omega\otimes_\Real\Complex)\times(T^*\Omega\otimes_\Real\Complex)$.

Let $E$ be a $\cC^\infty$ vector bundle over $\Omega$. We write $E^*$ to denote the dual bundle of $E$. 
The fiber of $E$ at $x\in\Omega$ will be denoted by $E_x$. We denote by ${\rm End\,}(E)$ 
the vector bundle over $\Omega$ with fiber $\End(E)_x=\End(E_x)$ over $x\in\Omega$. 

Let $F$ be a vector bundle over another $\cC^\infty$ paracompact manifold $\Omega'$. 
We introduce the vector bundle $F\boxtimes E^*=\pi^*_1(F)\otimes\pi^*_2(E^*)$ over
$\Omega'\times\Omega$, 
where $\pi_1$ and $\pi_2$ are the projections 
of $\Omega'\times\Omega$ on the first and second factor (see~\cite[p.\,337]{MM07}).
The fiber of $F\boxtimes E^*$ over $(x, y)\in \Omega'\times\Omega$
consists of the linear maps from $E_y$ to $F_x$.

%We write
%$F\boxtimes E^*$ to denote the vector bundle over $\Omega'\times\Omega$ 
%with fiber over $(x, y)\in \Omega'\times\Omega$
%consisting of the linear maps from $E_y$ to $F_x$.  Note that $F\boxtimes E^*$ is isomorphic
%to $\pi^*_1(F)\otimes\pi^*_2(E)^*$, where $\pi_1$ and $\pi_2$ are the projections 
%of $\Omega'\times\Omega$ on the first and second factor (see~\cite{MM07}). 

Let $Y\subset\Omega$ be an open set and take any $L^2$ inner product on 
$\cC^\infty_0(Y,E)$. By using this $L^2$ inner product, in this paper, 
we will consider  a distribution section of $E$ over $Y$ is a continuous 
linear form on $\cC^\infty_0(Y,E)$. From now on, the spaces distribution sections of $E$ over $Y$ 
will be denoted by $\mathscr D'(Y, E)$.
Let $\mathscr E'(Y, E)$ be the subspace of $\mathscr D'(Y, E)$ whose elements have compact support in $Y$.
For $m\in\Real$, we let $H^m(Y, E)$ denote the Sobolev space
of order $m$ of sections of $E$ over $Y$. Put
\begin{gather*}
H^m_{\rm loc\,}(Y, E)=\big\{u\in\mathscr D'(Y, E);\, \varphi u\in H^m(Y, E),
      \,\varphi\in\cC^\infty_0(Y)\big\}\,,\\
      H^m_{\rm comp\,}(Y, E)=H^m_{\rm loc}(Y, E)\cap\mathscr E'(Y, E)\,.
\end{gather*}

Let $M$ be a complex manifold of dimension $n$. We always assume that $M$ is paracompact. 
Let $T^{1,0}M$ and $T^{0,1}M$ denote the holomorphic tangent bundle of $M$ and the anti-holomorphic 
tangent bundle of $M$ respectively. Let $\Lambda^{1,0}(T^*M)$ be the holomorphic cotangent bundle of $M$ and 
let $\Lambda^{0,1}(T^*M)$ be the anti-holomorphic cotangent bundle of $M$.  
For $p, q\in\mathbb N$, let
$\Lambda^{p,q}(T^*M)=
\Lambda^p(\Lambda^{1,0}(T^*M))\otimes\Lambda^q(\Lambda^{0,1}(T^*M))$ 
be the bundle of $(p,q)$ forms of $M$. 

For an open set
$D\subset M$ we let $\Omega^{p,q}(D)$ denote the space
of smooth sections of $\Lambda^{p,q}(T^*M)$ over $D$ and let $\Omega^{0,q}_0(D)$ 
be the subspace of $\Omega^{0,q}(D)$ whose elements have compact support in $D$. Similarly, if
$E$ is a vector bundle over $D$, then we let $\Omega^{p,q}(D, E)$
denote the space of smooth sections of $\Lambda^{p,q}(T^*M)\otimes
E$ over $D$. Let $\Omega^{p,q}_0(D, E)$ be the subspace of
$\Omega^{p,q}(D, E)$ whose elements have compact support in $D$.

For a multi-index
$J=(j_1,\ldots,j_q)\in\{1,\ldots,n\}^q$ we set $\abs{J}=q$. We say
that $J$ is strictly increasing if $1\leqslant
j_1<j_2<\cdots<j_q\leqslant  n$. Let $\set{e_1,\ldots,e_n}$ be a local frame for 
$\Lambda^{0,1}(T^*M)$ on an open set $D\subset M$. For a multi-index
$J=(j_1,\ldots,j_q)\in\{1,\ldots,n\}^q$, we put $e^J=e_{j_1}\wedge\cdots\wedge e_{j_q}$. 
Let $E$ be a vector bundle over $D$ and let $f\in\Omega^{0,q}(D,E)$. $f$ has the local representation
\[f|_D=\sideset{}{'}\sum_{\abs{J}=q}f_J(z)e^J\,,\]
where $\sum^{'}$ means that the summation is performed only
over strictly increasing multi-indices and $f_J\in\cC^\infty(D,E)$.

\medskip
\noindent
\textbf{Metric data.}
Let $(M,\Theta)$ be a complex manifold of dimension $n$, where $\Theta$ is a smooth positive $(1,1)$ form, which
induces a Hermitian metric $\langle\,\cdot\,,\cdot\,\rangle$ on the holomorphic tangent bundle $ T^{1,0}M$. 
In local holomorphic coordinates $z=(z_1,\ldots,z_n)$, if
$\Theta=\sqrt{-1}\sum^n_{j,k=1}\Theta_{j,k}dz_j\wedge d\ol z_k$, 
then $\langle\,\frac{\pr}{\pr z_j}\,|\,\frac{\pr}{\pr z_k}\,\rangle=\Theta_{j,k}, j, k=1,\ldots,n$. 
We extend the Hermitian metric $\langle\,\cdot\,,\cdot\,\rangle$ to $TM\otimes_\Real\Complex$ in a natural way. 
The Hermitian metric $\langle\,\cdot\,,\cdot\,\rangle$ on $TM\otimes_\Real\Complex$
induces a Hermitian metric on
$\Lambda^{p,q}(T^*M)$ also denoted by $\langle\,\cdot\,,\cdot\,\rangle$.  

Let $(L,h)$ be a Hermitian holomorphic line bundle over $M$, where
the Hermitian metric on $L$ is denoted by $h$. Until further notice, we assume that $h$ is smooth. 
Given a local holomorphic frame $s$ of $L$ on an open subset $D\subset M$ we define the associated local weight of $h$ by
\begin{equation} \label{s1-e1}
\abs{s(x)}^2_{h}=e^{-2\phi(x)},\quad\phi\in\cC^\infty(D,\Real).
\end{equation}
Let $R^L=(\nabla^L)^2$ be the Chern curvature of $L$, where $\nabla^L$ 
is the Hermitian holomorphic connection. Then $R^L|_D=2\pr\ddbar\phi$. 

Let $L^k$, $k>0$, be the $k$-th tensor power of the line bundle $L$.
The Hermitian fiber metric on $L$ induces a Hermitian fiber metric
on $L^k$ that we shall denote by $h^k$. If $s$ is a local
trivializing holomorphic section of $L$ then $s^k$ is a local trivializing holomorphic 
section of $L^k$. For $p,q\in\mathbb N$, the Hermitian metric $\langle\,\cdot\,,\cdot\,\rangle$ on 
$\Lambda^{p,q}(T^*M)$ and $h^k$ induce a Hermitian metric on $\Lambda^{p,q}(T^*M)\otimes L^k$, 
denoted by $\langle\,\cdot\,,\cdot\,\rangle_{h^k}$. 
For $s\in\Omega^{p,q}(M, L^k)$, we denote the
pointwise norm $\abs{s(x)}^2_{h^k}:=\langle s(x),s(x)\rangle_{h^k}$. We take $dv_M=dv_M(x)$
as the induced volume form on $M$. The $L^2$--Hermitian inner products on the spaces
$\Omega^{p,q}_0(M,L^k)$ and $\Omega^{p,q}_0(M)$ are given by
\begin{equation}\label{toe2.2}
\begin{split}
&(s_1,s_2)_{k}=\int_M\langle s_1(x),s_2(x)\rangle_{h^k}\,dv_M(x)\,,\ \ s_1,s_2
\in\Omega^{p,q}_0(M,L^k),\\
&(f_1,f_2)=\int_M\langle f_1(x),f_2(x)\rangle\,dv_M(x)\,,\ \ f_1,f_2\in\Omega^{p,q}_0(M).\\
&\norm{s}^2_{k}=(s,s)_{k},\:\: s\in\Omega^{p,q}_0(M,L^k),\quad 
\norm{f}^2:=(f,f),\:\:f\in\Omega^{p,q}_0(M).
\end{split}
\end{equation}
Let $A_k:L^2_{(0,q)}(M,L^k)\To L^2_{(0,q)}(M,L^k)$ be a $k$-dependent continuous operator with smooth kernel $A_k(x,y)$. 
Let $s$, $\widehat s$ be local trivializing holomorphic sections of $L$ on $D_0\Subset M$, $D_1\Subset M$ 
respectively, $\abs{s}^2_{h}=e^{-2\phi}$, 
$\abs{\widehat s}^2_{h}=e^{-2\widehat\phi}$, where $D_0$, $D_1$ are open sets. 
The localized operator of $A_k$ on $D_0\times D_1$ is given by 
\begin{equation}\label{e-gue141118}
\begin{split}
A_{k,s,\widehat s}:\Omega^{0,q}_0(D_1)\To\Omega^{0,q}(D_0),
\:\:u\longmapsto s^{-k}e^{-k\phi}(A_{k}\widehat s^ke^{k\widehat\phi}u),
\end{split}
\end{equation}
and let
$A_{k,s,\widehat s}(x,y)\in\cC^\infty(D_0\times D_1,\Lambda^{0,q}(T^*M))\boxtimes(\Lambda^{0,q}(T^*M))^*)$ 
be the distribution kernel of $A_{k,s,\widehat s}$. For $s=\widehat s$, $D_0=D_1$, 
we set 
\begin{equation}\label{e-gue141118a}
A_{k,s}:=A_{k,s,s}\,,\quad A_{k,s}(x,y):=A_{k,s,s}(x,y). 
\end{equation}

%\subsection{A self-adjoint extension of the Kodaira Laplacian}\label{gaffney}
\medskip
\noindent
\textbf{A self-adjoint extension of the Kodaira Laplacian.}
We denote by
\begin{equation}\label{e:ddbar}
\ddbar_k:\Omega^{0,r}(M,L^k)\To\Omega^{0,r+1}(M,L^k)\,,\:\:\ol{\pr}^{\,*}_k:\Omega^{0,r+1}(M,L^k)\To\Omega^{0,r}(M,L^k)
\end{equation} 
the Cauchy-Riemann operator acting on sections of $L^k$ and its formal adjoint 
with respect to $(\,\cdot\,|\,\cdot)_{k}$ respectively. Let 
\begin{equation}\label{e:Kod_Lap}
\Box_{k}^{(q)}:=\ddbar_k\ol{\pr}^{\,*}_k+\ol{\pr}^{\,*}_k\ddbar_k:
\Omega^{0,q}(M, L^k)\To\Omega^{0,q}(M, L^k)
\end{equation}   
be the Kodaira Laplacian acting on $(0,q)$--forms with values in $L^k$. We extend 
$\ddbar_k$ to $L^2_{(0,r)}(M,L^k)$ by 
\begin{equation}\label{e:ddbar1}
\ddbar_k:{\rm Dom\,}\ddbar_k\subset L^2_{(0,r)}(M, L^k)\To L^2_{(0,r+1)}(M, L^k)\,,
\end{equation}
where ${\rm Dom\,}\ddbar_k:=\{u\in L^2_{(0,r)}(M, L^k);\, \ddbar_ku\in L^2_{(0,r+1)}(M, L^k)\}$, where %for any $u\in L^2_{(0,r)}(M,L^k)$, 
$\ddbar_k u$ is defined in the sense of distributions. 
We also write 
\begin{equation}\label{e:ddbar_star1}
\ol{\pr}^{\,*}_k:{\rm Dom\,}\ol{\pr}^{\,*}_k\subset L^2_{(0,r+1)}(M, L^k)\To L^2_{(0,r)}(M, L^k)
\end{equation}
to denote the Hilbert space adjoint of $\ddbar_k$ in the $L^2$ space with respect to $(\,\cdot\,,\cdot\,)_{k}$.
Let $\Box^{(q)}_k$ denote the Gaffney extension of the Kodaira Laplacian given by 
\begin{equation}\label{Gaf1}
\begin{split}
{\rm Dom\,}\Box^{(q)}_k=\Big\{s\in L^2_{(0,q)}(M,L^k);\, s\in
{\rm Dom\,}\ddbar_k\cap{\rm Dom\,}\ol{\pr}^{\,*}_k,
\:\:\ddbar_ks\in{\rm Dom\,}\ol{\pr}^{\,*}_k,\ \ol{\pr}^{\,*}_ks\in{\rm Dom\,}\ddbar_k\Big\}\,,
\end{split}
\end{equation}
and $\Box^{(q)}_ks=\ddbar_k\ol{\pr}^{\,*}_ks+\ol{\pr}^{\,*}_k\ddbar_ks$ for 
$s\in {\rm Dom\,}\Box^{(q)}_k$. By a result of Gaffney \cite[Proposition\,3.1.2]{MM07}, 
$\Box^{(q)}_k$ is a positive self-adjoint operator. Note that if $M$ is complete, the 
Kodaira Laplacian $\Box^{(q)}_k$ is essentially self-adjoint \cite[Corollary\,3.3.4]{MM07} 
and the Gaffney extension coincides with the Friedrichs extension of $\Box^{(q)}_k$.

Consider a singular Hermitian metric $h$ on a holomorphic line bundle $L$ over $M$.
If $h_0$ is a smooth Hermitian metric on $L$ then $h=h_0e^{-2\varphi}$ for some function 
$\varphi\in L^1_{loc}(M,\mathbb R)$. The \emph{Nadel multiplier ideal sheaf} of $h$ is defined by 
$\cali{I}(h)=\cali{I}(\varphi)$; the definition does not depend on the choice of $h_0$. Recall that the \emph{Nadel multiplier ideal sheaf} $\cali{I}(\varphi)\subset\cO_{M}$ is the ideal subsheaf of germs of holomorphic functions $f\in\cO_{M,x}$ such that $|f|^2e^{-2\varphi}$ is integrable with respect to the Lebesgue measure in local coordinates near $x$ for all $x\in M$.
Put
\begin{equation}\label{pcsing}
\begin{split}
\cC^\infty(M,L\otimes\cali{I}(h)):=\set{S\in\cC^\infty(M,L);\, \int_M\big\lvert S\big\rvert^2_{h}\,dv_M=
\int_M\big\lvert S\big\rvert^2_{h_0}\,e^{-2\varphi}\,dv_M<\infty},
\end{split}
\end{equation}
where $\abs{\,\cdot\,}_{h}$ and $\abs{\,\cdot\,}_{h_0}$ denote the pointwise norms for sections induced by $h$ 
and $h_0$ respectively. With the help of $h$ and the volume form $dv_M$ we define an $L^2$ inner 
product on $\cC^\infty(M,L\otimes\cali{I}(h))$:
\begin{equation}\label{l2sing}
(S,S')_{h}=\int_M\langle S,S'\rangle_{h_0}\,e^{-2\varphi}dv_M\,,\quad S,S'\in \cC^\infty(M,L\otimes\cali{I}(h))\,.
\end{equation}

The singular Hermitian metric $h$ induces a singular Hermitian
metric $h^k=h^k_0e^{-2k\varphi}$
on $L^k$, $k>0$. We denote by $(\,\cdot\,,\cdot\,)_{k}$ the natural inner products on 
$\cC^\infty(M,L^k\otimes\cali{I}(h^k))$ defined as in \eqref{l2sing} and by $L^2(M,L^k)$ 
the completion of $\cC^\infty(M,L^k\otimes\cali{I}(h^k))$ with respect to $(\,\cdot\,,\cdot\,)_{k}$. 
The space of global sections in the sheaf $\cali{O}(L^k)\otimes\cali{I}(h^k)$ is given by
\begin{equation}\label{l2:mult}
\begin{split}
&H^0(M,L^k\otimes\cali{I}(h^k))\\
&\quad=\left\{
s\in\cC^\infty(M,L^k);\, \ddbar_ks=0,\,\int_M\big\lvert s\big\rvert^2_{h^k}\,dv_M=
\int_M\big\lvert s\big\rvert^2_{h^k_0}\,e^{-2k\varphi}\,dv_M<\infty
\right\}\,.
\end{split}
\end{equation}

%\subsection{Schwartz kernel theorem and semi-classical H\"ormander symbol spaces}

\medskip
\noindent
\textbf{Schwartz kernel theorem and semi-classical H\"ormander symbol spaces.}
We recall here the Schwartz kernel theorem \cite[Theorems\,5.2.1, 5.2.6]{Hor03}, \cite[Thorem\,B.2.7]{MM07}. 
Let $\Omega$ be a $\cC^\infty$ paracompact manifold equipped with a smooth density of integration.
Let $E$ and $F$ be smooth vector bundles over $\Omega$. 
Any distribution (`kernel') 
\begin{equation}\label{e:kern0}
A(x,y)\in\mathscr D'(\Omega\times\Omega,F\boxtimes E^*),
\end{equation}
defines a continuous operator 
\begin{equation}\label{e:kern1}
A:\cC^\infty_0(\Omega,E)\To\mathscr D'(\Omega,F)\,,
\quad \langle Au,v\rangle:=\langle A(x,y),v(x)\otimes u(y)\rangle,
\end{equation}
%by 
%$\langle Au,v\rangle=\langle A(x,y),v(x)\otimes u(y)\rangle$
for any $u\in \cC^\infty_0(\Omega,E)$, $v\in \cC^\infty_0(\Omega,F)$. 
Conversely, any continuous linear operator
$A:\cC^\infty_0(\Omega, E)\To \mathscr D'(\Omega, F)$ is given by a distribution
$A(x, y)\in \mathscr D'(\Omega\times\Omega,F\boxtimes E^*)$ as above, 
called the Schwartz distribution kernel of $A$.
Moreover, the following two statements are equivalent
\begin{equation}\label{e-gue160711s}
\begin{split}
&\mbox{(a) $A$ is continuous: $\mathscr E'(\Omega, E)\To\cC^\infty(\Omega, F)$},\\
&\mbox{(b)  $A(x,y)\in\cC^\infty(\Omega\times\Omega, F\boxtimes E^*)$}.
\end{split}
\end{equation}
If $A$ satisfies (a) or (b), we say that $A$ is a \emph{smoothing operator}. Furthermore, $A$ is smoothing if and only if
$A: H^s_{\rm comp\,}(\Omega, E)\To H^{s+N}_{\rm loc\,}(\Omega, F)$
is continuous, for all $N\geq0$, $s\in\Real$.
Let
$A, B:\cC^\infty_0(\Omega, E)\to \mathscr D'(\Omega, F)$ be continuous operators.
We write $A\equiv B$ or $A(x,y)\equiv B(x,y)$ if $A-B$ is a smoothing operator.

We say that $A$ is properly supported if 
%${\rm Supp\,}A(\cdot,\cdot)\subset\Omega\times\Omega$ is proper, i.\,e., if 
the restrictions to ${\rm Supp\,}A(\cdot,\cdot)$ of the projections $\pi_1$ and $\pi_2$ from $\Omega\times\Omega$
to the first and second factor
 are proper. 
%$t_x:(x,y)\in{\rm Supp\,}K_A\To x\in\Omega$, $t_y:(x,y)\in{\rm Supp\,}K_A\To y\in\Omega$ are proper.

We say that $A$ is smoothing away the diagonal
if $\chi_1A\chi_2$ is smoothing for all $\chi_1, \chi_2\in\cC^\infty_0(\Omega)$ 
with ${\rm Supp\,}\chi_1\cap{\rm Supp\,}\chi_2=\emptyset$.

We recall the definition of semi-classical H\"{o}rmander symbol spaces 
\cite[Chapter 8]{DS99}:
%-------------------
\begin{defn} \label{s1-d1}
Let $U$ be an open set in $\Real^N$. Let 
\begin{gather*}
S(1)=S(1;U):=\Big\{a\in \cC^\infty(U)\,|\, \forall\alpha\in\mathbb N^N_0: 
\sup_{x\in U}\abs{\pr^\alpha a(x)}<\infty\Big\},\\
S^0(1;U):=\Big\{(a(\cdot,k))_{k\in\N}\in \cC^\infty(U)^\N\,|\, \forall\alpha\in\mathbb N^N_0
\,:\:\sup_{k\in\N}\sup_{x\in U}\abs{\pr^\alpha a(x,k)}<\infty\Big\}\,.
\end{gather*}
%
%$S(1;U)=S(1)$ be the set of
%$a\in\cC^\infty(U)$ such that for every $\alpha\in\mathbb N^N$, there
%exists $C_\alpha>0$, such that $\abs{\pr^\alpha_xa(x)}\leq
%C_\alpha$ on $U$. If $a=a(x,k)$ depends on $k\in]1,\infty[$, we say that
%$a(x,k)\in S(1)$ if $a(x,k)$ uniformly bounded
%in $S(1)$ when $k$ varies in $]1,\infty[$. 
For $m\in\Real$ let
\[
S^m(1;U)=\Big\{(a(\cdot,k))_{k\in\N}\in \cC^\infty(U)^\N\,|
\,(k^{-m}a(\cdot,k))\in S^0(1;U)\Big\}\,.
\]
Hence $(a(\cdot,k))\in S^m(1;U)$ if for every $\alpha\in\mathbb N^N_0$, there
exists $C_\alpha>0$, such that $\abs{\pr^\alpha a(\cdot,k)}\leq C_\alpha k^{m}$ on $W$. 
Consider a sequence $a_j\in S^{m_j}(1)$, $j\in\N$, where $m_j\searrow-\infty$, 
and let $a\in S^{m_0}(1)$. We say that 
\[
a(\cdot,k)\sim
\sum\limits^\infty_{j=0}a_j(\cdot,k),\:\:\text{in $S^{m_0}(1)$},
\] 
if for every
$\ell\in\N$ we have $a-\sum^{\ell}_{j=0}a_j\in S^{m_{\ell+1}}(1)$ .  
For a given sequence $a_j$ as above, we can always find such an asymptotic sum 
$a$, which is unique up to an element in 
$S^{-\infty}(1)=S^{-\infty}(1;U):=\cap _mS^m(1)$.
We define $S^m(1;Y, E)$ in the
natural way, where $Y$ is a smooth paracompact manifold and $E$ is a
vector bundle over $Y$.
\end{defn}

\section{Spectral kernel estimates away the diagonal}\label{s-gue140503}
%------
The goal of this section is to prove the off-diagonal decay for the kernel
$P^{(q)}_{k,k^{-N}}(\cdot,\cdot)$ of the spectral projection $P^{(q)}_{k,k^{-N}}$.
For this purpose we introduce a localization of the projection.
Let $s$, $\widehat s$ be local trivializing holomorphic sections of $L$ on $D_0\Subset M$, $D_1\Subset M$ 
respectively, $\abs{s}^2_{h}=e^{-2\phi}$, 
$\abs{\widehat s}^2_{h}=e^{-2\widehat\phi}$, where $D_0$, $D_1$ are open sets. 
We denote by $P^{(q)}_{k,k^{-N},s,\widehat s}$ the localization given by \eqref{e-gue141118}.
%------
%For $\lambda\geq0$, let 
%\begin{equation}\label{e-gue140521b}
%\begin{split}
%P^{(q),f}_{k,\lambda,s,\widehat s}:\Omega^{0,q}_0(D_1)&\To\Omega^{0,q}(D_0),\\
%u&\To s^{-k}e^{-k\phi}(P^{(q),f}_{k,\lambda}\widehat s^ke^{k\widehat\phi}u),
%\end{split}
%\end{equation}
%and let
%$P^{(q),f}_{k,\lambda,s,\widehat s}(x,y)\in\cC^\infty(D_0\times D_1,T^{*0,q}_yM\otimes T^{*0,q}_xM)$ 
%be the distribution kernel of $P^{(q),f}_{k,\lambda,s,\widehat s}$. For $s=\widehat s$, $D_0=D_1$, 
%we set $P^{(q),f}_{k,\lambda,s}:=P^{(q),f}_{k,\lambda,s,s}$, 
%$P^{(q),f}_{k,\lambda,s}(x,y):=P^{(q),f}_{k,\lambda,s,s}(x,y)$. 
%%------
%For $\lambda=0$, we put 
%\[
%P^{(q),f}_{k,s,\widehat s}:=P^{(q),f}_{k,0,s,\widehat s},\:\: 
%P^{(q),f}_{k,s,\widehat s}(x,y):=P^{(q),f}_{k,0,s,\widehat s}(x,y),\:\:
%P^{(q),f}_{k,s}:=P^{(q),f}_{k,0,s},\:\: P^{(q),f}_{k,s}(x,y):=P^{(q),f}_{k,0,s}(x,y).
%\]

Let $\set{e_1,e_2,\ldots,e_n}$ and $\set{w_1,w_2,\ldots,w_n}$ 
be orthonormal frames of $\Lambda^{0,1}(T^*M)$ on $D_0$ and $D_1$ respectively. Then, 
\[\set{e^J;\, \abs{J}=q,\ \ \mbox{$J$ is strictly increasing}},\ \ \set{w^J;\, \abs{J}=q,
\ \ \mbox{$J$ is strictly increasing}}\] 
are orthonormal frames of $\Lambda^{0,q}(T^*M)$ on $D_0$ and $D_1$ respectively. We write 
\begin{equation}\label{e-gue140503}
\begin{split}
&P^{(q)}_{k,k^{-N},s,\widehat s}(x,y)=\sideset{}{'}\sum_{\abs{I}=
\abs{J}=q}P^{(q),I,J}_{k,k^{-N},s,\widehat s}(x,y)e^I(x)\wedge(w^J(y))^\dagger,\\
&\mbox{$P^{(q),I,J}_{k,k^{-N},s,\widehat s}(x,y)\in\cC^\infty(D_0\times D_1)$, 
$\forall \abs{I}=\abs{J}=q$, $I$, $J$ are strictly increasing},
\end{split}
\end{equation}
in the sense that for every $u=\sideset{}{'}\sum\limits_{\abs{J}=q}u_Jw^J\in\Omega^{0,q}_0(D_1)$, we have 
\begin{equation}\label{e-gue140503I}
(P^{(q)}_{k,k^{-N},s,\widehat s}u)(x)=
\sideset{}{'}\sum_{\abs{I}=\abs{J}=q}e^I(x)\otimes\int P^{(q),I,J}_{k,k^{-N},s,\widehat s}(x,y)u_J(y)dv_M(y).
\end{equation}
The goal of this section is to prove the following. 
%------------------------
\begin{thm}\label{t-gue140503}
With the notations used above, we assume that $D_0\Subset M(j)$, $j\neq q$, $j\in\set{0,1,\ldots,n}$ 
or $D_0\Subset M(q)$ and $\ol D_0\bigcap\ol D_1=\emptyset$. Then, for every $N>1$, $m\in\mathbb N$, 
there exists $C_{N,m}>0$ independent of $k$, such that for all strictly increasing
$I, J$ with $\abs{I}=\abs{J}=q$,
\[\abs{P^{(q),I,J}_{k,k^{-N},s,\widehat s}(x,y)}_{\cC^m(D_0\times D_1)}\leq 
C_{N,m}k^{2n-\frac{N}{2}+2m}.\]
\end{thm}
%------------------------
As preparation, we recall the next result, established in \cite[Theorems\,4.11-12]{HM12}. 
The localization $P^{(q)}_{k,k^{-N},s}$ is defined as in \eqref{e-gue141118a}.
%------------------------
\begin{thm}\label{t-gue140503I}
With the notations used above, assume that $D_0\Subset M(q)$. 
Then, for every $N>1$, $m\in\mathbb N$, 
there exists $C_{N,m}>0$ independent of $k$ such that 
\[\abs{P^{(q)}_{k,k^{-N},s}(x,y)-e^{ik\Psi(x,y)}b(x,y,k)}_{\cC^m(D_0\times D_0)}\leq
 C_{N,m}k^{3n-N+2m},\]
where 
\[
\begin{split}
&b(x,y,k)\in S^{n}(1;D_0\times D_0,\Lambda^{0,q}(T^*M)\boxtimes(\Lambda^{0,q}(T^*M))^*), \\
&b(x,y,k)\sim\sum^\infty_{j=0}b_{j}(x,y)k^{n-j}\text{ in }
S^{n}(1;D_0\times D_0,\Lambda^{0,q}(T^*M)\boxtimes(\Lambda^{0,q}(T^*M))^*), \\
&b_j(x,y)\in\cC^\infty(D_0\times D_0,(\Lambda^{0,q}(T^*M)\boxtimes(\Lambda^{0,q}(T^*M))^*),\ \ j=0,1,2,\ldots,\\
&b_0(x,x)=(2\pi)^{-n}\big|\det\dot{R}^L(x)\big|I_{\det\ov{W}^{\,*}}(x),\ \ \forall x\in D_0,
\end{split}
\]
and $b(x,y,k)$ is properly supported and $\Psi(x,y)\in\cC^\infty(D_0\times D_0)$ is as in Theorem~\ref{t-gue140521}.

Assume that $D_0\Subset M(j)$, $j\neq q$, $j\in\set{0,1,2,\ldots,n}$. 
Then, for every $N>1$, $m\in\mathbb N$, there exists $\Td C_{N,m}>0$ independent of $k$ such that 
\[\abs{P^{(q)}_{k,k^{-N},s}(x,y)}_{\cC^m(D_0\times D_0)}\leq\Td C_{N,m}k^{3n-N+2m}.\]
\end{thm}
%---------------------
The following properties of the phase function $\Psi$ follow also from \cite[Theorem\,3.8]{HM12}.
%----------------------
\begin{thm}\label{t-gue140523}
With the assumptions and notations used in Theorem~\ref{t-gue140521}, for a given
point $p\in D_0$, let $x=z=(z_1,\ldots,z_n)$ be local holomorphic coordinates centered at $p$ satisfying 
\begin{equation}\label{e-gue140521II}
\begin{split}
&\Theta(p)=\sqrt{-1}\sum^n_{j=1}dz_j\wedge d\ol z_j\,,\\
&\phi(z)=\sum^{n}_{j=1}\lambda_j\abs{z_j}^2+O(\abs{z}^3)\,,\:
\text{$z$ near $p$}\,,\:\{\lambda_j\}_{j=1}^n\subset\Real\setminus\{0\}\,,
\end{split}
\end{equation}
then we have near $(0,0)$,
\begin{equation} \label{s3-e16-bisbg}
\Psi(z,w)=i\sum^n_{j=1}\abs{\lambda_j}\abs{z_j-w_j}^2+i\sum^n_{j=1}\lambda_j(\ol
z_jw_j-z_j\ol w_j)+O(\abs{(z,w)}^3).
\end{equation}
Moreover, when $q=0$, we have 
\begin{equation} \label{s3-eIV} 
\Psi(z,w)=i\bigr(\phi(z)+\phi(w)\bigr)-2i\sum_{\alpha,\beta\in\mathbb N,
\abs{\alpha}+\abs{\beta}\leq N}\frac{\pr^{\abs{\alpha}+\abs{\beta}}\phi}
{\pr z^\alpha\pr\ol z^\beta}(0)\frac{z^\alpha}{\alpha!}\frac{\ol w^\beta}{\beta!}+O(\abs{(z,w)}^{N+1}),
\end{equation} 
for every $N\in\mathbb N$. 
\end{thm}

Fix $N>1$. Let $\set{g_1(x),g_2(x),\ldots,g_{d_k}(x)}$ be an orthonormal frame 
for $\cE_{k^{-N}}(M,L^k)$, where $d_k\in\mathbb N\cup\set{\infty}$. On $D_0$, $D_1$, we write 
\begin{equation}\label{e-gue140503II}
\begin{split}
&g_j(x)=s^k(x)\Td g_j(x),\ \ \mbox{$\Td g_j(x)=
\sideset{}{'}\sum\limits_{\abs{J}=q}\Td g_{j,J}(x)e^J(x)$ on $D_0$},\ \ j=1,\ldots,d_k,\\
&g_j(x)=\widehat s^k(x)\widehat g_j(x),\ \ \mbox{$\widehat g_j(x)=
\sideset{}{'}\sum\limits_{\abs{J}=q}\widehat g_{j,J}(x)w^J(x)$ on $D_1$},\ \ j=1,\ldots,d_k.
\end{split}
\end{equation}
It is not difficult to check that for every strictly increasing $I$, $J$, with $\abs{I}=\abs{J}=q$, we have
\begin{equation}\label{e-gue140504}
\begin{split}
&P^{(q),I,J}_{k,k^{-N},s,\widehat s}(x,y)=
\sum^{d_k}_{j=1}\Td g_{j,I}(x)e^{-k\phi(x)}\ol{\widehat g_{j,J}}(y)e^{-k\widehat\phi(y)},\\
&P^{(q),I,J}_{k,k^{-N},s}(x,y)=\sum^{d_k}_{j=1}\Td g_{j,I}(x)e^{-k\phi(x)}\ol{\Td g_{j,J}}(y)e^{-k\widehat\phi(y)}.
\end{split}
\end{equation}

\begin{lem}\label{l-gue140504}
Assume that $D_0\Subset M(j)$, $j\neq q$. Then, for every $N>1$, $m\in\mathbb N$, there 
exists $C_{N,m}>0$ independent of $k$ such that for every strictly increasing $I$, $J$, with $\abs{I}=\abs{J}=q$,
\[\abs{P^{(q),I,J}_{k,k^{-N},s,\widehat s}(x,y)}_{\cC^m(D_0\times D_1)}\leq C_{N,m}k^{2n-\frac{N}{2}+2m}.\]
\end{lem}

\begin{proof}
Fix $I$, $J$ are strictly increasing, $\abs{I}=\abs{J}=q$, and let 
$\alpha, \beta\in\mathbb N^{2n}_0$. By \eqref{e-gue140504}, we have 
\begin{equation}\label{e-gue140504I}
\abs{\pr^{\alpha}_x\pr^{\beta}_yP^{(q),I,J}_{k,k^{-N},s,\widehat s}(x,y)}
\leq\sqrt{\sum^{d_k}_{j=1}\abs{\pr^\alpha_x(\Td g_{j,I}(x)e^{-k\phi(x)})}^2}
\sqrt{\sum^{d_k}_{j=1}\abs{\pr^\beta_y(\widehat g_{j,J}(y)e^{-k\widehat\phi(y)})}^2}.
\end{equation}
In view of Theorem~\ref{t-gue140503I}, we see that 
\begin{equation}\label{e-gue140504II}
\sum^{d_k}_{j=1}\abs{\pr^\alpha_x(\Td g_{j,I}(x)e^{-k\phi(x)})}^2
\leq C_\alpha k^{3n-N+4\abs{\alpha}},\: \text{on $D_0$},
\end{equation}
where $C_\alpha>0$ is a constant independent of $k$. Moreover, it is known (see~\cite[Theorem\,4.3]{HM12}) that
\begin{equation}\label{e-gue140504III}
\sum^{d_k}_{j=1}\abs{\pr^\beta_y(\widehat g_{j,J}(y)e^{-k\widehat\phi(y)})}^2
\leq C_\beta k^{n+4\abs{\beta}} \: \text{on $D_1$},
\end{equation}
where $C_\beta>0$ is a constant independent of $k$. 
From \eqref{e-gue140504I}, \eqref{e-gue140504II} and \eqref{e-gue140504III}, the lemma follows. 
\end{proof}
%--------
Lemma~\ref{l-gue140504} provides the proof of Theorem~\ref{t-gue140503} in the case 
$D_0\Subset M(j)$, $j\neq q$. Now, we assume that $D_0\Subset M(q)$. 
Fix $p\in D_0$, $I_0$, $J_0$ strictly 
increasing with $\abs{I_0}=\abs{J_0}=q$, and $\alpha, \beta\in\mathbb N^{2n}_0$. Put 
\begin{equation}\label{e-gue140504IV}
\begin{split}
&\pr^\alpha_x\pr^\alpha_ye^{ik\Psi(x,y)}b(x,y,k)=e^{ik\Psi(x,y)}a(x,y,k)=
\sideset{}{'}\sum_{\abs{I}=\abs{J}=q}e^{ik\Psi(x,y)}a_{I,J}(x,y,k)e^I(x)\wedge(e^J(y))^\dagger,\\
&\pr^\alpha_ye^{ik\Psi(x,y)}b(x,y,k)=e^{ik\Psi(x,y)}d(x,y,k)=
\sideset{}{'}\sum_{\abs{I}=\abs{J}=q}e^{ik\Psi(x,y)}d_{I,J}(x,y,k)e^I(x)\wedge(e^J(y))^\dagger,
\end{split}
\end{equation}
where $\Psi(x,y)$ and $b(x,y,k)$ are as in Theorem~\ref{t-gue140503I} 
and $a_{I,J}(x,y,k), d_{I,J}(x,y,k)\in\cC^\infty(D_0\times D_0)$, for any 
$I$, $J$ strictly increasing with $\abs{I}=\abs{J}=q$. 

\begin{lem}\label{l-gue140504I}
Assume that $a_{I_0,I_0}(p,p,k)\leq0$. Then, there exists $C_{\alpha,\beta}>0$ 
independent of $k$ and the point $p$ such that 
\[\abs{\pr^\alpha_x\pr^\beta_yP^{(q),I_0,J_0}_{k,k^{-N},s,\widehat s}(p,y)}\leq 
C_{\alpha,\beta}k^{2n-\frac{N}{2}+2\abs{\alpha}+2\abs{\beta}},\ \ \forall y\in D_1.\]
\end{lem}

\begin{proof}
In view of Theorem~\ref{t-gue140503I} and \eqref{e-gue140504}, we see that 
\begin{equation}\label{e-gue140504V}
\begin{split}
\sum^{d_k}_{j=1}\abs{\pr^\alpha_x(\Td g_{j,I_0}e^{-k\phi})(p))}^2=
\pr^\alpha_x\pr^\alpha_yP^{(q),I_0,I_0}_{k,k^{-N},s}(p,p)&\leq
\abs{\pr^\alpha_x\pr^\alpha_yP^{(q),I_0,I_0}_{k,k^{-N},s}(p,p)-a_{I_0,I_0}(p,p)}\\
&\leq C_\alpha k^{3n-N+4\abs{\alpha}},
\end{split}
\end{equation}
where $C_\alpha>0$ is a constant independent of $k$ and the point $p$. 
From \eqref{e-gue140504V}, \eqref{e-gue140504I} and \eqref{e-gue140504III}, the lemma follows. 
\end{proof}

Now, we assume that $a_{I_0,I_0}(p,p)>0$. Take $\chi\in\cC^\infty_0(\Real,[0,1])$ 
so that $\chi=1$ if $\abs{x}\leq1$, $\chi=0$ if $\abs{x}>2$. Put 
\begin{equation}\label{e-gue140504VI}
\begin{split}
&\Td u_k(x)=\frac{1}{\sqrt{a_{I_0,I_0}(p,p,k)}}\,
e^{ik\Psi(x,p)}\chi\left(\frac{\abs{x-p}^2}{\varepsilon}\right)\sideset{}{'}\sum_{\abs{I}=q}
d_{I,I_0}(x,p,k)e^I(x)\in\Omega^{0,q}_0(D_0),\\
&u_k(x)=s^k(x)\Td u_k(x)e^{k\phi(x)}\in\Omega^{0,q}_0(D_0,L^k),
\end{split}
\end{equation}
where $\varepsilon>0$ is a small constant and $d_{I,I_0}(x,y)$ is as in \eqref{e-gue140504IV}. We need

\begin{lem}\label{l-gue140504II}
We have 
\[u_k(x)\equiv\frac{P^{(q)}_{k,k^{-N}}u_k(x)}{\|P^{(q)}_{k,k^{-N}}u_k\|_{k}}
\mod O(k^{-\infty})\ \ \mbox{on $M$},\]
that is, for every local trivializing holomorphic section $s_1$ of $L$ on an open set $W\Subset M$, 
$\abs{s_1}^2_{h}=e^{-2\phi_1}$, we have 
\[s^{-k}_1e^{-k\phi_1}u_k(x)\equiv s_1^{-k}e^{-k\phi}\frac{P^{(q)}_{k,k^{-N}}u_k(x)}
{\|P^{(q)}_{k,k^{-N}}u_k\|_{k}}\mod O(k^{-\infty})\ \ \mbox{on $W$}.\]
\end{lem}

\begin{proof}
It is known from \cite[Theorems\,3.11--12]{HM12} that 
\begin{equation}\label{e-gue140504VII}
\begin{split}
&\int e^{ik\Psi(x,y)}e^{-ik\ol\Psi(x,y)}
\abs{\chi\Biggl(\frac{\abs{x-y}^2}{\varepsilon }\Biggr)}^2
\sideset{}{'}\sum_{\abs{I}=q}\abs{d_{I,I_0}(x,y)}^2dv_M(x)\\
&\equiv a_{I_0,I_0}(y,y,k)\mod O(k^{-\infty}).
\end{split}
\end{equation}
From \eqref{e-gue140504VII}, it is easy to see that 
\begin{equation}\label{e-gue140504VIII}
\|{u_k}\|_{k}\equiv1\mod O(k^{-\infty}). 
\end{equation}
Moreover, we have by \cite[Theorem\,3.11]{HM12},
\begin{equation}\label{e-gue140504a}
\left\|(\Box^{(q)}_k)^ju_k\right\|_{k}\equiv0\mod O(k^{-\infty}),\ \ j=1,2,\ldots.
\end{equation}
From \eqref{e-gue140504a}, we have 
\begin{equation}\label{e-gue140504aI}
\left\|u_k-P^{(q)}_{k,k^{-N}}u_k\right\|_{k}\leq 
k^{N}\left\|\Box^{(q)}_ku_k\right\|_{k}\equiv0\mod O(k^{-\infty}).
\end{equation}
From \eqref{e-gue140504a} and semi-classical G{\aa}rding inequalities (see~\cite[Lemma\,4.1]{HM12}), we obtain 
\begin{equation}\label{e-gue140504aII}
u_k\equiv P^{(q)}_{k,k^{-N}}u_k\mod O(k^{-\infty}).
\end{equation}
From \eqref{e-gue140504aII}, \eqref{e-gue140504VIII}, the lemma follows.
\end{proof}
%------
%-----
\begin{lem}\label{l-gue140504III}
With the notations and assumptions above, assume that for $k$ large, 
${\rm dist\,}(p,y)\geq c\frac{\log k}{\sqrt{k}}$, $\forall y\in D_1$, 
where $c>0$ is a constant independent of $k$. Then, there exists
$C_{\alpha,\beta}>0$ independent of $K$ and the point $p$ such that 
\[\abs{\pr^\alpha_x\pr^\beta_yP^{(q),I_0,J_0}_{k,k^{-N},s,\widehat s}(p,y)}\leq 
C_{\alpha,\beta}k^{2n-\frac{N}{2}+2(\abs{\alpha}+\abs{\beta})},\ \ \forall y\in D_1.\]
\end{lem}

\begin{proof}
Let us choose 
\begin{equation}\label{l-gue141104}
g_1=\dfrac{P^{(q)}_{k,k^{-N}}u_k}{\|P^{(q)}_{k,k^{-N}}u_k\|_k}
\end{equation} 
in the orthonormal frame $\set{g_1(x),g_2(x),\ldots,g_{d_k}(x)}$ of
$\cE_{k^{-N}}(M,L^k)$ (see \eqref{e-gue140503II}).
From \eqref{e-gue140504IV}, \eqref{e-gue140504VI} and 
Lemma~\ref{l-gue140504II}, it is not difficult to check that 
\begin{equation}\label{e-gue140504ab}
\abs{\pr^\alpha_x(\Td g_{1,I_0}e^{-k\phi})(p)}^2\equiv a_{I_0,I_0}(p,p)\mod O(k^{-\infty}).
\end{equation}
From \eqref{e-gue140504III}, \eqref{e-gue140504IV} and Theorem~\ref{t-gue140503I}, we conclude that 
\begin{equation}\label{e-gue140504aIV}
\sum^{d_k}_{j=2}\abs{\pr^\alpha_x(\Td g_{j,I_0}e^{-k\phi})(p)}^2\leq k^{3n-N+4\abs{\alpha}},
\end{equation}
where $C>0$ is a constant independent of $k$ and the point $p$. From \eqref{e-gue140504}, \eqref{e-gue140504III} and \eqref{e-gue140504aIV}, we have 
\begin{equation}\label{e-gue140504aV}
\begin{split}
&\abs{\pr^\alpha_x\pr^\beta_yP^{(q),I_0,J_0}_{k,k^{-N},s,\widehat s}(p,y)}\\
&\leq
\abs{\pr^\alpha_x(\Td g_{1,I_0}e^{-k\phi})(p)}
\abs{\pr^\beta_y(\widehat g_{1,J_0}e^{-k\widehat\phi})(y)}+
\sqrt{\sum^{d_k}_{j=2}\abs{\pr^\alpha_x(\Td g_{j,I_0}e^{-k\phi})(p)}^2}\
\sqrt{\sum^{d_k}_{j=2}\abs{\pr^\beta_y(\widehat g_{j,J_0}e^{-k\widehat\phi})(y)}^2}\\
&\leq\abs{\pr^\alpha_x(\Td g_{1,I_0}e^{-k\phi})(p)}
\abs{\pr^\beta_y(\widehat g_{1,J_0}e^{-k\widehat\phi})(y)}+C_1k^{2n-\frac{N}{2}+2(\abs{\alpha}+\abs{\beta})},
\end{split}
\end{equation}
where $C_1>0$ is a constant independent of $k$ and the point $p$. 
From Lemma~\ref{l-gue140504II} and noting that 
$\Td u_k(y)\equiv0\mod O(k^{-\infty})$ if ${\rm dist\,}(p,y)\geq c\frac{\log k}{\sqrt{k}}$, 
where $c>0$ is a constant independent of $k$, we conclude that 
\[\abs{\pr^\beta_y(\widehat g_{1,J_0}e^{-k\widehat\phi})(y)}\equiv0\mod O(k^{-\infty}),\ \ \forall y\in D_1.\]
From this observation and \eqref{e-gue140504aV}, the lemma follows. 
\end{proof}
%--------------
From Lemma~\ref{l-gue140504}, Lemma~\ref{l-gue140504I}, Lemma~\ref{l-gue140504III}, 
Theorem~\ref{t-gue140503} follows. 

We can repeat the proof of Theorem~\ref{t-gue140503} and conclude: 
%--------------
\begin{thm}\label{t-gue140505}
Let $s$ and $\widehat s$ be local trivializing holomorphic sections of $L$ on open sets $D_0\Subset M$, $D_1\Subset M$
 respectively, $\abs{s}^2_{h}=e^{-2\phi}$, $\abs{\widehat s}^2_{h}=e^{-2\widehat\phi}$. 
 Assume that $D_0\Subset M(j)$, $j\neq q$. Then 
\[P^{(q)}_{k,s,\widehat s}(x,y)\equiv0\mod O(k^{-\infty})\ \ \mbox{locally uniformly on $D_0\times D_1$}.\]
Assume now that $D_0\Subset M(q)$ and $\Box^{(q)}_k$ has an $O(k^{-N})$ small spectral gap 
on $D_0$. Suppose that  $\ol D_0\bigcap\ol D_1=\emptyset$. 
%for $k$ large, ${\rm dist\,}(x,y)\geq c\frac{\log k}{\sqrt{k}}$, $\forall x\in D_0, y\in D_1$, 
%where $c>0$ is a constant independent of $k$. 
Then, 
\[\mbox{$P^{(q)}_{k,s,\widehat s}(x,y)\equiv0\mod O(k^{-\infty})$ locally uniformly on $D_0\times D_1$}.\]
\end{thm}
%--------------
\section{Berezin-Toeplitz kernel estimates away the diagonal}\label{s:toediag}
%--------------
In this Section we prove the off-diagonal decay of the kernel $T^{(q),f}_{k,k^{-N}}(\cdot,\cdot)$
of the Berezin-Toeplitz quantization (cf.\ \eqref{e-gue140419}), where $f\in\cC^\infty(M)$ is
as usual a bounded function and $N>1$.
This yields one half of Theorem \ref{t-gue140521}, i.\,e.\ \eqref{e1.1}.

We consider as before the localization of $T^{(q),f}_{k,k^{-N}}(\cdot,\cdot)$ as follows.
Let $s$, $\widehat s$ be local trivializing holomorphic sections of $L$ on open sets $D_0\Subset M$, $D_1\Subset M$ 
respectively, $\abs{s}^2_{h}=e^{-2\phi}$, $\abs{\widehat s}^2_{h}=e^{-2\widehat\phi}$.  
Let $\set{e_1,e_2,\ldots,e_n}$ and $\set{w_1,w_2,\ldots,w_n}$ be orthonormal frames of $\Lambda^{0,1}(T^*M)$ on $D_0$
 and $D_1$ respectively. Then, $\set{e^J;\, \abs{J}=q, \ \mbox{$J$ strictly increasing}}$, 
 $\set{w^J;\, \abs{J}=q, \ \mbox{$J$ strictly increasing}}$ are orthonormal frames 
 of $\Lambda^{0,q}(T^*M)$ on $D_0$ and $D_1$ respectively. As in \eqref{e-gue140503}, we write 
\begin{equation}\label{e-gue140506}
\begin{split}
T^{(q),f}_{k,k^{-N},s,\widehat s}=\sideset{}{'}\sum\limits_{\abs{I}=
\abs{J}=q}T^{(q),f,I,J}_{k,k^{-N},s,\widehat s}(x,y)e^I(x)\wedge(w^J(y))^\dagger,\:\:
T^{(q),f,I,J}_{k,k^{-N},s,\widehat s}\in\cC^\infty(D_0\times D_1).
%&\mbox{$T^{(q),f,I,J}_{k,k^{-N},s,\widehat s}\in\cC^\infty(D_0\times D_1)$, 
%$\forall\abs{I}=\abs{J}=q$, $I$, $J$ are strictly increasing}. 
\end{split}
\end{equation}
Let $\set{g_j}_{j=1}^{d_k}$ and $\set{\delta_j}_{j=1}^{d_k}$ 
be orthonormal bases of
$\cE^q_{k^{-N}}(M,L^k)$, where $d_k\in\mathbb N\cup\set{\infty}$. 
On $D_0$, $D_1$, we write 
\begin{equation}\label{e-gue140506I}
\begin{split}
&g_j(x)=s^k(x)\Td g_j(x),\ \ \mbox{$\Td g_j(x)=
\sideset{}{'}\sum\limits_{\abs{J}=q}\Td g_{j,J}(x)e^J(x)$ on $D_0$},\ \ j=1,\ldots,d_k,\\
&\delta_j(x)=\widehat s^k(x)\widehat\delta_j(x),\ \ \mbox{$\widehat\delta_j(x)=
\sideset{}{'}\sum\limits_{\abs{J}=q}\widehat\delta_{j,J}(x)w^J(x)$ on $D_1$},\ \ j=1,\ldots,d_k.
\end{split}
\end{equation}
It is not difficult to check that for every strictly increasing $I$, $J$, $\abs{I}=\abs{J}=q$, we have 
\begin{equation}\label{e-gue140506II}
T^{(q),f,I,J}_{k,k^{-N},s,\widehat s}(x,y)=
\sum\limits^{d_k}_{j,\ell=1}\Td g_{j,I}(x)
e^{-k\phi(x)}(\,f\delta_\ell\,|\,g_j\,)_{k}\,\ol{\widehat\delta_{\ell,J}(y)}e^{-k\widehat\phi(y)}.
\end{equation}
%---------
\begin{lem}\label{l-gue140506}
Assume that $D_0\Subset M(j)$, $j\neq q$. Then, for every $N>1$ and 
$m\in\mathbb N$, there exists $C_{N,m}>0$ independent of $k$ 
such that for every $I$, $J$ strictly increasing, $\abs{I}=\abs{J}=q$, we have
\[\abs{T^{(q),f,I,J}_{k,k^{-N},s,\widehat s}(x,y)}_{\cC^m(D_0\times D_1)}
\leq C_{N,m}k^{2n-\frac{N}{2}+2m}.\]
\end{lem}

\begin{proof}
Fix $\alpha, \beta\in\mathbb N^{2n}_0$, $I_0$, $J_0$
strictly increasing with $\abs{I_0}=\abs{J_0}=q$,  and $(x_0, y_0)\in D_0\times D_1$. 
Take $\set{g_1, g_2,\ldots,g_{d_k}}$ and $\set{\delta_1, \delta_2,\ldots,\delta_{d_k}}$ so that 
\begin{equation}\label{e-gue140506III}
\begin{split}
&\abs{\pr^\alpha_x(\Td g_{1,I_0}e^{-k\phi})(x_0)}^2=
\sum^{d_k}_{j=1}\abs{\pr^\alpha_x(\Td g_{j,I_0}e^{-k\phi})(x_0)}^2,\\
&\abs{\pr^\beta_y(\widehat\delta_{1,J_0}e^{-k\widehat\phi})(y_0)}^2=
\sum^{d_k}_{j=1}\abs{\pr^\beta_y(\widehat\delta_{j,J_0}e^{-k\widehat\phi})(y_0)}^2. 
\end{split}
\end{equation}
This is always possible, see \cite[Proposition\,4.5]{HM12}. From \eqref{e-gue140506II} 
and \eqref{e-gue140506III}, we see that 
\begin{equation}\label{e-gue140506IV}
(\pr^\alpha_x\pr^\beta_yT^{(q),f,I_0,J_0}_{k,k^{-N,s,\widehat s}})(x_0,y_0)=
\pr^\alpha_x(\Td g_{1,I_0}e^{-k\phi})(x_0)(\,f\delta_1\,|\,g_1\,)_{k}
\pr^\beta_y(\ol{\widehat\delta_{1,J_0}}e^{-k\widehat\phi})(y_0).
\end{equation}
In view of Theorem~\ref{t-gue140503I} and \eqref{e-gue140504III}, we see that 
\begin{equation}\label{e-gue140506V}
\begin{split}
&\abs{\pr^\alpha_x(\Td g_{1,I_0}e^{-k\phi})(x_0)}^2\leq C_\alpha k^{3n-N+4\abs{\alpha}},\\
&\abs{\pr^\beta_y(\widehat\delta_{1,J_0}e^{-k\widehat\phi})(y_0)}^2\leq C_\beta k^{n+4\abs{\beta}},
\end{split}
\end{equation}
where $C_\alpha>0$, $C_\beta>0$ are constants independent of $k$ and the points $x_0$ 
and $y_0$. From \eqref{e-gue140506IV} and \eqref{e-gue140506V}, the lemma follows. 
\end{proof}
Now, we assume that $D_0\Subset M(q)$. Fix $D_0\Subset\Td D_0\Subset M(q)$ and take 
$\tau(x)\in\cC^\infty_0(\Td D_0)$, $\tau=1$ on $D_0$.
%----------------
\begin{lem}\label{l-gue140506I}
With the assumptions and notations above, for every $N>1$ and $m\in\mathbb N$, 
there exists $C_{N,m}>0$ independent of $k$ such that 
\[\abs{T^{(q),(1-\tau)f,I,J}_{k,k^{-N},s,\widehat s}(x,y)}_{\cC^m(D_0\times D_1)}\leq 
C_{N,m}k^{2n-\frac{N}{2}+2m},\]
for every strictly increasing multi-indices $I$, $J$, $\abs{I}=\abs{J}=q$. 
\end{lem}

\begin{proof}
Fix $\alpha, \beta\in\mathbb N^{2n}_0$, $\abs{I_0}=\abs{J_0}=q$, $I_0$, $J_0$ 
are strictly increasing and $(p,y_0)\in D_0\times D_1$. Take $\set{\delta_1,\delta_2,\ldots,\delta_{d_k}}$ so that 
\begin{equation}\label{e-gue140506VI}
\abs{\pr^\beta_y(\widehat\delta_{1,J_0}e^{-k\widehat\phi})(y_0)}^2=
\sum^{d_k}_{j=1}\abs{\pr^\beta_y(\widehat\delta_{j,J_0}e^{-k\widehat\phi})(y_0)}^2.
\end{equation}
Assume that $a_{I_0,I_0}(p,p,k)\leq0$, where $a_{I,J}(x,y,k)$ is as in \eqref{e-gue140504IV}. 
From \eqref{e-gue140504V} and \eqref{e-gue140504III}, we have
\begin{equation}\label{e-gue140506b}
\begin{split}
&\abs{\pr^\alpha_x\pr^\beta_yT^{(q),(1-\tau)f,I_0,J_0}_{k,k^{-N},s,\widehat s}(p,y_0)}\\
&=\abs{\sum^{d_k}_{j=1}\pr^\alpha_x(\Td g_{j,I_0}e^{-k\phi})(p)(\,(1-\tau)f\delta_1\,|\,g_j\,)_{k}
\pr^\beta_y(\ol{\widehat\delta_{1,J_0}}e^{-k\widehat\phi})(y_0)}\\
&\leq C\sqrt{\sum^{d_k}_{j=1}\abs{\pr^\alpha_x(\Td g_{j,I_0}e^{-k\phi})(p)}^2}
\abs{\pr^\beta_y(\ol{\widehat\delta_{1,J_0}}e^{-k\widehat\phi})(y_0)}\\
&\leq C_{\alpha,\beta}(k^{3n-N+4\abs{\alpha}}k^{n+4\abs{\beta}})^{\frac{1}{2}}=
C_{\alpha,\beta}k^{2n-\frac{N}{2}+2\abs{\alpha}+2\abs{\beta}},
\end{split}
\end{equation}
where $C_{\alpha,\beta}>0$, $C>0$ are constants independent of $k$ and the points $p$, $y_0$.

Now, we assume that $a_{I_0,I_0}(p,p,k)>0$. We define now $u_k$ is as in \eqref{e-gue140504VI} and
$g_1$ as in \eqref{l-gue141104}. 
%Let $g_1=\frac{P^{(q)}_{k,k^{-N}}u_k}{\norm{P^{(q)}_{k,k^{-N}}u_k}_{k}}$, 
%where $u_k$ is as in \eqref{e-gue140504VI}. 
Since $g_1\equiv u_k\mod O(k^{-\infty})$ and 
$u_k(x)\equiv0\mod O(k^{-\infty})$ if ${\rm dist\,}(x,p)\geq c\frac{\log k}{\sqrt{k}}$, 
where $c>0$ is a constant independent of $k$, we conclude that 
\begin{equation}\label{e-gue140506VII}
\abs{\pr^\alpha_x(\Td g_{1,I_0}e^{-k\phi})(p)(\,(1-\tau)f\delta_1\,|\,g_1\,)_{k}
\pr^\beta_y(\ol{\widehat\delta_{1,J_0}}e^{-k\widehat\phi})(y_0)}\equiv0\mod O(k^{-\infty}).
\end{equation}
From \eqref{e-gue140504aIV} and \eqref{e-gue140504III}, we have 
\begin{equation}\label{e-gue140506VIII}
\begin{split}
&\abs{\sum^{d_k}_{j=2}\pr^\alpha_x(\Td g_{j,I_0}e^{-k\phi})(p)(\,(1-\tau)f\delta_1\,|\,g_j\,)_{k}
\pr^\beta_y(\ol{\widehat\delta_{1,J}}e^{-k\widehat\phi})(y_0)}\\
&\leq C_1\sqrt{\sum^{d_k}_{j=2}\abs{\pr^\alpha_x(\Td g_{j,I_0}e^{-k\phi})(p)}^2}
\abs{\pr^\beta_y(\ol{\widehat\delta_{1,J}}e^{-k\widehat\phi})(y_0)}\\
&\leq\Td C_{\alpha,\beta}(k^{3n-N+4\abs{\alpha}}k^{n+4\abs{\beta}})^{\frac{1}{2}}=
\Td C_{\alpha,\beta}k^{2n-\frac{N}{2}+2\abs{\alpha}+2\abs{\beta}},
\end{split}
\end{equation}
where $C_1>0$, $\Td C_{\alpha,\beta}>0$ are constants independent of $k$ and the points $p$, $y_0$. 
From \eqref{e-gue140506VII} and \eqref{e-gue140506VIII}, we obtain 
\begin{equation}\label{e-gue140506a}
\abs{\pr^\alpha_x\pr^\beta_yT^{(q),(1-\tau)f,I_0,J_0}_{k,k^{-N},s,\widehat s}(p,y_0)}\leq
\widehat C_{\alpha,\beta}k^{2n-\frac{N}{2}+2\abs{\alpha}+2\abs{\beta}},
\end{equation}
where $\widehat C_{\alpha,\beta}>0$ is a constant independent of $k$ and the points $p$, $y_0$. 

From \eqref{e-gue140506b} and \eqref{e-gue140506a}, the lemma follows. 
\end{proof}

\begin{lem}\label{l-gue140506II}
With the assumptions and notations above, assume that  $\ol D_0\bigcap\ol D_1=\emptyset$. 
%for $k$ large, 
%${\rm dist\,}(x,y)\geq c\frac{\log k}{\sqrt{k}}$, $\forall x\in\Td D_0$, $y\in D_1$, 
%where $c>0$ is a constant independent of $k$. 
Then, for every $N>1$ and $m\in\mathbb N$, 
there exists $C_{N,m}>0$ independent of $k$ such that 
\[\abs{T^{(q),\tau f,I,J}_{k,k^{-N},s,\widehat s}(x,y)}_{\cC^m(D_0\times D_1)}\leq 
C_{N,m}k^{2n-\frac{N}{2}+2m},\]
for every $I$, $J$ strictly increasing, $\abs{I}=\abs{J}=q$.
\end{lem}

\begin{proof}
Fix $\alpha, \beta\in\mathbb N^{2n}_0$, $\abs{I_0}=\abs{J_0}=q$, $I_0$, $J_0$ 
are strictly increasing and $(p,y_0)\in D_0\times D_1$. Take $\set{\delta_1,\delta_2,\cdots,\delta_{d_k}}$ so that 
\begin{equation}\label{e-gue140506aI}
\abs{\pr^\beta_y(\widehat\delta_{1,J_0}e^{-k\phi})(y_0)}^2=
\sum^{d_k}_{j=1}\abs{\pr^\beta_y(\widehat\delta_{j,J_0}e^{-k\widehat\phi})(y_0)}^2.
\end{equation}
Assume that $a_{I_0,I_0}(p,p,k)\leq0$, where $a_{I,J}(x,y,k)$ is as in \eqref{e-gue140504IV}. 
We can repeat the procedure in the proof of Lemma~\ref{l-gue140506I} and conclude that 
\begin{equation}\label{e-gue140506aII}
\abs{\pr^\alpha_x\pr^\beta_yT^{(q),\tau f,I_0,J_0}_{k,k^{-N},s,\widehat s}(p,y_0)}\leq 
C_{\alpha,\beta}k^{2n-\frac{N}{2}+2\abs{\alpha}+2\abs{\beta}},
\end{equation}
where $C_{\alpha,\beta}>0$ is a constant independent of $k$ and the points $p$ and $y_0$. 
Now, we assume that $a_{I_0,I_0}(p,p,k)>0$. Let $g_1$ be as in \eqref{l-gue141104}, 
where $u_k$ is as in \eqref{e-gue140504VI}. From Lemma~\ref{l-gue140504II} and 
\eqref{e-gue140504VI}, we have 
\begin{equation}\label{e-gue140506aIII}
\begin{split}
&\Td g_1(x)e^{-k\phi(x)}\equiv\frac{1}{\sqrt{a_{I_0,I_0}(p,p,k)}}\,e^{ik\Psi(x,p)}
\,\chi\!\left(\frac{\abs{x-p}^2}{\varepsilon}\right)\sideset{}{'}\sum_{\abs{I}=q}d_{I,I_0}(x,p,k)e^I(x)
\mod O(k^{-\infty}),\\
&\pr^\alpha_x(\Td g_{1,I_0}e^{-k\phi})(p)\equiv\sqrt{a_{I_0,I_0}(p,p,k)}\mod O(k^{-\infty}).
\end{split}
\end{equation}
From \eqref{e-gue140506aIII}, it is straightforward to see  that for every $N\in\mathbb N$, 
there exists $C_N>0$ independent of $k$ and the points $p$ and $y_0$ such that 
\begin{equation}\label{e-gue140506aIV}
\begin{split}
&\abs{\pr^\alpha_x(\Td g_{1,I_0}e^{-k\phi})(p)(\,\tau f\delta_1\,|\,g_1\,)_{k}
\pr^\beta_y(\ol{\widehat\delta_{1,J_0}}e^{-k\widehat\phi})(y_0)}\\
&\leq\int e^{-k{\rm Im\,}\Psi(x,p)}\chi\!\left(\frac{\abs{x-p}^2}{\varepsilon}\right)\sideset{}{'}\sum_{\abs{I}=
q}\abs{d_{I,I_0}(x,p,k)}\abs{\tau(x)}\abs{f(x)}dv_M(x)\\
&\times\sideset{}{'}\sum_{\abs{I}=q}\sup 
\set{\abs{\widehat\delta_{1,I}(x)e^{-k\widehat\phi(x)}\pr^\beta_y(\ol{\widehat\delta_{1,J_0}}
e^{-k\widehat\phi})(y_0)};\, x\in{\operatorname{Supp}}\,\chi\!\left(\frac{\abs{x-p}^2}{\varepsilon}\right)
\Subset\Td D_0}+C_Nk^{-N}.
\end{split}
\end{equation}
From \eqref{e-gue140504IV}, we can check that 
\begin{equation}\label{e-gue140506aV}
\sideset{}{'}\sum_{\abs{I}=q}\abs{d_{I,I_0}(x,p,k)}\leq 
C_\alpha k^{n+\abs{\alpha}},\ \ \forall x\in{\rm Supp\,}\chi\!\left(\frac{\abs{x-p}^2}{\varepsilon}\right),
\end{equation}
where $C_\alpha>0$ is a constant independent of $k$ and the point $p$. F
rom \eqref{e-gue140506aV} and \eqref{s3-e16-bisbg}, it is not-difficult to check that 
\begin{equation}\label{e-gue140506aVI}
\int e^{-k{\rm Im\,}\Psi(x,p)}\chi\!\left(\frac{\abs{x-p}^2}{\varepsilon}\right)\sideset{}{'}\sum_{\abs{I}=
q}\abs{d_{I,I_0}(x,p,k)}\abs{\tau(x)}\abs{f(x)}dv_M(x)\leq C_0k^{\abs{\alpha}},
\end{equation}
where $C_0>0$ is a constant independent of $k$ and the point $p$. Moreover, from Theorem~\ref{t-gue140503}, we see that 
\begin{equation}\label{e-gue140506aVII}
\sideset{}{'}\sum_{\abs{I}=q}\sup \set{\abs{\widehat\delta_{1,I}(x)
e^{-k\widehat\phi(x)}\pr^\beta_y(\ol{\widehat\delta_{1,J_0}}e^{-k\widehat\phi})(y_0)};
\, x\in{\rm Supp\,}\chi\!\left(\frac{\abs{x-p}^2}{\varepsilon}\right)}\leq C_\beta 
k^{2n-\frac{N}{2}+2\abs{\beta}},
\end{equation}
where $C_{\beta}>0$ is a constant independent of $k$ and the points $p$, $y_0$.
From \eqref{e-gue140506aIV}, \eqref{e-gue140506aVI} and \eqref{e-gue140506aVII}, we conclude that 
\begin{equation}\label{e-gue140506aVIII}
\abs{\pr^\alpha_x(\Td g_{1,I_0}e^{-k\phi})(p)(\,\tau f\delta_1\,|\,g_1\,)_{k}
\pr^\beta_y(\widehat\delta_{1,J_0}e^{-k\widehat\phi})(y_0)}\leq 
C_{\alpha,\beta}k^{2n-\frac{N}{2}+2\abs{\alpha}+2\abs{\beta}},
\end{equation}
where $C_{\alpha,\beta}>0$ is a constant independent of $k$ and the points $p$, $y_0$. 

From \eqref{e-gue140504aIV} and \eqref{e-gue140504III}, we have 
\begin{equation}\label{e-gue140506f}
\begin{split}
&\abs{\sum^{d_k}_{j=2}\pr^\alpha_x(\Td g_{j,I_0}e^{-k\phi})(p)
(\,\tau f\delta_1\,|\,g_j\,)_{k}\pr^\beta_y(\ol{\widehat\delta_{1,J_0}}e^{-k\widehat\phi})(y_0)}\\
&\leq C_2\sqrt{\sum^{d_k}_{j=2}\abs{\pr^\alpha_x(\Td g_{j,I_0}e^{-k\phi})(p)}^2}
\abs{\pr^\beta_y(\ol{\widehat\delta_{1,J_0}}e^{-k\widehat\phi})(y_0)}\\
&\leq\widehat C_{\alpha,\beta}(k^{3n-N+4\abs{\alpha}}k^{n+4\abs{\beta}})^{\frac{1}{2}}=
\widehat C_{\alpha,\beta}k^{2n-\frac{N}{2}+2\abs{\alpha}+2\abs{\beta}},
\end{split}
\end{equation}
where $C_2>0$, $\widehat C_{\alpha,\beta}>0$ are constants independent of $k$ and the point $p$. 
From \eqref{e-gue140506aVIII} and \eqref{e-gue140506f}, the lemma follows. 
\end{proof}

From Lemmas~\ref{l-gue140506}-\ref{l-gue140506II}
we deduce: 

\begin{thm}\label{t-gue140506}
Let $s$, $\widehat s$ be local trivializing holomorphic sections of $L$ on $D_0\Subset M$ 
and $D_1\Subset M$ respectively.
%, $\abs{s}^2_{h}=e^{-2\phi}$, $\abs{\widehat s}^2_{h}=e^{-2\widehat\phi}$, 
%where $D_0$ and $D_1$ are open sets. 
Assume that $D_0\Subset M(j)$, 
$j\neq q$ or $D_0\Subset M(q)$ and $\ol D_0\bigcap\ol D_1=\emptyset$. 
%for $k$ large, ${\rm dist\,}(x,y)\geq c\frac{\log k}{\sqrt{k}}$, 
%$\forall x\in D_0$, $\forall y\in D_1$, where $c>0$ is a constant independent of $k$. 
Then, for every $m\in\mathbb N$, $N>1$, there exists $C_{N,m}>0$ 
independent of $k$ such that 
\[\abs{T^{(q),f}_{k,k^{-N},s,\widehat s}(x,y)}_{\cC^m(D_0\times D_1)}
\leq C_{N,m}k^{2n-\frac{N}{2}+2m}.\]
\end{thm}
Theorem \ref{t-gue140506} implies immediately one half of
Theorem \ref{t-gue140521}, more precisely\ \eqref{e1.1}.

We can repeat the proof of Theorem~\ref{t-gue140506} and deduce: 
\begin{thm}\label{t-gue140508}
Let $s$, $\widehat s$ be local trivializing holomorphic sections of $L$ on $D_0\Subset M$ and 
$D_1\Subset M$ respectively.
%$\abs{s}^2_{h}=e^{-2\phi}$, 
%$\abs{\widehat s}^2_{h}=e^{-2\widehat\phi}$, where $D_0$ and $D_1$ are open sets. 
Assume that $D_0\Subset M(j)$, $j\neq q$. Then, 
\[\mbox{$T^{(q),f}_{k,s,\widehat s}(x,y)\equiv0\mod O(k^{-\infty})$ 
locally uniformly on $D_0\times D_1$}.\]
Assume that $D_0\Subset M(q)$ and $\Box^{(q)}_k$ has $O(k^{-N})$ 
small spectral gap on $D_0$. Suppose that  $\ol D_0\bigcap\ol D_1=\emptyset$.
%for $k$ large, 
%${\rm dist\,}(x,y)\geq c\frac{\log k}{\sqrt{k}}$, $\forall x\in D_0$, $\forall y\in D_1$, 
%where $c>0$ is a constant independent of $k$. 
Then, 
\[\mbox{$T^{(q),f}_{k,s,\widehat s}(x,y)\equiv0\mod O(k^{-\infty})$ 
locally uniformly on $D_0\times D_1$}.\]
\end{thm}

Let's explain why in Theorem~\ref{t-gue140508}, we have "$\equiv0\mod O(k^{-\infty})$". Recall that Theorem~\ref{t-gue140506} is based on Theorem~\ref{t-gue140503I} which says that if $D_0\Subset M(q)$, then, for every $N>1$, $m\in\mathbb N$, 
there exists $C_{N,m}>0$ independent of $k$ such that 
\begin{equation}\label{e-gue160610}
\abs{P^{(q)}_{k,k^{-N},s}(x,y)-e^{ik\Psi(x,y)}b(x,y,k)}_{\cC^m(D_0\times D_0)}\leq
 C_{N,m}k^{3n-N+2m},
 \end{equation}
and if $D_0\Subset M(j)$, $j\neq q$, $j\in\set{0,1,2,\ldots,n}$, then, for every $N>1$, $m\in\mathbb N$, there exists $\Td C_{N,m}>0$ independent of $k$ such that 
\begin{equation}\label{e-gue160610I}
\abs{P^{(q)}_{k,k^{-N},s}(x,y)}_{\cC^m(D_0\times D_0)}\leq\Td C_{N,m}k^{3n-N+2m}.
\end{equation}
The estimates $\lesssim k^{3n-N+2m}$ in \eqref{e-gue160610} and \eqref{e-gue160610I} imply that we have the estimate $\lesssim k^{2n-\frac{N}{2}+2m}$ in Theorem~\ref{t-gue140506}. Now, we consider the Bergman kernel. As in  Theorem~\ref{t-gue140503I}, assume that
$D_0\Subset M(q)$ and $\Box^{(q)}_k$ has $O(k^{-N})$ 
small spectral gap on $D_0$, then
\begin{equation}\label{e-gue160610II}
P^{(q)}_{k,s}(x,y)-e^{ik\Psi(x,y)}b(x,y,k)\equiv0\mod O(k^{-\infty})\ \ \mbox{on $D_0\times D_0$}.
 \end{equation}
Moreover, if $D_0\Subset M(j)$, $j\neq q$, $j\in\set{0,1,2,\ldots,n}$, then 
\begin{equation}\label{e-gue160610III}
P^{(q)}_{k,s}(x,y)\equiv0\mod O(k^{-\infty})\ \ \mbox{on $D_0\times D_0$},
\end{equation}
by Theorem 4.12 and Theorem 4.14 in~\cite{MM12}. 
From \eqref{e-gue160610II} and \eqref{e-gue160610III}, we can repeat the proof of Theorem~\ref{t-gue140506} and deduce that in 
Theorem~\ref{t-gue140508}, we actually have "$\equiv0\mod O(k^{-\infty})$".

%------------
\section{Asymptotic expansion of Berezin-Toeplitz quantization}\label{s-gue140508}
%------------
In this section, we will establish the full asymptotic expansion for 
the kernel of the Toeplitz kernel $T^{(q),f}_{k,k^{-N}}(\cdot,\cdot)$
corresponding to lower energy forms. 
This leads to the proof of Theorem~\ref{t-gue140521}.

Let $s$ be a local trivializing holomorphic section of $L$ on an open set $D\Subset M$, 
$\abs{s}^2_{h}=e^{-2\phi}$. Fix $N>1$. We assume that $D\Subset M(q)$. Put
\begin{equation}\label{e-gue140509}
S_k(x,y):=e^{ik\Psi(x,y)}b(x,y,k),
\end{equation}
where $\Psi(x,y)$ and $b(x,y,k)$ are as in Theorem~\ref{t-gue140503I}. 
Fix an open set $D_0\Subset D$ and $\tau\in \cC^\infty_0(D)$ with $\tau=1$ on $D_0$. Put 
\begin{equation}\label{e-gue140509I}
R_k(x,y)=\int (P^{(q)}_{k,k^{-N},s}(x,z)-S_k(x,z))\tau(z)f(z)P^{(q)}_{k,k^{-N},s}(z,y)dv_M(z).
\end{equation}
Let $\set{e_1,e_2,\ldots,e_n}$ be an orthonormal frame of $\Lambda^{0,1}(T^*M)$ on $D$. Then, 
\[\set{e^J;\, \mbox{$\abs{J}=q$, $J$ is strictly increasing}}\]
is an orthonormal frame of $\Lambda^{0,q}(T^*M)$ on $D$. As in \eqref{e-gue140503}, we write 
\begin{equation}\label{e-gue140509II}
\begin{split}
&R_k(x,y)=\sideset{}{'}\sum_{\abs{I}=\abs{J}=q}R^{I,J}_k(x,y)e^I(x)\wedge( e^J(y))^\dagger,\\
&S_k(x,y)=\sideset{}{'}\sum_{\abs{I}=\abs{J}=q}S^{I,J}_k(x,y)e^I(x)\wedge( e^J(y))^\dagger=
\sideset{}{'}\sum_{\abs{I}=\abs{J}=q}e^{ik\Psi(x,y)}b_{I,J}(x,y,k)e^I(x)\wedge( e^J(y))^\dagger,\\
&P^{(q)}_{k,k^{-N},s}(x,y)=\sideset{}{'}\sum_{\abs{I}=
\abs{J}=q}P^{(q),I,J}_{k,k^{-N},s}(x,y)e^I(x)\wedge(e^J(y))^\dagger.
\end{split}
\end{equation}
It is easy to see that for every $\abs{I}=\abs{J}=q$, $I$, $J$ are strictly increasing, we have 
\begin{equation}\label{e-gue140509III}
R^{I,J}_k(x,y)=\sideset{}{'}\sum_{\abs{K}=q}\int(P^{(q),I,K}_{k,k^{-N},s}(x,z)-
S^{I,K}_k(x,z))\tau(z)f(z)P^{(q),K,J}_{k,k^{-N},s}(z,y)dv_M(z).
\end{equation}
Take $\set{g_1(x),g_2(x),\ldots,g_{d_k}(x)}$ and $\set{\delta_1(x),\delta_2(x),\ldots,\delta_{d_k}(x)}$ 
be orthonormal frames for $\cE_{k^{-N}}(M,L^k)$, where $d_k\in\mathbb N\cup\set{\infty}$. On $D$, we write 
\begin{equation}\label{e-gue140509a}
\begin{split}
&g_j(x)=s^k(x)\Td g_j(x),\ \ \mbox{$\Td g_j(x)=
\sideset{}{'}\sum\limits_{\abs{J}=q}\Td g_{j,J}(x)e^J(x)$ on $D$},\ \ j=1,\ldots,d_k,\\
&\delta_j(x)=s^k(x)\Td\delta_j(x),\ \ \mbox{$\Td\delta_j(x)=
\sideset{}{'}\sum\limits_{\abs{J}=q}\Td\delta_{j,J}(x)e^J(x)$ on $D$},\ \ j=1,\ldots,d_k.
\end{split}
\end{equation}
%--------------------
\begin{lem}\label{l-gue140509}
With the assumptions and notations above, for every $N>1$ and $m\in\mathbb N$, 
there exists $C_{N,m}>0$ independent of $k$ such that 
\[\abs{R_k(x,y)}_{\cC^m(D\times D)}\leq C_{N,m}k^{2n-\frac{N}{2}+2m}.\]
\end{lem}
%--------------------
\begin{proof}
Fix $(p,y_0)\in D\times D$, strictly increasing $I_0$, $J_0$, $\abs{I_0}=\abs{J_0}=q$, 
and $\alpha, \beta\in\mathbb N^{2n}_0$. Assume that $a_{I_0,I_0}(p,p,k)\leq0$, 
where $a_{I,J}(x,y,k)$ is as in \eqref{e-gue140504IV}. In view of the proof of Lemma~\ref{l-gue140504I}, we see that 
\begin{equation}\label{e-gue140509IV}
\abs{\pr^\alpha_x\pr^\alpha_yP^{(q),I_0,I_0}_{k,k^{-N},s}(p,p)}+
\abs{a_{I_0,I_0}(p,p)}\leq C_\alpha k^{3n-N+4\abs{\alpha}},
\end{equation}
where $C_\alpha>0$ is a constant independent of $k$ and the point $p$. 
It is not difficult to see that for every $\abs{I}=\abs{J}=q$, $I$, $J$ are strictly increasing, we have 
\begin{equation}\label{e-gue140509aI}
\begin{split}
&R^{I,J}_k(x,y)
=\sideset{}{'}\sum_{\abs{K}=q}\int\Bigr(\sum^{d_k}_{j=1}\Td g_{j,I}(x)
e^{-k\phi(x)}\ol{\Td g_{j,K}}(z)e^{-k\phi(z)}-S^{I,K}_k(x,z)\Bigr)\tau(z)f(z)\\
&\quad\quad\quad\times\Bigr(\sum^{d_k}_{\ell=1}\Td\delta_{\ell,K}(z)e^{-k\phi(z)}
\ol{\Td\delta_{\ell,J}}(y)e^{-k\phi(y)}\Bigr)dv_M(z).
\end{split}
\end{equation}
Take $\set{g_1,g_2,\ldots,g_{d_k}}$ and $\set{\delta_1,\delta_2,\ldots,\delta_{d_k}}$ so that 
\begin{equation}\label{e-gue140509aII}
\begin{split}
&\sum^{d_k}_{j=1}\abs{\pr^\alpha_x(\Td g_{j,I_0}e^{-k\phi})(p)}^2=
\pr^\alpha_x\pr^\alpha_yP^{(q)}_{k,k^{-N},s}(p,p)=\abs{\pr^\alpha_x(\Td g_{1,I_0}e^{-k\phi})(p)}^2,\\
&\sum^{d_k}_{j=1}\abs{\pr^\beta_x(\Td\delta_{j,J_0}e^{-k\phi})(y_0)}^2=
\pr^\beta_x\pr^\beta_yP^{(q)}_{k,k^{-N},s}(y_0,y_0)=\abs{\pr^\beta_x(\Td\delta_{1,J_0}e^{-k\phi})(y_0)}^2.
\end{split}
\end{equation}
From \eqref{e-gue140509aII} and \eqref{e-gue140509aI}, we get 
\begin{equation}\label{e-gue140509aIII}
\begin{split}
&\pr^\alpha_x\pr^\beta_yR^{I_0,J_0}_k(p,y_0)
=\sideset{}{'}\sum_{\abs{K}=q}\int\Bigr(\pr^\alpha_x(\Td g_{1,I_0}
e^{-k\phi})(p)\ol{\Td g_{1,K}}(z)e^{-k\phi(z)}-\pr^\alpha_xS^{I_0,K}_k(p,z)\Bigr)\tau(z)f(z)\\
&\quad\quad\quad\quad\quad\times\Td\delta_{1,K}(z)e^{-k\phi(z)}
\pr^\beta_y(\ol{\Td\delta_{1,J_0}}e^{-k\phi})(y_0)dv_M(z).
\end{split}
\end{equation}
From \eqref{e-gue140509IV}, \eqref{e-gue140509aII} and \eqref{e-gue140504III}, 
we have 
\begin{equation}\label{e-gue140509aIV}
\begin{split}
&\abs{\sideset{}{'}\sum_{\abs{K}=q}\int\Bigr(\pr^\alpha_x(\Td g_{1,I_0}
e^{-k\phi})(p)\ol{\Td g_{1,K}}(z)e^{-k\phi(z)}\tau(z)f(z)\Td\delta_{1,K}(z)
\pr^\beta_y(\ol{\Td\delta_{1,J_0}}e^{-k\phi})(y_0)e^{-k\phi(z)}dv_M(z)}\\
&\leq C\sqrt{\pr^\alpha_x\pr^\alpha_yP^{(q),I_0,I_0}_{k,k^{-N},s}(p,p)}
\abs{\pr^\beta_y(\ol{\Td\delta_{1,J_0}}e^{-k\phi})(y_0)}\\
&\leq C_{\alpha,\beta}k^{\frac{3}{2}n-\frac{N}{2}+2\abs{\alpha}}
k^{\frac{n}{2}+2\abs{\beta}}=C_{\alpha,\beta}k^{2n-\frac{N}{2}+2\abs{\alpha}+2\abs{\beta}},
\end{split}
\end{equation}
where $C>0$, $C_{\alpha,\beta}>0$ are constants independent of $k$ and the points $p$ and $y_0$. 

It is known by \cite[Theorem\,3.11]{HM12} that 
\begin{equation}\label{e-gue140509V}
\sideset{}{'}\sum_{\abs{K}=q}\int\abs{\pr^\alpha_x
S^{I_0,K}_k(p,z)}^2\abs{\tau(z)}^2e^{-2k\phi(z)}dv_M(z)\equiv a_{I_0,I_0}(p,p,k)\mod O(k^{-\infty}).
\end{equation}
From \eqref{e-gue140509V} and \eqref{e-gue140509IV}, we obtain 
\begin{equation}\label{e-gue140509VI}
\begin{split}
&\abs{\sideset{}{'}\sum_{\abs{K}=q}\int\pr^\alpha_xS^{I_0,K}_k(p,z)
\tau(z)f(z)\Td\delta_{1,K}(z)e^{-k\phi(z)}\pr^\beta_y(\ol{\Td\delta_{1,J_0}}e^{-k\phi})(y_0)dv_M(z)}\\
&\leq\Td C_{\alpha,\beta}k^{2n-\frac{N}{2}+2\abs{\alpha}+2\abs{\beta}},
\end{split}
\end{equation}
where $\Td C_{\alpha,\beta}>0$ is a constant independent of $k$ and the 
points $p$ and $y_0$. From \eqref{e-gue140509VI}, \eqref{e-gue140509aIV} 
and \eqref{e-gue140509aIII}, we deduce that 
\begin{equation}\label{e-gue140509VII}
\abs{\pr^\alpha_x\pr^\beta_yR^{I_0,J_0}_k(p,y_0)}\leq\widehat C_{\alpha,\beta}
k^{2n-\frac{N}{2}+2\abs{\alpha}+2\abs{\beta}},
\end{equation}
where $\widehat C_{\alpha,\beta}>0$ is a constant independent of $k$ and the points $p$ and $y_0$. 

Now, we assume that $a_{I_0,I_0}(p,p,k)>0$. Take $g_1=(P^{(q)}_{k,k^{-N}}u_k)/|u_k|_{h^k}$, 
where $u_k$ is as in \eqref{e-gue140504VI}. From Theorem~\ref{t-gue140503I} and 
Lemma~\ref{l-gue140504II}, we can check that for every $N>0$, there is $C_N>0$ 
independent of $k$ and the point $p$ such that 
\begin{equation}\label{e-gue140510}
\abs{\sideset{}{'}\sum_{\abs{K}=q}
\Bigr(\pr^\alpha_x(\Td g_{1,I_0}e^{-k\phi})(p)\ol{\Td g_{1,K}(z)}e^{-k\phi(z)}-
\pr^\alpha_xS^{I_0,K}_k(p,z)\Bigr)}\leq C_Nk^{-N},\ \ \forall z\in D
\end{equation}
and 
\begin{equation}\label{e-gue140510I}
\sum^{d_k}_{j=2}\abs{\pr^\alpha_x(\Td g_{j,I_0}e^{-k\phi})(p)}^2\leq C_\alpha k^{3n-N+4\abs{\alpha}},
\end{equation}
where $C_\alpha>0$ is a constant independent of $k$ and the point $p$. Take $\set{\delta_1,\delta_2,\ldots,\delta_{d_k}}$ so that 
\begin{equation}\label{e-gue140510II}
\abs{\pr^\beta_x(\Td\delta_{1,J_0}e^{-k\phi})(y_0)}^2=\pr^\beta_x\pr^\beta_yP^{(q),J_0,J_0}_{k,k^{-N},s}(y_0,y_0).
\end{equation}
From \eqref{e-gue140509aI}, \eqref{e-gue140510}, \eqref{e-gue140510I}, 
\eqref{e-gue140510II} and \eqref{e-gue140504III}, we have 
\begin{equation}\label{e-gue140510III}
\begin{split}
\abs{\pr^\alpha_x\pr^\beta_yR^{I_0,J_0}_k(p,y_0)}&\leq 
C\sqrt{\sum^{d_k}_{j=2}\abs{\pr^\alpha_x(\Td g_{j,I_0}e^{-k\phi})(p)}^2}
\abs{\pr^\beta_y(\ol{\Td\delta_{1,J_0}}e^{-k\phi})(y_0)}+C_Nk^{-N}\\
&\leq C_{\alpha,\beta}k^{2n-\frac{N}{2}+2\abs{\alpha}+2\abs{\beta}}+C_Nk^{-N},
\end{split}
\end{equation}
for every $N>0$, where $C,C_N,C_{\alpha,\beta}>0$ are independent 
of $k$ and the points $p$ and $y_0$. From \eqref{e-gue140510III} and \eqref{e-gue140509VII}, the lemma follows. 
\end{proof}
%--------------
Put 
\begin{equation}\label{e-gue140603}
\Td R_k(x,y)=\int S_k(x,z)\tau(z)f(z)\Bigr(P^{(q)}_{k,k^{-N},s}(z,y)-S_k(z,y)\Bigr)dv_M(z).
\end{equation}
We can repeat the proof of Lemma~\ref{l-gue140509} and conclude:
%--------------
\begin{lem}\label{l-gue140603}
With the assumptions and notations above, for every $N>1$ and $m\in\mathbb N$, there exists
$\Td C_{N,m}>0$ independent of $k$ such that 
\[\abs{\Td R_k(x,y)}_{\cC^m(D\times D)}\leq\Td C_{N,m}k^{2n-\frac{N}{2}+2m}.\]
\end{lem}
%--------------
\begin{lem}\label{l-gue140603I}
We have 
\[\int S_k(x,z)\tau(z)f(z)S_k(z,y)dv_M(z)\equiv e^{ik\Psi(x,y)}b_f(x,y,k)\mod O(k^{-\infty})\]
locally uniformly on $D\times D$, where $b_f(x,y,k)\sim\sum^\infty_{j=0}b_{f,j}(x,y)k^{n-j}$ 
in $S^n(1;D\times D,\Lambda^{0,q}(T^*M))\boxtimes(\Lambda^{0,q}(T^*M))^*)$,  
$b_{f,0}(x,x)=(2\pi)^{-n}f(x)\abs{\det\dot{R}^L(x)}I_{\det\ov{W}^{\,*}}(x)$, for any $x\in D_0$.
\end{lem}

\begin{proof}
From the stationary phase formula of Melin-Sj\"ostrand~\cite{MS74}, there is a 
complex phase function $\Psi_1(x,y)\in\cC^\infty(D\times D)$ with 
$\Psi_1(x,x)=0$, ${\rm Im\,}\Psi_1(x,y)\geq c\abs{x-y}^2$ 
on $D\times D$, for some $c>0$, such that for every bounded function $f\in\cC^\infty(M)$, we have 
\begin{equation}\label{e-gue140603I}
\int S_k(x,z)\tau(z)f(z)S_k(z,y)dv_M(z)\equiv e^{ik\Psi_1(x,y)}\Td b_f(x,y,k)\mod O(k^{-\infty})
\end{equation}
locally uniformly on $D\times D$, where 
\[
\Td b_f(x,y,k)\sim\sum^\infty_{j=0}\Td b_{f,j}(x,y)k^{n-j}\:\: 
\text{in $S^n(1;D\times D,\Lambda^{0,q}(T^*M))\boxtimes(\Lambda^{0,q}(T^*M))^*)$},
\] 
with $\Td b_{f,j}\in\cC^\infty(D\times D,\Lambda^{0,q}(T^*M))\boxtimes(\Lambda^{0,q}(T^*M))^*)$, 
$j\in\N$\,.
Moreover,
for all $x\in D_0$ we have $b_{f,0}(x,x)=(2\pi)^{-n}f(x)\abs{\det\dot{R}^L(x)}I_{\det\ov{W}^{\,*}}(x)$. 
Basically, here we used the fact that composition of complex Fourier 
integral operators is still a complex Fourier integral operator.
Take $f=1$. Fix $D'\Subset\set{\tau=1}$. We claim that 
\begin{equation}\label{e-gue140603II}
\int S_k(x,z)\tau(z)S_k(z,y)dv_M(z)\equiv e^{ik\Psi(x,y)}b(x,y,k)\mod O(k^{-\infty})
\end{equation}
locally uniformly on $D'\times D'$,
where 
\[
b(x,y,k)\sim\sum^\infty_{j=0}b_{j}(x,y)k^{n-j}\:\: \text{in 
$S^n(1;D'\times D',\Lambda^{0,q}(T^*M)\boxtimes (\Lambda^{0,q}(T^*M))^*)$},
\] 
with $b_{j}\in \cC^\infty(D'\times D',\Lambda^{0,q}(T^*M)\boxtimes(\Lambda^{0,q}(T^*M))^*)$, 
$j\in\N$\,. The $S_k$ constructed in~\cite{HM12} is called approximated Szeg\"o kernel. $S_k$ satisfies 
\begin{equation}\label{e-gue160710y}
S_k\circ S_k\equiv S_k\mod O(k^{-\infty})\ \ \mbox{on $D$}
\end{equation}
(see~\cite[Theorems\,3.11--12]{HM12}). Relation \eqref{e-gue160710y} says that 
\begin{equation}\label{e-gue160710yI}
\int S_k(x,z)S_k(z,y)dv_M(z)\equiv S_k(x,y)= e^{ik\Psi(x,y)}b(x,y,k)\mod O(k^{-\infty})
\end{equation}
locally uniformly on $D\times D$. Since $S_k$ is properly supported (see the discussion after \eqref{e-gue160711s} for the meaning of properly supported), the integral \eqref{e-gue160710yI} is well-defined. 
Now, 
\begin{equation}\label{e-gue160710yII}
\int S_k(x,z)\tau(z)S_k(z,y)dv_M(z)=\int S_k(x,z)S_k(z,y)dv_M(z)-\int S_k(x,z)(1-\tau(z))S_k(z,y)dv_M(z).
\end{equation}
Note that $S_k(x,y)=O(k^{-\infty})$ if $\abs{x-y}\geq c$, for some $c>0$. 
From this observation, we conclude that for $(x,y)\in D'\times D'$ (recall that $D'\Subset\set{\tau=1}$)
\begin{equation}\label{e-gue160710yIII}
\int S_k(x,z)(1-\tau(z))S_k(z,y)dv_M(z)=\int_{z\notin\ol{D'}}S_k(x,z)(1-\tau(z))S_k(z,y)dv_M(z)\equiv0\mod O(k^{-\infty})
\end{equation}
locally uniformly on $D'\times D'$. From \eqref{e-gue160710yI}, \eqref{e-gue160710yII} 
and \eqref{e-gue160710yIII}, we get \eqref{e-gue140603II}. 

%\begin{equation}\label{e-gue140603II}
%\int S_k(x,z)\tau(z)S_k(z,y)dv_M(z)\equiv e^{ik\Psi(x,y)}b(x,y,k)\mod O(k^{-\infty})
%\end{equation}
%locally uniformly on $D'\times D'$,
%where $b(x,y,k)\sim\sum^\infty_{j=0}b_{j}(x,y)k^{n-j}$ in 
%$S^n(1;D'\times D',\Lambda^{0,q}_y(T^*M)\boxtimes \Lambda^{0,q}_x(T^*M))$, 
%$b_{j}\in \cC^\infty(D'\times D',\Lambda^{0,q}_y(T^*M)\boxtimes\Lambda^{0,q}_x(T^*M))$, $j=0,1,2,\ldots$\,. 
We claim that 
\begin{equation}\label{e-gue140603III}
\mbox{$\Psi(x,y)-\Psi_1(x,y)$ vanishes to infinite order on ${\rm diag\,}(D'\times D')$}.
\end{equation}
Note that $\Psi(x,x)=\Psi_1(x,x)=0$. We assume that there are 
$\alpha_0, \beta_0\in\mathbb N^{2n}_0$, $\abs{\alpha_0}+\abs{\beta_0}\geq1$ 
and $(x_0,x_0)\in D'\times D'$, such that 
\begin{equation}\label{e-gue140603IV}
\begin{split}
&\pr^{\alpha_0}_x\pr^{\beta_0}_y\Bigr(\Psi_1(x,y)-\Psi(x,y)\Bigr)|_{(x_0,x_0)}\neq0,\\
&\pr^\alpha_x\pr^\beta_y\Bigr(\Psi_1(x,y)-\Psi(x,y)\Bigr)|_{(x_0,x_0)}=0,
\ \ \forall\alpha, \beta\in\mathbb N^{2n}_0, \abs{\alpha}+\abs{\beta}<\abs{\alpha_0}+\abs{\beta_0}.
\end{split}
\end{equation}
From \eqref{e-gue140603I} and \eqref{e-gue140603II}, we have 
\begin{equation}\label{e-gue140603V}
\mbox{$e^{ik(\Psi_1(x,y)-\Psi(x,y))}\Td b_1(x,y,k)-b(x,y,k)=e^{-ik\Psi(x,y)}F_k(x,y)$ on $D'\times D'$},
\end{equation}
where $F_k\equiv0\mod O(k^{-\infty})$ locally uniformly on $D'\times D'$. From \eqref{e-gue140603IV}, it is easy to see that 
\begin{equation}\label{e-gue140603VI}
\begin{split}
&\lim_{k\To\infty}k^{-n-1}\pr^{\alpha_0}_x\pr^{\beta_0}_y
\Bigl(e^{ik(\Psi_1(x,y)-\Psi(x,y))}\Td b_1(x,y,k)-b(x,y,k)\Bigr)|_{(x_0,x_0)}\\
&=i\pr^{\alpha_0}_x\pr^{\beta_0}_y\Bigr(\Psi(x,y)-\Psi_1(x,y)\Bigr)|_{(x_0,x_0)}\Td 
b_{1,0}(x_0,x_0)\neq0.
\end{split}
\end{equation}
It is obviously that
\begin{equation}\label{e-gue140603VII}
\lim_{k\To\infty}k^{-n-1}\pr^{\alpha_0}_x\pr^{\beta_0}_y\Bigr(e^{-ik\Psi(x,y)}F_k(x,y)\Bigr)|_{(x_0,x_0)}=0.
\end{equation}
From \eqref{e-gue140603V}, \eqref{e-gue140603VI} and \eqref{e-gue140603VII}, 
we get a contradiction. The claim follows. Since $\tau$ is arbitrary, $\Psi$ and $\Psi_1$ are independent of 
$\tau$, we conclude that $\Psi_1(x,y)-\Psi(x,y)$ vanishes to infinite order on $D\times D$. 
Thus we can replace $\Psi_1$ by $\Psi$ in \eqref{e-gue140603I}. The lemma follows.
\end{proof}
%---
From Lemma~\ref{l-gue140509}, Lemma~\ref{l-gue140603}, 
Lemma~\ref{l-gue140603I} and Lemma~\ref{l-gue140506I}, we obtain the following. 
%---
\begin{thm}\label{t-gue140725}
With the notations above, let $s$ be local trivializing holomorphic section of $L$ on $D_0\Subset M$.
Assume that $D_0\Subset M(q)$. Then, for every $N>1$, $m\in\mathbb N$, 
there exists $\Td C_{N,m}>0$ independent of $k$ such that 
\[\abs{T^{(q),f}_{k,k^{-N},s}(x,y)-e^{ik\Psi(x,y)}b_f(x,y,k)}_{\cC^m(D_0\times D_0)}
\leq\Td C_{N,m}k^{2n-\frac{N}{2}+2m},\]
where 
\begin{equation}  \label{e-gue140521Im}
\begin{split}
&b_f(x,y,k)\in S^{n}(1;D_0\times D_0,\Lambda^{0,q}(T^*M))\boxtimes(\Lambda^{0,q}(T^*M))^*), \\
&b_f(x,y,k)\sim\sum^\infty_{j=0}b_{f,j}(x,y)k^{n-j}
\text{ in }S^{n}(1;D_0\times D_0,\Lambda^{0,q}(T^*M))\boxtimes(\Lambda^{0,q}(T^*M))^*), \\
%&b_{f,j}(x,y)\in\cC^\infty(D_0\times D_0,T^{*0,q}_yM\otimes T^{*0,q}_xM),\ \ j=0,1,2,\ldots,\\
&b_{f,0}(x,x)=(2\pi)^{-n}f(x)\big|\det\dot{R}^L(x)\big|I_{\det\ov{W}^{\,*}}(x),\ \ \forall x\in D_0,
\end{split}
\end{equation}
and $\Psi$ is as in Theorem~\ref{t-gue140503I}.
\end{thm}
%-----------
\begin{proof}[{Proof of Theorem \ref{t-gue140521}}]
From Theorem~\ref{t-gue140506} and Theorem~\ref{t-gue140725}, Theorem~\ref{t-gue140521} follows. 
\end{proof}
%--------------------
\section{Asymptotics of the composition of Toeplitz operators}\label{s:ctoe}

In this Section, we establish the expansion of the composition of two Toeplitz operators and prove Theorems
\ref{t-gue140729II}, \ref{t-comp},
\ref{t-gue140523IV},
\ref{s1-sing-semi-main} and \ref{sing-main}.

Let $f,g\in\cC^\infty(M)$ be bounded. For $\lambda\geq0$, put 
\[T^{(q),f,g}_{k,\lambda}:=T^{(q),f}_{k,\lambda}\circ T^{(q),g}_{k,\lambda}: 
L^2_{(0,q)}(M,L^k)\to\cE^q_\lambda(M, L^k)\]
and set $T^{(q),f,g}_k:=T^{(q),f,g}_{k,0}$. 
%--------------
\begin{thm}\label{t-gue140729}
Let $s$, $\widehat s$ be local trivializing holomorphic sections of $L$ on $D_0\Subset M$ and $D_1\Subset M$ 
respectively, $\abs{s}^2_{h}=e^{-2\phi}$, $\abs{\widehat s}^2_{h}=e^{-2\widehat\phi}$, where 
$D_0$ and $D_1$ are open sets. Assume that $D_0\Subset M(j)$, $j\neq q$ or 
$D_0\Subset M(q)$ and  $\ol D_0\bigcap\ol D_1=\emptyset$.
%for $k$ large, 
%${\rm dist\,}(x,y)\geq c\frac{\log k}{\sqrt{k}}$, $\forall x\in D_0$, $\forall y\in D_1$, 
%where $c>0$ is a constant independent of $k$. 
Then, for every $m\in\mathbb N$, $N>1$, 
there exists $C_{N,m}>0$ independent of $k$ such that 
\begin{equation}\label{e:3n}
\abs{T^{(q),f,g}_{k,k^{-N},s,\widehat s}(x,y)}_{\cC^m(D_0\times D_1)}
\leq C_{N,m}k^{3n-\frac{N}{2}+2m}.
\end{equation}
\end{thm}
\begin{proof}
%\subsection{The outline of the estimate $T^{(q),f,g}_{k,k^{-N}}(x,y)$}\label{s-gue160713}
The proof is similar to the proof of Theorem~\ref{t-gue140506},
so we will insist here on the appearance of the exponent $3n$ in the power
of $k$ in  \eqref{e:3n}, compared to $2n$ in the previous estimates.
%It should be noticed that in Theorem~\ref{t-gue140729}, 
%we get the power $3n$ and it was $2n$ in the previous equations. 
%We pause and explain how to estimate the kernel $T^{(q),f,g}_{k,k^{-N}}(x,y)$ 
%and we will see why we get $3n$ here. 
The argument holds for any complex manifold, 
\emph{not necessarily compact}. For simplicity, we only consider $q=0$. Let $s$, $\widehat s$ 
be local trivializing holomorphic sections of $L$ on $\Td D_0\Subset M$ and $\Td D_1\Subset M$ 
respectively, $\abs{s}^2_{h}=e^{-2\phi}$, $\abs{\widehat s}^2_{h}=e^{-2\widehat\phi}$, 
where 
$\Td D_0$ and $\Td D_1$ are open sets. Fix $D_0\Subset\Td D_0$, 
$D_1\Subset\Td D_1$, $D_0$ and $D_1$ are open sets. 
Let $\tau\in\cC^\infty_0(\Td D_0)$ and $\tau=1$ on $D_0$. 
We will first show that how to estimate the kernel of $T^{(0),(1-\tau)f,g}_{k,k^{-N}}(x,y)$ on $D_0\times D_1$. 
Take 
\[\set{\alpha_1(x),\alpha_2(x),\ldots,\alpha_{d_k}(x)},\ \ \set{\delta_1(x),\delta_2(x),\ldots,\delta_{d_k}(x)}\]
be orthonormal frames for $\cE^0_{k^{-N}}(M,L^k)$, where $d_k\in\mathbb N\cup\set{\infty}$.
On $\Td D_0$, we write $\alpha_j(x)=s^k(x)\Td \alpha_j(x)$, $j=1,\ldots,d_k$. On $\Td D_1$, we write 
$\delta_j(x)=\hat s^k(x)\hat\delta_j(x)$, $ j=1,\ldots,d_k$. For every $y\in\Td D_1$, put
\begin{equation}\label{e-gue160711g}
\begin{split}
T^{(0),g}_{k,k^{-N}}(x,y)(\hat s^{-k}e^{-k\widehat\phi})(y)
:=\sum^{d_k}_{j,\ell=1}\alpha_j(x)(\,g\delta_\ell\,|\,\alpha_j)_k\ol{\widehat\delta_\ell}(y)
e^{-k\widehat\phi}(y).
\end{split}
\end{equation}
Since $\sum^{d_k}_{j=1}\abs{\alpha_j(x)}^2_{h^k}$ and 
$\sum^{d_k}_{j=1}\abs{\delta_j(x)}^2_{h^k}$ converge
locally uniformly in $C^\infty$ topology, for fixed $y$, 
$T^{(0),g}_{k,k^{-N}}(\cdot,y)(\hat s^{-k}e^{-k\widehat\phi})(y)$ 
is a smooth section of $L^k$. It is easy to see that for every $(x,y)\in D_0\times D_1$, 
\begin{equation}\label{e-gue160711w}
T^{(0),(1-\tau)f,g}_{k,k^{-N},s,\hat s}(x,y)=
\sum^{d_k}_{j=1}e^{-k\phi(x)}\Td\alpha_j(x)\Bigr(((1-\tau)f)(\cdot)T^{(0),g}_{k,k^{-N}})
(\cdot,y)(\hat s^{-k}e^{-k\hat\phi})(y)\,|\,\alpha_j(\cdot)\,\Bigr)_k.
\end{equation}
When $M$ is compact, $d_k$ is finite and $d_k\approx k^n$ and it is easy 
to estimate \eqref{e-gue160711w}. When $M$ is \emph{non-compact}, 
$d_k$ could be infinite, so to estimate \eqref{e-gue160711w} 
we need a more detailed analysis. Now, we fix $p\in D_0$. 
From Theorem~\ref{t-gue140503I} and Lemma~\ref{l-gue140504II}, 
we can find $v_k\in\cE^0_{k^{-N}}(M,L^k)$ with $\norm{v_k}_k=1$, 
\begin{equation}\label{e-gue160711wI}
\int_{M\setminus D_0}\abs{v_k}^2_{h^k}dv_M=O(k^{-\infty})
\end{equation}
and 
\begin{equation}\label{e-gue160711wII}
\abs{P^{(0)}_{k,k^{-N},s}(p,p)-\abs{v_k(p)}^2_{h^k}}\lesssim k^{3n-N}.
\end{equation}
We take $\alpha_1=v_k$ and obtain from \eqref{e-gue160711wII} that
\begin{equation}\label{e-gue160711wIII}
\sum^{d_k}_{j=2}e^{-2k\phi(p)}\abs{\Td\alpha_j(p)}^2\lesssim k^{3n-N}. 
\end{equation}
Now, 
\begin{equation}\label{e-gue160711wV}
\begin{split}
&\abs{e^{-k\phi(p)}\Td\alpha_1(p)\Bigr(((1-\tau)f)(\cdot)T^{(0),g}_{k,k^{-N}})
(\cdot,y)(\hat s^{-k}e^{-k\hat\phi})(y)\,|\,\alpha_1(\cdot)\,\Bigr)_k}\\
&\leq\abs{e^{-k\phi(x)}\Td\alpha_1(p)}\norm{T^{(0),g}_{k,k^{-N}}(\cdot,y)
(\hat s^{-k}e^{-k\widehat\phi})(y)}_{k}\norm{(1-\tau)f\alpha_1}_k.
\end{split}
\end{equation}
We claim that 
\begin{equation}\label{e-gue160711gI}
\norm{T^{(0),g}_{k,k^{-N}}(\cdot,y)(\hat s^{-k}e^{-k\widehat\phi})(y)}^2_{k}
\lesssim k^n\ \ \mbox{locally uniformly on $y\in D_1$}. 
\end{equation}
Fix $y_0\in D_1$. We take $\set{\delta_1(x),\delta_2(x),\ldots,\delta_{d_k}(x)}$ 
so that $\hat\delta_1(y_0)\neq0$, $\hat\delta_j(y_0)=0$, $j=2,3,\ldots,d_k$. Then, 
\begin{equation}\label{e-gue160711gII}
T^{(0),g}_{k,k^{-N}}(x,y_0)(\hat s^{-k}e^{-k\widehat\phi})(y_0)
=\sum^{d_k}_{j=1}\alpha_j(x)(\,g\delta_1\,|\,\alpha_j)_k\ol{\widehat\delta_1}(y_0)e^{-k\hat\phi(y_0)}.
\end{equation}
From \eqref{e-gue160711gII}, we can check that 
\begin{equation}\label{e-gue160711gIII}
\begin{split}
\norm{T^{(0),g}_{k,k^{-N}}(\cdot,y_0)}^2_k&=
\sum^{d_k}_{j=1}\norm{\alpha_j}^2_k
\abs{(\,g\delta_1\,|\,\alpha_j)_k}^2\abs{\ol{\widehat\delta_1}(y_0)}^2e^{-2k\widehat\phi(y_0)}\\
&=\norm{P^{(0)}_{k,k^{-N}}(g\delta_1)}^2_ke^{-2k\hat\phi(y_0)}\abs{\ol{\delta_1}(y_0)}^2.
\end{split}
\end{equation}
Since $g$ is a bounded function, $\|P^{(0)}_{k,k^{-N}}(g\delta_1)\|^2_k\leq C$, 
for some constant $C>0$ independent of $k$. Moreover, 
$e^{-2k\hat\phi(y_0)}\abs{\ol{\delta_1}(y_0)}^2\lesssim k^n$ 
locally uniformly on $D_1$. From this observation and \eqref{e-gue160711gIII}, 
the estimate \eqref{e-gue160711gI} follows. 
Relation \eqref{e-gue160711wI} yields 
\begin{equation}\label{e-gue160711z}
\norm{(1-\tau)f\alpha_1}_k=O(k^{-\infty}). 
\end{equation}
From \eqref{e-gue160711wV}, \eqref{e-gue160711gI} , \eqref{e-gue160711z}  and since
$\abs{e^{-k\phi(p)}\Td\alpha_1(p)}\lesssim k^{\frac{n}{2}}$, we conclude that 
\begin{equation}\label{e-gue160711zI}
\abs{e^{-k\phi(p)}\Td\alpha_1(p)\Bigr(((1-\tau)f)(\cdot)T^{(0),g}_{k,k^{-N}})(\cdot,y)(\hat s^{-k}e^{-k\hat\phi})(y)\,|\,\alpha_1(\cdot)\,\Bigr)_k}=O(k^{-\infty}).
\end{equation}
%-------------
From \eqref{e-gue160711wIII} and \eqref{e-gue160711gI}, we have
\begin{equation}\label{e-gue160713p}
\begin{split}
&\abs{\sum^{d_k}_{j=2}e^{-k\phi(p)}\Td\alpha_j(p)\Bigr(((1-\tau)f)(\cdot)
T^{(0),g}_{k,k^{-N}})(\cdot,y)(\hat s^{-k}e^{-k\hat\phi})(y)\,|\,\alpha_j(\cdot)\,\Bigr)_k}\\
&\leq\sqrt{\sum^{d_k}_{j=2}e^{-2k\phi(p)}\abs{\Td\alpha_j(p)}^2}
\sqrt{\sum^{d_k}_{j=2}\abs{\Bigr(((1-\tau)f)(\cdot)T^{(0),g}_{k,k^{-N}})
(\cdot,y)(\hat s^{-k}e^{-k\hat\phi})(y)\,|\,\alpha_j(\cdot)\,\Bigr)_k}^2}\\
&\leq\sqrt{\sum^{d_k}_{j=2}e^{-2k\phi(p)}\abs{\Td\alpha_j(p)}^2}
\sqrt{\sum^{d_k}_{j=1}\abs{\Bigr(((1-\tau)f)(\cdot)T^{(0),g}_{k,k^{-N}})
(\cdot,y)(\hat s^{-k}e^{-k\hat\phi})(y)\,|\,\alpha_j(\cdot)\,\Bigr)_k}^2}\\
&=\sqrt{\sum^{d_k}_{j=2}e^{-2k\phi(p)}\abs{\Td\alpha_j(p)}^2}
\norm{P^{(0)}_{k,k^{-N}}(((1-\tau)f)(\cdot)T^{(0),g}_{k,k^{-N}})
(\cdot,y)(\hat s^{-k}e^{-k\hat\phi})(y))}_k\\
&\leq\sqrt{\sum^{d_k}_{j=2}e^{-2k\phi(p)}\abs{\Td\alpha_j(p)}^2}
\norm{((1-\tau)f)(\cdot)T^{(0),g}_{k,k^{-N}})(\cdot,y)(\hat s^{-k}e^{-k\hat\phi})(y)}_k\\
&\lesssim k^{2n-N}.
\end{split}
\end{equation}
%From \eqref{e-gue160711wIII} and \eqref{e-gue160711gI}, it is easy to see that 
%\begin{equation}\label{e-gue160711zII}
%\abs{e^{-k\phi(p)}\Td\alpha_2(p)\Bigr(((1-\tau)f)(\cdot)T^{(0),g}_{k,k^{-N}})(\cdot,y)(\hat s^{-k}e^{-k\hat\phi})(y)\,|\,\alpha_1(\cdot)\,\Bigr)_k}\lesssim k^{2n-N}. 
%\end{equation}
Note that here we still get the exponent $2n-N$. 
%From \eqref{e-gue160711zI}, \eqref{e-gue160711zII} and \eqref{e-gue160711wIV}, we get that 
From \eqref{e-gue160711w}, \eqref{e-gue160711zI} and \eqref{e-gue160713p}, we get 
\begin{equation}\label{e-gue160711zIII}
\abs{T^{(0),(1-\tau)f,g}_{k,k^{-N},s,\hat s}(x,y)}\lesssim k^{2n-N}\ \ \mbox{locally uniformly on $D_0\times D_1$}. 
\end{equation}
Thus, to estimate $T^{(0),f,g}_{k,k^{-N},s,\hat s}(x,y)$, we only need to estimate $T^{(0),\tau f,g}_{k,k^{-N},s,\hat s}(x,y)$. Let $\hat\tau\in\cC^\infty_0(\Td D_1)$ and $\hat\tau=1$ on $D_1$. We can repeat the procedure above and conclude that 
\begin{equation}\label{e-gue160711zIV}
\abs{T^{(0),\tau f,(1-\hat\tau)g}_{k,k^{-N},s,\hat s}(x,y)}\lesssim k^{2n-N}\ \ \mbox{locally uniformly on $D_0\times D_1$}. 
\end{equation}
Thus, to estimate $T^{(0),f,g}_{k,k^{-N},s,\hat s}(x,y)$, we only need to estimate $T^{(0),\tau f,\hat\tau g}_{k,k^{-N},s,\hat s}(x,y)$. We now explain how to estimate $T^{(0),\tau f,\hat\tau g}_{k,k^{-N},s,\hat s}(x,y)$. Take $\tau_1(x)\in\cC^\infty_0(\Td D_0)$, $\tau_1=1$ on ${\rm Supp\,}\tau$. We have 
\begin{equation}\label{e-gue160711zV}
\begin{split}
&T^{(0),\tau f,\hat\tau g}_{k,k^{-N}}=\Td T^{(0),\tau f,\hat\tau g}_{k,k^{-N}}+\hat T^{(0),\tau f,\hat\tau g}_{k,k^{-N}},\\
&\Td T^{(0),\tau f,\hat\tau g}_{k,k^{-N}}=P^{(0)}_{k,k^{-N}}\tau fP^{(0)}_{k,k^{-N}}\tau_1P^{(0)}_{k,k^{-N}}\hat\tau gP^{(0)}_{k,k^{-N}}=T^{(0),\tau f}_{k,k^{-N}}\tau_1 T^{(0),\hat\tau g}_{k,k^{-N}},\\
&\hat T^{(0),\tau f,\hat\tau g}_{k,k^{-N}}=P^{(0)}_{k,k^{-N}}\tau fP^{(0)}_{k,k^{-N}}(1-\tau_1)P^{(0)}_{k,k^{-N}}\hat\tau gP^{(0)}_{k,k^{-N}}=T^{(0),\tau f}_{k,k^{-N}}(1-\tau_1)T^{(0),\hat\tau g}_{k,k^{-N}}.
%=P^{(0)}_{k,k^{-N}}\tau fT^{(0),1-\tau_1}_{k,k^{-N}}\hat\tau gP^{(0)}_{k,k^{-N}}.
\end{split}
\end{equation}
The estimate of $\Td T^{(0),\tau f,\hat\tau g}_{k,k^{-N},}$ is as compact case since 
\[\Td T^{(0),\tau f,\hat\tau g}_{k,k^{-N},s,\hat s}(x,y)=\int T^{(0),\tau f}_{k,k^{-N},s,s}(x,z)\tau_1(z) T^{(0),\hat\tau g}_{k,k^{-N},s,\hat s}(z,y)dv_M(z)\]
and the integral is over some compact set of $M$. We only need to show how to estimate $\hat T^{(0),\tau f,\hat\tau g}_{k,k^{-N},s,\hat s}(x,y)$. Note that
\begin{equation}\label{e-gue160713pI}
\hat T^{(0),\tau f,\hat\tau g}_{k,k^{-N}}=P^{(0)}_{k,k^{-N}}\tau fT^{(0),1-\tau_1}_{k,k^{-N}}\hat\tau gP^{(0)}_{k,k^{-N}}.
\end{equation}
From \eqref{e-gue160713pI}, it is easy to see that 
\begin{equation}\label{e-gue160711zVI}
\begin{split}
&\hat T^{(0),\tau f,\hat\tau g}_{k,k^{-N},s,\hat s}(x,y)=\sum^{d_k}_{j,\ell=1}\Td\alpha_j(x)e^{-k\phi(x)}\ol{\hat\delta_\ell(y)}e^{-k\hat\phi(y)}\\
&\times\int T^{(0),1-\tau_1}_{k,k^{-N},s,\hat s}(z,u)\hat\tau(u)g(u)\hat\delta_\ell(u)\tau(z)f(z)
\ol{\Td\alpha_j(z)}e^{-k\hat\phi(u)-k\phi(z)}dv_M(u)dv_M(z).
\end{split}
\end{equation}
Now, fix $x_0\in D_0$ and $y_0\in D_1$. We take $\set{\alpha_1(x),\alpha_2(x),\ldots,\alpha_{d_k}(x)}$ and $\set{\delta_1(x),\delta_2(x),\ldots,\delta_{d_k}(x)}$ so that 
$\Td\alpha_1(x_0)\neq0$, $\Td\alpha_j(x_0)=0$, $j=2,3,\ldots,d_k$, $\hat\delta_1(y_0)\neq0$, $\hat\delta_j(y_0)=0$, $j=2,3,\ldots,d_k$. Thus, 
\begin{equation}\label{e-gue160711zVII}
\begin{split}
&\hat T^{(0),\tau f,\hat\tau g}_{k,k^{-N},s,\hat s}(x_0,y_0)=\Td\alpha_1(x_0)e^{-k\phi(x_0)}\ol{\hat\delta_1(y_0)}e^{-k\hat\phi(y_0)}\\
&\times\int T^{(0),1-\tau_1}_{k,k^{-N},s,\hat s}(z,u)\hat\tau(u)g(u)\hat\delta_1(u)\tau(z)f(z)
\ol{\Td\alpha_1(z)}e^{-k\hat\phi(u)-k\phi(z)}dv_M(u)dv_M(z).
\end{split}
\end{equation}
From Lemma~\ref{l-gue140506I}, we see that 
$|T^{(0),1-\tau_1}_{k,k^{-N},s,\hat s}(z,u)|\lesssim k^{2n-N}$ 
locally uniformly on $\Td D_0\times\Td D_1$. From this observation and since 
$|\Td\alpha_1(x_0)e^{-k\phi(x_0)}\ol{\Td\delta_1(y_0)}e^{-k\hat\phi(y_0)}|\lesssim k^n$, 
we deduce that 
\[\abs{\hat T^{(0),\tau f,\hat\tau g}_{k,k^{-N},s,\hat s}(x_0,y_0)}\lesssim k^{3n-N}.\] 
Here we get the power $3n$. 
\end{proof}

%---------
We have moreover: 
%---------
\begin{thm}\label{t-gue140729I}
With the notations above, let $s$ be local trivializing holomorphic section of $L$ on $D_0\Subset M$.
Assume that $D_0\Subset M(q)$. Then, for every $N>1$, $m\in\mathbb N$, 
there exists $\Td C_{N,m}>0$ independent of $k$ such that 
\[\abs{T^{(q),f,g}_{k,k^{-N},s}(x,y)-
e^{ik\Psi(x,y)}b_{f,g}(x,y,k)}_{\cC^m(D_0\times D_0)}\leq\Td C_{N,m}k^{3n-\frac{N}{2}+2m},\]
where 
\begin{equation}  \label{e-gue140729III}
\begin{split}
&b_{f,g}(x,y,k)\in S^{n}(1;D_0\times D_0,\Lambda^{0,q}(T^*M))\boxtimes(\Lambda^{0,q}(T^*M))^*), \\
&b_{f,g}(x,y,k)\sim\sum^\infty_{j=0}b_{f,g,j}(x,y)k^{n-j}
\text{ in }S^{n}(1;D_0\times D_0,\Lambda^{0,q}(T^*M))\boxtimes(\Lambda^{0,q}(T^*M))^*), \\
%&b_{f,g,j}(x,y)\in\cC^\infty(D_0\times D_0,T^{*0,q}_yM\otimes T^{*0,q}_xM),\ \ j=0,1,2,\ldots,\\
&b_{f,g,0}(x,x)=(2\pi)^{-n}f(x)g(x)\big|\det\dot{R}^L(x)\big|I_{\det\ov{W}^{\,*}}(x),\ \ \forall x\in D_0,
\end{split}
\end{equation}
and $\Psi$ is as in Theorem~\ref{t-gue140503I}.
\end{thm}

\begin{proof}
The proof of this theorem is similar to the proof of Theorem~\ref{t-gue140725}. 
We only give the outline of the proof and for simplicity we consider only $q=0$. 
Fix $D_0\subset\Td D_0\Subset M(q)$ and take $\tau(x)\in\cC^\infty_0(\Td D_0)$, 
$\tau=1$ on $D_0$. We may assume that the section $s$ defined on $\Td D_0$. 
We can repeat the proof of Lemma~\ref{l-gue140506I} with minor changes 
and conclude that for every $N>1$ and $m\in\mathbb N$, 
there is $C_{N,m}>0$ independent of $k$ such that 
\begin{equation}\label{e-gue140729IV}
\begin{split}
&\abs{T^{(0),(1-\tau)f,g}_{k,k^{-N},s}(x,y)}_{\cC^m(D_0\times D_0)}\leq C_{N,m}k^{3n-\frac{N}{2}+2m},\\
&\abs{T^{(0),\tau f,(1-\tau)g}_{k,k^{-N},s}(x,y)}_{\cC^m(D_0\times D_0)}\leq C_{N,m}k^{3n-\frac{N}{2}+2m}.
\end{split}
\end{equation}
From \eqref{e-gue140729IV}, we only need to consider $T^{(0),\tau f,\tau g}_{k,k^{-N},s}$. Take $\tau_1(x)\in\cC^\infty_0(\Td D_0)$, $\tau_1=1$ on ${\rm Supp\,}\tau$. We have 
\begin{equation}\label{e-gue140729V}
\begin{split}
&T^{(0),\tau f,\tau g}_{k,k^{-N}}=\Td T^{(0),\tau f,\tau g}_{k,k^{-N}}+\widehat T^{(0),\tau f,\tau g}_{k,k^{-N}},\\
&\Td T^{(0),\tau f,\tau g}_{k,k^{-N}}=P^{(0)}_{k,k^{-N}}\tau fP^{(0)}_{k,k^{-N}}\tau_1P^{(0)}_{k,k^{-N}}\tau gP^{(0)}_{k,k^{-N}}=T^{(0),\tau f}_{k,k^{-N}}\tau_1 T^{(0),\tau g}_{k,k^{-N}},\\
&\widehat T^{(0),\tau f,\tau g}_{k,k^{-N}}=P^{(0)}_{k,k^{-N}}\tau fP^{(0)}_{k,k^{-N}}(1-\tau_1)P^{(0)}_{k,k^{-N}}\tau gP^{(0)}_{k,k^{-N}}=T^{(0),\tau f}_{k,k^{-N}}(1-\tau_1)T^{(0),\tau g}_{k,k^{-N}}.
\end{split}
\end{equation}
Let $\Td T^{(0),\tau f,\tau g}_{k,k^{-N,s}}(x,y), \widehat T^{(0),\tau f,\tau g}_{k,k^{-N},s}(x,y)\in\cC^\infty(\Td D_0\times\Td D_0,\Lambda^{0,q}(T^*M))\boxtimes(\Lambda^{0,q}(T^*M))^*)$ be the distribution kernels of $s^{-k}e^{-k\phi}\Td T^{(0),\tau f,\tau g}_{k,k^{-N}}s^ke^{k\phi}$ and $s^{-k}e^{-k\phi}\widehat T^{(0),\tau f,\tau g}_{k,k^{-N}}s^ke^{k\phi}$ respectively. We have 
\begin{equation}\label{e-gue140729VI}
T^{(0),\tau f,\tau g}_{k,k^{-N},s}(x,y)=\Td T^{(0),\tau f,\tau g}_{k,k^{-N},s}(x,y)+\widehat T^{(0),\tau f,\tau g}_{k,k^{-N},s}(x,y).
\end{equation}
We first consider $\widehat T^{(0),\tau f,\tau g}_{k,k^{-N},s}(x,y)$. Take 
\[\set{\alpha_1(x),\alpha_2(x),\ldots,\alpha_{d_k}(x)},\ \ \set{\delta_1(x),\delta_2(x),\ldots,\delta_{d_k}(x)}\] 
be orthonormal frames for $\cE^0_{k^{-N}}(M,L^k)$, where $d_k\in\mathbb N\cup\set{\infty}$. On $\Td D_0$, we write 
\[
\alpha_j(x)=s^k(x)\Td \alpha_j(x),\:\:
\delta_j(x)=s^k(x)\Td\delta_j(x),\ \ j=1,\ldots,d_k.
\]
It is straightforward to check that 
%\begin{equation}\label{e-gue140729VII}
%\begin{split}
%&\widehat T^{(0),\tau f,\tau g}_{k,k^{-N},s}(x,y)
%=\sum^{d_k}_{j,s=1}\Td\alpha_j(x)e^{-k\phi(x)}\Bigr(\int T^{(0),1-\tau_1}_{k,k^{-N},s}(z,u)\tau(u)g(u)\Td\delta_s(u)\\
%&\quad\quad\quad\quad\quad\times\tau(z)f(z)\ol{\Td\alpha_j(z)}e^{-k\phi(u)-k\phi(z)}dv_M(u)dv_M(z)\Bigr)
%\ol{\Td\delta_s(y)}e^{-k\phi(y)}.
%\end{split}
%\end{equation}
\begin{equation}\label{e-gue140729VII}
\begin{split}
&\widehat T^{(0),\tau f,\tau g}_{k,k^{-N},s}(x,y)=\sum^{d_k}_{j,s=1}\Td\alpha_j(x)e^{-k\phi(x)}\ol{\Td\delta_s(y)}e^{-k\phi(y)}\\
&\times\int T^{(0),1-\tau_1}_{k,k^{-N},s}(z,u)\tau(u)g(u)\Td\delta_s(u)\tau(z)f(z)
\ol{\Td\alpha_j(z)}e^{-k\phi(u)-k\phi(z)}dv_M(u)dv_M(z).
\end{split}
\end{equation}
From Lemma~\ref{l-gue140506I}, \eqref{e-gue140504III} and \eqref{e-gue140729VII}, 
it is not difficult to see that for every $N>1$ and $m\in\mathbb N$, there exists $C_{N,m}>0$ independent of $k$ such that 
\begin{equation}\label{e-gue140729VIII}
\abs{\widehat T^{(0),\tau f,\tau g}_{k,k^{-N},s}(x,y)}_{\cC^m(D_0\times D_0)}\leq C_{N,m}k^{3n-\frac{N}{2}+2m}.
\end{equation}
We now consider $\Td T^{(0),\tau f,\tau g}_{k,k^{-N},s}(x,y)$. We have
\begin{equation}\label{e-gue140729a}
\Td T^{(0),\tau f,\tau g}_{k,k^{-N},s}(x,y)=
\int T^{(0),\tau f}_{k,k^{-N},s}(x,z)\tau_1(z)T^{(0),\tau g}_{k,k^{-N},s}(z,y)dv_M(z).
\end{equation}
Put 
\[S_{k,f}(x,y)=e^{ik\Psi(x,y)}b_f(x,y,k),\ \ S_{k,g}(x,y)=e^{ik\Psi(x,y)}b_g(x,y,k),\] 
where $\Psi(x,y)$ is as in Theorem~\ref{t-gue140503I} and 
$b_f(x,y,k), b_g(x,y,k)\in S^{n}(1;\Td D_0\times\Td D_0)$ are as in Theorem~\ref{t-gue140725}. Put 
\[
A_k(x,y)
=\int T^{(0),\tau f}_{k,k^{-N},s}(x,z)\tau_1(z)T^{(0),\tau g}_{k,k^{-N},s}(z,y)dv_M(z)-
\int S_{k,f}(x,z)\tau_1(z)S_{k,g}(z,y)dv_M(z).
\]
From Theorem~\ref{t-gue140725}, it is straightforward to see that
for every $N>1$ and $m\in\mathbb N$, there exists $C_{N,m}>0$ independent of $k$ such that 
\begin{equation}\label{e-gue140729aI}
\abs{A_k(x,y)}_{\cC^m(D_0\times D_0)}\leq C_{N,m}k^{3n-\frac{N}{2}+2m}.
\end{equation}
%As in Lemma~\ref{l-gue140603I}, we can show that
We claim that
\begin{equation}\label{e-gue140729aII}
\int S_{k,f}(x,z)\tau_1(z)S_{k,g}(z,y)dv_M(z)\equiv e^{ik\Psi(x,y)}b_{f,g}(x,y,k)\mod O(k^{-\infty})
\end{equation}
locally uniformly on $\Td D_0\times\Td D_0$, where 
$b_{f,g}(x,y,k)\in S^{n}(1;\Td D_0\times\Td D_0)$,
\[b_{f,g}(x,y,k)\sim\sum^\infty_{j=0}b_{f,g,j}(x,y)k^{n-j}\text{ in }S^{n}(1;\Td D_0\times\Td D_0),\]
with $b_{f,g,j}\in\cC^\infty(\Td D_0\times\Td D_0)$, and 
$b_{f,g,0}(x,x)=(2\pi)^{-n}\tau(x)^2\tau_1(x)f(x)g(x)\big|\det\dot{R}^L(x)\big|$, $x\in\Td D_0$.

%Let's explain why \eqref{e-gue140729aII} holds. 
We use now the theory of complex Fourier integral operator, 
in particular the fact that composition of complex Fourier integral operators 
is still a complex Fourier integral operator. 
Indeed, the complex stationary phase formula of Melin-Sj\"ostrand~\cite{MS74} 
tells us that there is a complex phase $\Psi_1(x,y)\in\cC^\infty(\Td D_0\times\Td D_0)$ 
with ${\rm Im\,}\Psi_1(x,y)\approx\abs{x-y}^2$, 
such that for any $A(x,y)=e^{ik\Psi(x,y)}a(x,y,k)$, $C(x,y)=e^{ik\Psi(x,y)}c(x,y,k)$, 
where $a(x,y,k), c(x,y,k)\in S^{n}(1;\Td D_0\times\Td D_0)$, 
and every $\chi\in\cC^\infty_0(\Td D_0)$, we have 
\begin{equation}\label{e-gue160711}
\int A(x,z)\chi(z)B(z,y)dv_M(z)\equiv e^{ik\Psi_1(x,y)}h(x,y,k)\mod O(k^{-\infty})
\end{equation}
locally uniformly on $\Td D_0\times\Td D_0$, where 
$h(x,y,k)\in S^{n}(1;\Td D_0\times\Td D_0)$,
\[h(x,y,k)\sim\sum^\infty_{j=0}h_j(x,y)k^{n-j}\text{ in }S^{n}(1;\Td D_0\times\Td D_0)\] 
and $h_0(x,x)=(2\pi)^{-n}\chi(x)a_0(x,x)c_0(x,x)$, $x\in\Td D_0$, where $a_0$ and $c_0$ denote the leading terms of $a(x,y,k)$ and $c(x,y,k)$ respectively. In the proof of Lemma~\ref{l-gue140603I}, we proved that $\Psi(x,y)-\Psi_1(x,y)$ vanishes to infinite order on $x=y$
(see \eqref{e-gue140603III}). Thus, we can replace $\Psi_1$ in \eqref{e-gue160711} by $\Psi$ and we get \eqref{e-gue140729aII}.

From \eqref{e-gue140729aII}, \eqref{e-gue140729aI}, \eqref{e-gue140729VIII}, 
\eqref{e-gue140729VI} and \eqref{e-gue140729IV}, the theorem follows.
\end{proof}
%--------------------
\begin{proof}[Proof of Theorem~\ref{t-gue140729II}] 
Theorems~\ref{t-gue140729} and \ref{t-gue140729I} yield immediately Theorem~\ref{t-gue140729II}. 
\end{proof}
%--------------------
\begin{proof}[Proof of Theorem~\ref{t-gue140523IV}] 
This follows by using the asymptotics of the
Bergman kernel proved in \cite[Theorem\,1.6]{HM12} in the case of an $O(k^{-N})$ small spectral gap
and adapting the proofs of Theorems~\ref{t-gue140521} and \ref{t-gue140729II} to the current situation. 
\end{proof}
%--------------------
\begin{proof}[Proof of Theorem~\ref{s1-sing-semi-main}] 
By \cite[Theorem 8.2]{HM12} we know that $\Box^{(0)}_k$ has an $O(k^{-N})$ 
small spectral gap on every $D\Subset M'\cap M(0)$. 
This observation and Theorem~\ref{t-gue140523IV} yield 
Theorem~\ref{s1-sing-semi-main}. 
\end{proof}
%--------------------
\begin{proof}[Proof of Theorem~\ref{sing-main}]
$M\setminus\Sigma$ is a non-compact complex manifold. Let $\Box^{(0)}_k$ be 
the Gaffney extension of Kodaira Laplacian on $M\setminus\Sigma$ and let 
$P^{(0)}_{k,M\setminus\Sigma}$ be the associated Bergman projection. 
By a result of Skoda (see~\cite[Lemma 7.2]{HM12}), we know that 
\begin{equation}\label{e-gue140729}
P^{(0)}_{k,\cali{I}}=P^{(0)}_{k,M\setminus\Sigma}\ \ \mbox{on $M\setminus\Sigma$}.
\end{equation}
Moreover, we know that $\Box^{(0)}_k$ has $O(k^{-N})$ small spectral gap on every 
$D\Subset M\setminus\Sigma$ (see~\cite[Theorem 9.1]{HM12}). 
This observation, \eqref{e-gue140729} and Theorem~\ref{t-gue140523IV} imply
Theorem~\ref{sing-main}.
\end{proof}

In the following, we will prove Theorem~\ref{t-comp}. 
Fix $N>1$. Let $f, g\in\cC^\infty_0(D)$, $D\Subset M(0)$. For simplicity, we may assume that $L|_D$ 
is trivial and let $s$ be a local trivializing holomorphic section of $L$ on $D$, $\abs{s}^2_{h}=e^{-2\phi}$. 
Take $\tau\in C^\infty_0(D)$ with $\tau=1$ on ${\rm Supp\,}f\cup{\rm Supp\,}g$. Put 
\begin{equation}\label{e-gue140922}
R_k=T^{(0),f}_{k,k^{-N}}T^{(0),g}_{k,k^{-N}}-\tau T^{(0),f}_{k,k^{-N}}T^{(0),g}_{k,k^{-N}}\tau.
\end{equation}
We can repeat the proof of Lemma~\ref{l-gue140506I} with minor changes and obtain: 
%-----
\begin{lem}\label{l-gue140922}
Let $s_1$, $s_2$ be local trivializing holomorphic sections of $L$ on $D_1\Subset M$ and $D_2\Subset M$ respectively, 
where $D_1$ and $D_2$ are open sets. Then, for every $m\in\mathbb N$, there exists 
$C_{m}>0$ independent of $k$ such that 
\[\abs{R_{k,s_1,s_2}(x,y)}_{\cC^m(D_1\times D_2)}\leq C_{m}k^{3n-\frac{N}{2}+2m},\]
where $R_{k,s_1,s_2}(x,y)$ denotes the distribution kernel of 
$R_{k,s_1,s_2}:=s^{-k}_1e^{-k\phi_1}R_ks^{k}_2e^{k\phi_2}$.

In particular, $T^{(0),f}_{k,k^{-N}}T^{(0),g}_{k,k^{-N}}-
\tau T^{(0),f}_{k,k^{-N}}T^{(0),g}_{k,k^{-N}}\tau=
\mO(k^{3n-\frac{N}{2}})$ locally in the $L^2$ operator norm. 
\end{lem} 

Let $b_{f,g}(x,y,k)\in S^n(1;D\times D)$ be as in Theorem~\ref{t-gue140729II}. 
Then
\[
b_{f,g}(x,y,k)\sim\sum\limits^\infty_{j=0}b_{f,g,j}(x,y)k^{n-j}\:\: \text{in $S^n(1;D\times D)$.}
\]
Since $f, g\in\cC^\infty_0(D)$, we can take 
$b_{f,g}(x,y,k), b_{f,g,j}(x,y)\in\cC^\infty_0(D\times D)$, $j\in\N$. 
Note that $b_{f,g}(x,y,k)$ and $b_{f,g,j}(x,y)$ have uniquely determined Taylor expansion at $x=y$. Consider 
\[\begin{split}
B_k:L^2(M,L^k)&\To L^2(M,L^k)\\
u&\mapsto s^ke^{k\phi}\tau\int e^{ik\Psi(x,y)}b_{f,g}(x,y,k)s^{-k}e^{-k\phi(y)}\tau(y)u(y)dv_M(y).
\end{split}\]
In view of Theorem~\ref{t-gue140729II} and Lemma~\ref{l-gue140922}, we see that 
\begin{equation}\label{e-gue140922I}
\mbox{$B_k-T^{(0),f}_{k,k^{-N}}T^{(0),g}_{k,k^{-N}}=\mO(k^{3n-\frac{N}{2}})$ locally in the $L^2$ operator norm}.
\end{equation}
%-----------------------
\begin{lem}\label{l-gue140922I}
For any $p\in\N$ there exist $C_p(f,g)\in\cC^\infty_0(D)$ such that 
\[\mbox{$b_{f,g}(x,y,k)\sim\sum\limits^\infty_{p=0}b_{C_p(f,g)}(x,y,k)k^{-p}$ in $S^n(1;D\times D)$},\]
where $b_{C_p(f,g)}(x,y,k)\in S^n(1;D\times D)$ for each $p\in\N$. 
% is as in \eqref{e1.2}. {e-gue141118I}
\end{lem}
%-----------------------
\begin{proof}
Set 
\begin{equation}\label{e_c0}
C_0(f,g)=fg\in\cC^\infty_0(D).
\end{equation}
From \eqref{e-gue140521Iab} and \eqref{e-gue140521I}, we see that 
\begin{equation}\label{e-gue140922II}
b_{f,g,0}(x,x)=b_{C_0(f,g),0}(x,x),\ \ \forall x\in D.
\end{equation}
Note that $b_{f,g,0}(x,y)$ and $b_{C_0(f,g),0}(x,y)$ are holomorphic with respect to $x$ and 
\[b_{f,g,0}(x,y)=\ol b_{f,g,0}(y,x),\:\: b_{C_0(f,g),0}(x,y)=\ol b_{C_0(f,g),0}(y,x).\] 
From this observation and \eqref{e-gue140922II}, it is easy to see that $b_{f,g,0}(x,y)-b_{C_0(f,g),0}(x,y)$ 
vanishes to infinite order on $x=y$. Thus, we can take $b_{C_0(f,g),0}(x,y)$ so that 
$b_{C_0(f,g),0}(x,y)=b_{f,g,0}(x,y)$ and hence $b_{f,g}(x,y,k)-b_{C_0(f,g)}(x,y,k)\in S^{n-1}(1,D\times D)$. 
Consider the expansion
\begin{equation}\label{e-gue160717I}
b_{f,g}(x,y,k)-b_{C_0(f,g)}(x,y,k)\sim\sum\limits^\infty_{j=0}a_j(x,y)k^{n-1-j}\:\:\text{in $S^{n-1}(1;D\times D)$}, 
\end{equation}
where $a_j(x,y)\in\cC^\infty_0(D)$, $j\in\N$. Set 
\begin{equation}\label{e-gue160717II}
C_1(f,g)(x)=(2\pi)^na_0(x,x)\big|\det\dot{R}^L(x)\big|^{-1}\in\cC^\infty_0(D).
\end{equation}
From \eqref{e-gue140521I}, we have $b_{C_1(f,g),0}(x,x)=a_0(x,x)$ and as in the discussion above, we can take $b_{C_1(f,g),0}(x,y)$ so that $b_{C_1(f,g),0}(x,y)=a_0(x,y)$ and hence 
\[b_{f,g}(x,y,k)-b_{C_0(f,g)}(x,y,k)-\frac{1}{k}b_{C_1(f,g)}(x,y,k)\in S^{n-2}(1,D\times D).\]
Continuing inductively, the lemma follows.
\end{proof}
%--------------
\begin{proof}[Proof of Theorem~\ref{t-comp}]
Let $a(x,y,k)\in S^{n-j_0}(1,D\times D)$, $j_0\in\mathbb N$. Consider the operator 
\[\begin{split}
A_k:L^2(M,L^k)&\To L^2(M,L^k)\\
u&\mapsto s^ke^{k\phi}\tau\int e^{ik\Psi(x,y)}a(x,y,k)s^{-k}e^{-k\phi(y)}\tau(y)u(y)dv_M(y).
\end{split}\]
By \cite[Theorem 3.11]{HM12} we have
\begin{equation}\label{e-gue140922III}
\mbox{$A_k=\mO(k^{-j_0})$ locally in the $L^2$ operator norm}.
\end{equation}
For every $p\in\N$ put 
\[\begin{split}
B_{k,p}:L^2(M,L^k)&\To L^2(M,L^k)\\
u&\mapsto s^ke^{k\phi}\tau\int e^{ik\Psi(x,y)}b_{C_p(f,g)}(x,y,k)s^{-k}e^{-k\phi(y)}\tau(y)u(y)dv_M(y).
\end{split}\]
As in \eqref{e-gue140922I}, we can check that for $p=0,1,2,\ldots$\,,
\begin{equation}\label{e-gue140922IV}
\mbox{$B_{k,p}-T^{(0),C_p(f,g)}_{k,k^{-N}}=\mO(k^{3n-\frac{N}{2}})$ locally in the $L^2$ operator norm}. 
\end{equation}
Moreover, from \eqref{e-gue140922III} and Lemma~\ref{l-gue140922I}, we have
\begin{equation}\label{e-gue140922V}
\mbox{$B_k-\sum\limits^\ell_{p=0}B_{k,p}k^{-p}=\mO(k^{-\ell-1})$ locally in the $L^2$ operator norm},\ \ \ell=0,1,2,\ldots.
\end{equation}
From \eqref{e-gue140922I}, \eqref{e-gue140922IV} and \eqref{e-gue140922V}, we conclude that 
\[T^{(0),f}_{k,k^{-N}}T^{(0),g}_{k,k^{-N}}-\sum\limits^\ell_{p=0}T^{(0),C_p(f,g)}_{k,k^{-N}}k^{-p}=\mO(k^{-\ell-1}+k^{3n-\frac{N}{2}}) ,\ \ \ell=0,1,2,\ldots,\]
locally in the $L^2$ operator norm. 
Moreover, we have $C_0(f,g)=C_0(g,f)=fg$ by \eqref{e_c0}.
We also have $C_1(f,g)=-\frac{1}{2\pi}\langle\,\pr f\,|\,\pr\ol g\,\rangle_\omega$
by \eqref{e:c1}, so as in \cite[(0.23)]{MM12} we obtain
\begin{equation}\label{e-gue140922V1}
C_1(f,g)-C_1(g,f)=\sqrt{-1}\{f,g\}, 
\end{equation}
where $\{f,g\}$ is the Poisson bracket of the functions $f,g$ with respect to the symplectic form
$2\pi\omega$ on $M(0)$ (see also \cite[(4.89)]{MM08b}, \cite[(7.4.3)]{MM07}).
Therefore \eqref{e-gue160717}
follows. 
\end{proof}
Recall that the Poisson bracket 
$\{ \,\cdot\, , \,\cdot\, \}$
on $(M,2\pi \omega)$ is defined as follows.  For $f, g\in \cC^\infty (M)$,
let $\xi_{f}$ be the Hamiltonian vector field generated by $f$, 
which is defined by $2 \pi \omega(\xi_{f},\cdot)=df$. Then 
\begin{align}\label{toe4.1}
\{f, g\}:= \xi_{f}(dg).
\end{align}
\begin{rem}\label{rem_sp}
Berezin introduced in his ground-breaking work \cite{Berez:74} 
a star-product by using Toeplitz operators. 
Formal star-products are known to exist on symplectic manifolds by
\cite{DeLe83, Fedo:96}. 
The Berezin-Toeplitz star-product gives 
a concrete geometric realization of such product. 
For compact K\"ahler manifold the Berezin-Toeplitz star product
was introduced in \cite{KS01,Schlich:00}.
For general compact symplectic manifolds this was realized in \cite{MM07,MM08b}
by using Toeplitz operators obtained by projecting on the kernel of 
the Dirac operator.
Due to Theorem \ref{t-comp} we can also define an associative star-product
on the set $M(0)$ where a holomorphic line bundle $L\to M$ is positive,
namely by setting for any $f,g\in\cC^\infty_0(M(0))$,
\begin{equation}\label{toe4.4}
f*g:=\sum_{k=0}^\infty C_k(f,g) \hbar^{k}\in\cC^\infty(X)[[\hbar]].
\end{equation}
%Note also that a star-product was defined in \cite[Corollary 0.4]{MM12}
%also for sections $f,g\in\cC^\infty(X,\End(E))$, in the case the
%Berezin-Toeplitz operators act on holomorphic sections of $L^k\otimes E$,
%where $E\to M$ is a twisting holomorphic vector bundle.
\end{rem}

%%%%%%%%%%%%%%%%%%%%%%%%%
\section{Calculation of the leading coefficients 
%Proof of Theorem~\ref{t-gue140523I}
}\label{s-gue140729}

In this section,  
we will give formulas for the top coefficients of the expansion \eqref{e-gue140521I} 
in the case $q=0$, cf.\ Theorem~\ref{t-gue140523I}. 
We introduce the geometric objects used in Theorem~\ref{t-gue140523I} below. 
Consider the $(1,1)$-form on $M$,
\begin{equation} \label{coeI} 
\omega:=\frac{\sqrt{-1}}{2\pi}R^L. 
\end{equation}
On $M(0)$ the $(1,1)$-form $\omega$ is positive and induces a Riemannian metric 
$g^{TM}_\omega(\cdot,\cdot)=\omega(\cdot,J\cdot)$.  
In local holomorphic coordinates $z=(z_1,\ldots,z_n)$, put
\begin{equation} \label{s1-e7} 
\begin{split}
&\omega=\sqrt{-1}\sum^n_{j,k=1}\omega_{j,k}dz_j\wedge d\ol z_k,\\
&\Theta=\sqrt{-1}\sum^n_{j,k=1}\Theta_{j,k}dz_j\wedge d\ol z_k. 
\end{split}
\end{equation}
We notice that $\Theta_{j,k}=\langle\,\frac{\pr}{\pr z_j}\,|\,\frac{\pr}{\pr z_k}\,\rangle$, $\omega_{j,k}=\langle\,\frac{\pr}{\pr z_j}\,|\,\frac{\pr}{\pr z_k}\,\rangle_{\omega}$, $j,k=1,\ldots,n$.
Put 
\begin{equation} \label{s1-e8}
h=\left(h_{j,k}\right)^n_{j,k=1},\ \ h_{j,k}=\omega_{k,j},\ \ j, k=1,\ldots,n,
\end{equation}
and $h^{-1}=\left(h^{j,k}\right)^n_{j,k=1}$, $h^{-1}$ is the inverse matrix of $h$. The complex Laplacian with respect to $\omega$ is given by 
\begin{equation} \label{sa1-e9} 
\triangle_{\omega}=(-2)\sum^n_{j,k=1}h^{j,k}\frac{\pr^2}{\pr z_j\pr\ol z_k}\cdot
\end{equation} 
We notice that $h^{j,k}=\langle\,dz_j\,|\,dz_k\,\rangle_{\omega}$, $j, k=1,\ldots,n$. Put 
\begin{equation} \label{s1-e10} 
\begin{split}
&V_\omega:=\det\left(\omega_{j,k}\right)^n_{j,k=1},\\
&V_\Theta:=\det\left(\Theta_{j,k}\right)^n_{j,k=1}
\end{split}
\end{equation}
and set  
\begin{equation} \label{sa1-e11} 
\begin{split}
&r=\triangle_{\omega}\log V_\omega,\\
&\hat r=\triangle_{\omega}\log V_\Theta.
\end{split}
\end{equation} 
$r$ is called the scalar curvature with respect to $\omega$. Let $R^{\det}_\Theta$ be the 
curvature of the canonical line bundle $K_M=\det\Lambda^{1,0}(T^*M)$ 
with respect to the real two form $\Theta$. 
We recall that 
\begin{equation} \label{sa1-e12}
R^{\rm det\,}_\Theta=-\ddbar\pr\log V_\Theta.
\end{equation} 

Let $\nabla^{TM}_\omega$ 
be the Levi-Civita connection on $(M(0),g^{TM}_\omega)$, 
$R^{TM}_\omega=(\nabla^{TM}_\omega)^2$ its curvature.
Let $h$ be as in \eqref{s1-e8}. Put 
$\theta=h^{-1}\pr h=\left(\theta_{j,k}\right)^n_{j,k=1}$, $\theta_{j,k}\in \Lambda^{1,0}(T^*M)$, $j,k=1,\ldots,n$.
$\theta$ is the Chern connection 
matrix with respect to $\omega$. Then, 
\begin{equation} \label{sa1-e13}
\begin{split}
&R^{TM}_{\omega}=\ddbar\theta=\left(\ddbar\theta_{j,k}\right)^n_{j,k=1}=
\left(\mathcal{R}_{j,k}\right)^n_{j,k=1}
\in\cC^\infty(M,\Lambda^{1,1}(T^*M)\otimes{\rm End\,}(T^{1,0}M)),\\
&R^{TM}_{\omega}(\ol U, V)\in {\rm End\,}(T^{1,0}M),\ \ \forall U, V\in T^{1,0}M,\\
&R^{TM}_\omega(\ol U,V)\xi=\sum^n_{j,k=1}\langle\,
\mathcal{R}_{j,k}\,|\,\ol U\wedge V\,\rangle\,\xi_k\frac{\pr}{\pr z_j},\ \ \xi=
\sum^n_{j=1}\xi_j\frac{\pr}{\pr z_j},\ \ U, V\in T^{1,0}M.
\end{split}
\end{equation} 
We denote by $\langle\,\cdot\,,\cdot\,\rangle_{\omega}$ the pointwise 
Hermitian metrics induced by $g^{TM}_\omega$ on $\Lambda^{p,q}(T^*M)\otimes\Lambda^{r,s}(T^*M)$, 
$p, q, r, s\in\set{0,1,\ldots,n}$, and by $\abs{\,\cdot\,}_{\omega}$ the corresponding norms.

Set 
\begin{equation} \label{sa1-e14} 
\abs{R^{TM}_{\omega}}^2_{\omega}:=
\sum^n_{j,k,s,t=1}\abs{\langle\,R^{TM}_{\omega}(\ol e_j,e_k)e_s\,|\,e_t\,\rangle_{\omega}}^2,
\end{equation} 
where $e_1,\ldots,e_n$ is an orthonormal frame for $T^{1,0}M$ with respect to 
$\langle\,\cdot\,,\cdot\,\rangle_{\omega}$. 
It is straightforward to see that the definition of $\abs{R^{TM}_{\omega}}^2_{\omega}$ 
is independent of the choices of orthonormal frames. 
Thus, $\abs{R^{TM}_{\omega}}^2_{\omega}$ is globally defined. 
The Ricci curvature with respect to $\omega$ is given by 
\begin{equation} \label{sa1-e15}
{\rm Ric\,}_{\omega}:=-\sum^n_{j=1}\langle\,R^{TM}_\omega(\cdot,e_j)\cdot\,|\,e_j\,\rangle_\omega,
\end{equation}
where $e_1,\ldots,e_n$ is an orthonormal frame for $T^{1,0}M$ with respect to 
$\langle\,\cdot\,,\cdot\,\rangle_{\omega}$. That is, 
\[\langle\,{\rm Ric\,}_{\omega}\,|\, U\wedge V\,\rangle=
-\sum^n_{j=1}\langle\,R^{TX}_\omega(U,e_j)V\,|\,e_j\,\rangle_\omega,\ \ 
U, V\in TM\otimes_\Real\Complex.\]
${\rm Ric\,}_{\omega}$ is a global $(1,1)$ form.

Let 
\begin{equation} \label{sa1-e15-I}
D^{1,0}:\cC^\infty(M,\Lambda^{1,0}(T^*M))\To\cC^\infty(M,\Lambda^{1,0}(T^*M)\otimes\Lambda^{1,0}(T^*M))
\end{equation}
be the $(1,0)$ component of the Chern connection on $\Lambda^{1,0}(T^*M)$ induced 
by $\langle\,\cdot\,,\cdot\,\rangle_\omega$. That is, in local coordinates $z=(z_1,\ldots,z_n)$, put 
\[A=\left(a_{j,k}\right)^n_{j,k=1},\ \ a_{j,k}=\langle\,dz_k\,|\,dz_j\,\rangle_\omega,\ \ j,k=1,\ldots,n,\] 
and set 
\begin{equation} \label{s1-e15-Ibis}
\mathcal{A}=A^{-1}\pr A=\left(\alpha_{j,k}\right)^n_{j,k=1},\ \ \alpha_{j,k}\in\Lambda^{1,0}(T^*M),\ \ j,k=1,\ldots,n.
\end{equation}
Then, for $u=\sum^n_{j=1}u_jdz_j\in\cC^\infty(M,\Lambda^{1,0}(T^*M))$, we have 
\[D^{1,0}u=\sum^n_{j=1}\pr u_j\otimes dz_j+
\sum^n_{j,k=1}u_j\alpha_{k,j}\otimes dz_k\in\cC^\infty(M,\Lambda^{1,0}(T^*M)\otimes\Lambda^{1,0}(T^*M)).\] 

\begin{thm}\label{t-gue140523I}
With the assumptions and notations used in Theorem~\ref{t-gue140521}, the coefficients 
$b_{f,1}(x,x)$ and $b_{f,2}(x,x)$  in the expansion \eqref{e1.2} for 
$q=0$ have the following form: for every $x\in D_0$, 
\begin{equation} \label{s1-e16-II} 
b_{f,1}(x,x)=(2\pi)^{-n}f(x)\det\dot{R}^L(x)\Bigr(\frac{1}{4\pi}\widehat r-\frac{1}{8\pi} r\Bigr)(x)
+(2\pi)^{-n}\det\dot{R}^L(x)\Bigr(-\frac{1}{4\pi}\triangle_{\omega}f\Bigr)(x),
\end{equation} 
\begin{equation} \label{s1-e16-III}
\begin{split}
b_{f,2}(x,x)&=(2\pi)^{-n}f(x)\det\dot{R}^L(x)\Bigr(\frac{1}{128\pi^2}r^2-
\frac{1}{32\pi^2}r\widehat r+\frac{1}{32\pi^2}(\widehat r)^2
-\frac{1}{32\pi^2}\triangle_{\omega}\widehat r\\
&\quad-\frac{1}{8\pi^2}\abs{R^{\det}_\Theta}^2_{\omega}+\frac{1}{8\pi^2}\langle\,
{\rm Ric\,}_{\omega}\,|\,R^{\det}_\Theta\,\rangle_{\omega}+\frac{1}{96\pi^2}\triangle_{\omega}r-
\frac{1}{24\pi^2}\abs{{\rm Ric\,}_{\omega}}^2_{\omega}\\
&\quad+\frac{1}{96\pi^2}\abs{R^{TX}_{\omega}}^2_{\omega}\Bigr)(x)+
(2\pi)^{-n}\det\dot{R}^L(z)\Bigr(\frac{1}{16\pi^2}(\triangle_{\omega}f)(-\widehat r+\frac{1}{2}r)\\
&\quad-\frac{1}{4\pi^2}\langle\,\ddbar\pr f\,|\,R^{\det}_\Theta\,\rangle_{\omega}+
\frac{1}{8\pi^2}\langle\,\ddbar\pr f\,|\,{\rm Ric\,}_{\omega}\,\rangle_{\omega}+
\frac{1}{32\pi^2}\triangle^2_{\omega}f\Bigr)(x).
\end{split}
\end{equation}
\end{thm}
%----
The formulas given in Theorem \ref{t-gue140523I} simplify if we assume that $\omega=\Theta$.
In this case we have $V_\omega=V_\Theta$ and $r=\widehat{r}$. See also 
\cite[Section\,2.7]{MM11}, \cite[Remark\,0.5]{MM12},
concerning the calculation of the coefficients for an arbitrary underlying Hermitian metric $\Theta$.

Let $q=0$ and let 
\[S_k(x,y)=e^{ik\Psi(x,y)}b(x,y,k)\] 
be as in \eqref{e-gue140509}. Note that $b(x,y,k)\in S^{n}(1;D_0\times D_0)$,
\[b(x,y,k)\sim\sum^\infty_{j=0}b_{j}(x,y)k^{n-j}\text{ in }S^{n}(1;D_0\times D_0),\]
where $b_{j}(x,y)\in\cC^\infty(D_0\times D_0)$, $j\in\N$, and
$b_{0}(x,x)=(2\pi)^{-n}\big|\det\dot{R}^L(x)\big|$. We have 
\begin{equation}\label{e-gue140731}
\int S_k(x,z,k)f(z)S_k(z,y,k)dv_M(z)\equiv e^{ik\Psi(x,y)}b_f(x,y,k)\mod O(k^{-\infty})
\end{equation}
locally uniformly on $D_0\times D_0$, where $b_f(x,y,k)\in S^{n}(1;D_0\times D_0)$,
\begin{equation}\label{s1-e5}
b_f(x,y,k)\sim\sum^\infty_{j=0}b_{f,j}(x,y)k^{n-j}\text{ in }S^{n}(1;D_0\times D_0),
\end{equation}
where $b_{f,j}(x,y)\in\cC^\infty(D_0\times D_0)$, $j\in\N$\,, and
$b_{0,f}(x,x)=(2\pi)^{-n}f(x)\big|\det\dot{R}^L(x)\big|$. In this section, we will calculate 
$b_{1,f}(x,x)$ and $b_{2,f}(x,x)$, $x\in D_0$. Fix $p\in D_0$. 
In a small neighbourhood of the point $p$ there exist local coordinates 
$z=(z_1,\ldots,z_n)=x=(x_1,\ldots,x_{2n})$, $z_j=x_{2j-1}+ix_{2j}$, 
$j=1,\ldots,n$, and a local frame $s$ of $L$, $\abs{s}_{h}^2=e^{-2\phi}$ so that 
\begin{equation} \label{coecaluI}
\begin{split}
&z(p)=0,\\
&\phi(z)=\sum^n_{j=1}\lambda_j\abs{z_j}^2+\phi_1(z),\\
&\phi_1(z)=O(\abs{z})^4),\ \ \frac{\pr^{\abs{\alpha}+
\abs{\beta}}\phi_1}{\pr z^\alpha\pr\ol z^\beta}(0)=0\ \ 
\text{for $\alpha, \beta\in\N^n$, $\abs{\alpha}\leq1$ or $\abs{\beta}\leq1$}\,,\\
&\Theta(z)=\sqrt{-1}\sum^n_{j=1}dz_j\wedge d\ol z_j+O(\abs{z}).
\end{split}
\end{equation} 
Until further notice, we work with this local coordinates $x$ and we identify $p$ 
with the point $x=z=0$. It is well-known (see~\cite[Section 4.5]{HM12}) that for every $N\in\N$, we have
\begin{equation} \label{coecaluXI} 
\begin{split}
\Psi(z,0)=i\phi(z)+O(\abs{z}^{N})\,,\:\:\Psi(0,z)=i\phi(z)+O(\abs{z}^{N}).
\end{split}
\end{equation}
We have
\begin{equation} \label{s4-eI}
\begin{split}
&\int S_k(0,z,k)f(z)S_k(z,0,k)dv_M(z)\\
&=\int_{D_0}e^{ik(\Psi(0,z)+\Psi(z,0))}b(0,z,k)b(z,0,k)f(z)V_\Theta(z)d\lambda(z)+r_k,
\end{split}
\end{equation}
where $d\lambda(z)=2^ndx_1dx_2\cdots dx_{2n}$, $dv_M(z)=V_\Theta(z)d\lambda(z)$ and 
\[\lim_{k\to\infty}\frac{r_k}{k^N}=0,\ \ \forall N\geq0.\] 
We notice that since $b(z,w,k)$ is properly supported, we have 
\begin{equation} \label{s4-eII}
b(0,z,k)\in\cC^\infty_0(D_0),\ \ b(z,0,k)\in\cC^\infty_0(D_0).
\end{equation}
We recall the stationary phase formula of H\"{o}rmander (see~\cite[Theorem 7.7.5]{Hor03}). 
\begin{thm} \label{s4-tI} 
Let $K\subset D$ be a compact set and $N$ a positive integer. If $u\in\cC^\infty_0(K)$, 
$F\in\cC^\infty(D)$ and ${\rm Im\,}F\geq0$ in $D$, 
${\rm Im\,}F(0)=0$, $F'(0)=0$, ${\rm det\,}F''(0)\neq0$, $F'\neq0$ in $K\setminus\set{0}$ then 
\begin{equation} \label{s4-eIII} 
\begin{split}
&\abs{\int e^{ikF(z)}u(z)V_\Theta(z)d\lambda(z)-2^ne^{ikF(0)}{\rm det\,}
\left(\frac{kF''(0)}{2\pi i}\right)^{-\frac{1}{2}}\sum_{j<N}k^{-j}L_ju} \\
&\quad\leq Ck^{-N}\sum_{\abs{\alpha}\leq 2N}\sup{\abs{\pr^\alpha_x u}},\ \ k>0,
\end{split}
\end{equation} 
where $C$ is uniform when $F$ runs in a relatively compact set of $\cC^\infty(D)$, 
$\frac{\abs{x}}{\abs{F'(x)}}$ has a uniform bound and
\begin{equation} \label{s4-eIV} 
L_ju=\sum_{\nu-\mu=j}\sum_{2\nu\geq 3\mu}i^{-j}2^{-\nu}
\langle F''(0)^{-1}D,D\rangle^\nu\frac{(h^\mu V_\Theta u)(0)}{\nu!\mu!}.
\end{equation} 
Here 
\begin{equation} \label{s4-eV} 
h(x)=F(x)-F(0)-\frac{1}{2}\langle F''(0)x,x\rangle
\end{equation} 
and $D=\begin{pmatrix} -i\pr_{x_1} \\ \vdots  \\ -i\pr_{x_{2n}} \end{pmatrix}$.
\end{thm} 
%--------
\noindent
We now apply \eqref{s4-eIII} to the integral in \eqref{s4-eI}. Put 
\[F(z)=\Psi(0,z)+\Psi(z,0).\] 
From \eqref{coecaluI} and \eqref{coecaluXI}, we see that 
\begin{equation} \label{s4-eVI}
\begin{split}
&F(z)=2i\sum^n_{j=1}\lambda_j\abs{z_j}^2+2i\phi_1(z)+O(\abs{z}^N),\ \ \forall N\geq0,\\
&h(z)=2i\phi_1(z)+O(\abs{z}^N),\ \ \forall N\geq0,
\end{split}
\end{equation}
where $h$ is given by \eqref{s4-eV}. Moreover, we can check that 
\begin{equation} \label{s4-eVII}
{\rm det\,}\left(\frac{kF''(0)}{2\pi i}\right)^{-\frac{1}{2}}=
k^{-n}\pi^n2^{-n}\lambda^{-1}_1\lambda^{-1}_2\cdots\lambda_n^{-1}
=k^{-n}\pi^n\bigr(\det\dot{R}^L(0)\bigr)^{-1}
\end{equation}
and
\begin{equation} \label{s4-eVIII}
\begin{split}
\langle F''(0)^{-1}D,D\rangle=i\triangle_0\,,\qquad
\triangle_0=\sum^n_{j=1}\frac{1}{\lambda_j}\frac{\pr^2}{\pr\ol z_j\pr z_j}\,\cdot
\end{split}
\end{equation}
From \eqref{s4-eVI}, \eqref{s4-eVIII} and using that $h=O(\abs{z}^4)$, it is not 
difficult to see that  
\begin{equation} \label{s4-eIX} 
L_j(b(0,z,k)b(z,0,k)f)=\sum_{\nu-\mu=j}\sum_{2\nu\geq 4\mu}
(-1)^\mu2^{-j}\frac{\triangle_0^\nu\bigr(\phi_1^\mu V_\Theta b(0,z,k)b(z,0,k)f\bigr)(0)}{\nu!\mu!},
\end{equation} 
where $L_j$ is given by \eqref{s4-eIV}. We notice that 
\[b(0,z,k)\equiv\sum^\infty_{j=0}b_j(0,z)k^{n-j}\mod O(k^{-\infty}),\ \ 
b(z,0,k)\equiv\sum^\infty_{j=0}b_j(z,0)k^{n-j}\mod O(k^{-\infty}).\] 
From this observation, \eqref{s4-eIX} becomes: 
\begin{equation} \label{s4-eX} 
\begin{split}
&L_j\big(b(0,z,k)b(z,0,k)f\big)\\
&=\sum_{\nu-\mu=j}\sum_{2\nu\geq 4\mu}\sum_{0\leq s+t\leq N}(-1)^\mu2^{-j}
\frac{k^{2n-s-t}\triangle_0^\nu\bigr(\phi_1^\mu V_\Theta fb_s(0,z)b_t(z,0)\bigr)(0)}{\nu!\mu!}+O(k^{2n-N-1}),
\end{split}
\end{equation} 
for all $N\geq0$. From \eqref{s4-eX}, \eqref{s4-eVII}, \eqref{s4-eIII}, 
\eqref{s4-eI} and \eqref{e-gue140731}, we get 
\begin{equation} \label{s4-eXI}
\begin{split}
&b_f(0,0,k)=(2\pi)^n(\det\dot{R}^L(0))^{-1}\times\\
&\sum^N_{j=0}k^{n-j}\Bigr(\sum_{0\leq m\leq j}\sum_{\nu-\mu=m}
\sum_{2\nu\geq 4\mu}\sum_{s+t=j-m}(-1)^\mu2^{-m}
\frac{\triangle_0^\nu\bigr(\phi_1^\mu V_\Theta fb_s(0,z)b_t(z,0)\bigr)(0)}{\nu!\mu!}\Bigr)\\
&\quad+O(k^{n-N-1}),\ \ \forall N\geq0.
\end{split}
\end{equation} 
Combining \eqref{s4-eXI} with \eqref{s1-e5}, we obtain

\begin{thm} \label{s4-tII} 
The coefficients $b_{f,j}$ of the expansion \eqref{s1-e5} of $b_f(x,y,k)$, are given by 
\begin{equation} \label{s4-eXII}
\begin{split}
&b_{f,j}(0,0)\\
&=(2\pi)^n(\det\dot{R}^L(0))^{-1}\sum_{0\leq m\leq j}\,
\sum_{\nu-\mu=m}\,\sum_{2\nu\geq 4\mu}\,\sum_{s+t=j-m}(-1)^\mu2^{-m}
\frac{\triangle_0^\nu\bigr(\phi_1^\mu V_\Theta fb_s(0,z)b_t(z,0)\bigr)(0)}{\nu!\mu!},
\end{split}
\end{equation} 
for all $j=0,1,\ldots$\,. 
In particular, 
\begin{equation} \label{s4-eXIII}
b_{f,0}(0,0)=(2\pi)^n(\det\dot{R}^L(0))^{-1}f(0)b_0(0,0)^2,
\end{equation}
\begin{equation} \label{s4-eXIV}
\begin{split}
b_{f,1}(0,0)&=(2\pi)^n(\det\dot{R}^L(0))^{-1}\Bigr(2f(0)b_0(0,0)b_1(0,0)\\
&\quad+\frac{1}{2}\triangle_0\bigr(V_\Theta fb_0(0,z)b_0(z,0)\bigr)(0)
-\frac{1}{4}\triangle^2_0\bigr(\phi_1V_\Theta fb_0(0,z)b_0(z,0)\bigr)(0)\Bigr)
\end{split}
\end{equation}
and 
\begin{equation} \label{s4-eXV}
\begin{split}
b_{f,2}(0,0)&=(2\pi)^n(\det\dot{R}^L(0))^{-1}\Bigr(2f(0)b_0(0,0)b_2(0,0)+f(0)b_1(0,0)^2\\
&\quad+\frac{1}{2}\triangle_0\bigr(V_\Theta f(b_0(0,z)b_1(z,0)+b_1(0,z)b_0(z,0))\bigr)(0)\\
&\quad-\frac{1}{4}\triangle^2_0\bigr(\phi_1V_\Theta f(b_0(0,z)b_1(z,0)+b_1(0,z)b_0(z,0))\bigr)(0)\\
&\quad+\frac{1}{8}\triangle^2_0\bigr(V_\Theta fb_0(0,z)b_0(z,0)\bigr)(0)
-\frac{1}{24}\triangle^3_0\bigr(\phi_1V_\Theta fb_0(0,z)b_0(z,0)\bigr)(0)\\
&\quad+\frac{1}{192}\triangle^4_0\bigr(\phi_1^2V_\Theta fb_0(0,z)b_0(z,0)\bigr)(0)\Bigr).
\end{split}
\end{equation}
\end{thm}

In~\cite[Section 4.5]{HM12}, we determined all the derivatives of $b_0(x,y)$, $b_1(x,y)$, $b_2(x,y)$ at $(0,0)$. 
From this observation and Theorem~\ref{s4-tII}, we can repeat the procedure in~\cite[Section 4]{Hsiao09} 
and obtain Theorem~\ref{t-gue140523I}. Since the calculation is the same, we omit the details.

Let 
%\begin{equation} \label{s1-e15-I}
$D^{1,0}:\cC^\infty(M,\Lambda^{1,0}(T^*M))\To\cC^\infty(M,\Lambda^{1,0}(T^*M)\otimes\Lambda^{1,0}(T^*M))$
%\end{equation}
be the $(1,0)$ component of the Chern connection on $\Lambda^{1,0}(T^*M)$ induced by 
$\langle\,\cdot\,,\cdot\,\rangle_\omega$ (see the discussion after \eqref{sa1-e15-I}).
From Theorem~\ref{t-gue140523I} and the proof of Theorem~\ref{t-gue140729II}, we can repeat the proof of 
\cite[Theorem\,1.5]{Hsiao09} and get the following (see also Ma-Marinescu~\cite{MM07} for another method).

\begin{thm} \label{s1-tmaincomp} 
With the notations as in Theorem~\ref{t-gue140729II}, let $q=0$. 
Then, for $b_{f,g,1}$, $b_{f,g,2}$ in \eqref{e-gue141118I}, we have  
\begin{equation} \label{e-compmainII} 
\begin{split}
b_{f,g,1}(x)=b_{fg,1}(x)+(2\pi)^{-n}\det\dot{R}^L(x)\Bigr(-\frac{1}{2\pi}\langle\,\pr f\,|\,\pr\ol g\,\rangle_\omega\Bigr)(x),
\end{split}
\end{equation} 
\begin{equation} \label{e-compmainIII}
\begin{split}
&b_{f,g,2}(x)=b_{fg,2}(x)\\
&\quad+(2\pi)^{-n}\det\dot{R}^L(x)\Bigr(-\frac{1}{4\pi^2}\langle\,\ddbar g\wedge\pr f\,|\,{\rm Ric\,}_\omega\,\rangle_\omega+\frac{1}{4\pi^2}\langle\,\ddbar g\wedge\pr f\,|\,R^{\det}_\Theta\,\rangle_\omega \\
&\quad+\frac{1}{8\pi^2}\langle\,\pr\triangle_\omega f\,|\,\pr\ol g\,\rangle_\omega
+\frac{1}{8\pi^2}\langle\,\ddbar\triangle_\omega g\,|\,\ddbar\,\ol f\,\rangle_\omega
-\frac{1}{8\pi^2}\langle\,D^{1,0}\pr f\,|\,D^{1,0}\pr\ol g\,\rangle_\omega\\
&\quad-\frac{1}{4\pi^2}\langle\,\ddbar\pr f\,|\,\ddbar\pr\ol g\,\rangle_\omega+
\frac{1}{8\pi^2}\langle\,\pr f\,|\,\pr\ol g\,\rangle_\omega(-\widehat r+\frac{1}{2}r)\Bigr)(x).
\end{split}
\end{equation}
\end{thm} 
%----------------
\begin{cor}\label{cor:C}
The coefficients $C_1(f,g)$ and $C_2(f,g)$ of the expansion
\eqref{0.2.1} of the composition $T^{(0),f}_{k,k^{-N}}\circ T^{(0),g}_{k,k^{-N}}$
of two Toeplitz operators are given by
\begin{equation}\label{e:c1}
C_1(f,g)=-\frac{1}{2\pi}\langle\,\pr f\,|\,\pr\ol g\,\rangle_\omega\,,
\end{equation}
\begin{equation}\label{e:c2}
C_2(f,g)=\frac{1}{8\pi^2}\langle\,D^{1,0}\pr f\,|\,D^{1,0}\pr\ol g\,\rangle_\omega+
\frac{1}{4\pi^2}\langle\,\overline\pr g\wedge\pr f\,|\,R^{det}_\Theta\,\rangle_\omega\,.
\end{equation}
\end{cor}
%----------------
\begin{proof}
Formula \eqref{e:c1} follows from \eqref{e-compmainII}
and
\[b_{1,f,g} = b_{1,fg} + b_{0,C_1(f,g)}=b_{1,fg} +(2\pi)^{-n}(\det\dot{R}^L)C_1(f,g)\,,\] 
see \cite[(5.21)]{Hsiao09} or \cite[(5.76)]{MM12}.
Formula \eqref{e:c2} follows as in \cite[Section 5.3]{Hsiao09}.
\end{proof}
%\section{Proofs of Theorem~\ref{s1-maindege} and Theorem~\ref{t-gue140523III}}\label{s-gue140802}
\section{Behavior on the degenerate set and the Weyl law}\label{s-gue140802}

In this section, we will prove Theorem~\ref{s1-maindege} and Theorem~\ref{t-gue140523III}. 
We recall first the following.
%-----------------
\begin{thm}[{\cite[Theorem 1.3]{HM12}}]\label{t-gue140731a}
Set
\[M_{\mathrm{deg}}=\set{x\in M;\, \mbox{$\dot{R^L}$ is degenerate at $x\in M$}}.\]
Then for every $x_0\in M_{\mathrm{deg}}$, $\varepsilon>0$, $N>1$ and every $m\in\{0,1,\dots,n\}$ , there 
exist a neighbourhood $U$ of $x_0$ and $k_0>0$, such that for all $k\geq k_0$ we have
\begin{equation} \label{s1-e3maina}
\abs{P^{(m)}_{k,k^{-N}}(x,x)}\leq \varepsilon k^n,\ \ x\in U.
\end{equation}
\end{thm}

\begin{proof}[Proof of Theorem~\ref{s1-maindege}]
Fix  $x_0\in M_{\mathrm{deg}}$, $\varepsilon>0$ and $m\in\{0,1,\dots,n\}$. 
Let $U$ be a small neighbourhood of $x_0$ as in Theorem~\ref{t-gue140731a}. 
Let $p$ be any point of $U$ and let $s$ be a local section of $L$ defined in a small open set 
$D\Subset U$ of $p$, $\abs{s}^2_{h}=e^{-2\phi}$. 
Fix  $\abs{I_0}=\abs{J_0}=q$, $I_0$, $J_0$ are strictly increasing. 
Take $\set{\alpha_1(x),\alpha_2(x),\ldots,\alpha_{d_k}(x)}$ and 
$\set{\beta_1(x),\beta_2(x),\ldots,\beta_{d_k}(x)}$ be orthonormal frames 
for $\cE_{k^{-N}}(M,L^k)$ so that
\[\begin{split}
&\abs{\Td\alpha_{1,I_0}(p)e^{-k\phi(p)}}^2=\sum^{d_k}_{j=1}\abs{\Td\alpha_{j,I_0}(p)e^{-k\phi(p)}}^2,\\
&\abs{\Td\beta_{1,J_0}(p)e^{-k\phi(p)}}^2=\sum^{d_k}_{j=1}\abs{\Td\beta_{j,J_0}(p)e^{-k\phi(p)}}^2,\end{split}\] 
where $d_k\in\mathbb N\cup\set{\infty}$ and on $D$, we write
\[\begin{split}
&\alpha_j(x)=s^k(x)\Td\alpha_j(x),\ \ \mbox{$\Td\alpha_j(x)=
\sideset{}{'}\sum\limits_{\abs{J}=q}\Td \alpha_{j,J}(x)e^J(x)$ on $D$},\ \ j=1,\ldots,d_k,\\
&\beta_j(x)=s^k(x)\Td\beta_j(x),\ \ \mbox{$\Td\beta_j(x)=
\sideset{}{'}\sum\limits_{\abs{J}=q}\Td \beta_{j,J}(x)e^J(x)$ on $D$},\ \ j=1,\ldots,d_k.\end{split}\]
We have
\begin{equation}\label{e-gue140731a}
\begin{split}
T^{(q),f,I_0,J_0}_{k,k^{-N}}(p,p)&=
\sum\limits^{d_k}_{j,\ell=1}\Td\alpha_{j,I_0}(p)
e^{-k\phi(p)}(\,f\beta_\ell\,|\,\alpha_j\,)_{k}\ol{\Td\beta_{\ell,J_0}(p)}e^{-k\phi(p)}\\
&=\Td\alpha_{1,I_0}(p)e^{-k\phi(p)}(\,f\beta_1\,|\,\alpha_1\,)_{k}\ol{\Td\beta_{1,J_0}(p)}e^{-k\phi(p)}.
\end{split}
\end{equation}
From \eqref{e-gue140731a}, it is not difficult to see that 
\begin{equation}\label{e-gue140731aI}
\abs{T^{(q),f,I_0,J_0}_{k,k^{-N}}(p,p)}\leq\sup\set{\abs{f(x)};\, x\in M}\abs{P^{(q)}_{k,k^{-N}}(p,p)}.
\end{equation}
From \eqref{e-gue140731aI} and \eqref{s1-e3maina}, the theorem follows. 
\end{proof}

We now prove Theorem~\ref{t-gue140523III}. We introduce some notations. For $\lambda\geq0$, put 
\[\begin{split}
P^{(q)}_{k,0<\mu\leq\lambda}:=E(]0,\lambda]),\:\:
\cE^q_{0<\mu\leq\lambda}(M,L^k):={\rm Rang\,}E(]0,\lambda]).\end{split}\] 
Recall that $E$ denotes the spectral measure of $\Box^{(q)}_k$. 
Let $P^{(q)}_{k,0<\mu\leq\lambda}(\,\cdot\,,\cdot)$ be the Schwartz kernel of 
$P^{(q)}_{k,0<\mu\leq\lambda}$. The trace of $P^{(q)}_{k,0<\mu\leq\lambda}(x,x)$ is given by
\[{\rm Tr\,}P^{(q)}_{k,0<\mu\leq\lambda}(x,x):=
\sum^d_{j=1}\big\langle\,P^{(q)}_{k,0<\mu\leq\lambda}(x,x)\,e_{J_j}(x)\,|\,e_{J_j}(x)\,\big\rangle,\]
where $e_{J_1},\ldots,e_{J_d}$ is a local orthonormal basis of $\Lambda^{0,q}(T^*M)$ with respect to 
$\langle\,\cdot\,,\cdot\,\rangle$. Now, we assume that $M$ is compact. We need the following.
%-----------
\begin{lem}\label{l-gue140731a}
There exists $C>0$ independent of $k$ such that 
\begin{equation}\label{e-gue140731aII}
\abs{T^{(q),f}_{k}(x,x)-T^{(q),f}_{k,k^{-N}}(x,x)}\leq 
Ck^{\frac{n}{2}}\sqrt{{\rm Tr\,}P^{(q)}_{k,0<\mu\leq k^{-N}}(x,x)},\ \ \forall x\in M.
\end{equation}
\end{lem}

\begin{proof}
Let $p$ be any point of $M$ and let $s$ be a local trivializing holomorphic section of $L$ defined in a 
small open set $D\Subset U$ of $p$, $\abs{s}^2_{h}=e^{-2\phi}$. 
Fix $\abs{I_0}=\abs{J_0}=q$, $I_0$, $J_0$ are strictly increasing. 
Take $\set{\alpha_1(x),\alpha_2(x),\ldots,\alpha_{m_k}(x)}$, 
$\set{\beta_1(x),\beta_2(x),\ldots,\beta_{m_k}(x)}$ to be orthonormal frames for $\cE^q_{0}(M,L^k)$ 
and 
\[\set{\alpha_{m_k+1}(x),\alpha_{m_k+2}(x),\ldots,\alpha_{d_k}(x)},\ \ 
\set{\beta_{m_k+1}(x),\beta_{m_k+2}(x),\ldots,\beta_{d_k}(x)}\] 
to be orthonormal frames for $\cE^q_{0<\mu\leq k^{-N}}(M,L^k)$ so that
\[\begin{split}
&\abs{\Td\alpha_{1,I_0}(p)e^{-k\phi(p)}}^2=\sum^{m_k}_{j=1}\abs{\Td\alpha_{j,I_0}(p)e^{-k\phi(p)}}^2,\\
&\abs{\Td\alpha_{m_k+1,I_0}(p)e^{-k\phi(p)}}^2=\sum^{d_k}_{j=m_k+1}\abs{\Td\alpha_{j,I_0}(p)e^{-k\phi(p)}}^2,\\
&\abs{\Td\beta_{1,J_0}(p)e^{-k\phi(p)}}^2=\sum^{m_k}_{j=1}\abs{\Td\beta_{j,J_0}(p)e^{-k\phi(p)}}^2,\\
&\abs{\Td\beta_{m_k+1,J_0}(p)e^{-k\phi(p)}}^2=
\sum^{d_k}_{j=m_k+1}\abs{\Td\beta_{j,J_0}(p)e^{-k\phi(p)}}^2,\end{split}\] 
where $d_k\in\mathbb N\cup\set{\infty}$ and on $D$, we write
\[\begin{split}
&\alpha_j(x)=s^k(x)\Td\alpha_j(x),\ \ \mbox{$\Td\alpha_j(x)=\sideset{}{'}\sum\limits_{\abs{J}=q}\Td \alpha_{j,J}(x)e^J(x)$ 
on $D$},\ \ j=1,\ldots,d_k,\\
&\beta_j(x)=s^k(x)\Td\beta_j(x),\ \ \mbox{$\Td\beta_j(x)=\sideset{}{'}\sum\limits_{\abs{J}=q}\Td \beta_{j,J}(x)e^J(x)$ 
on $D$},\ \ j=1,\ldots,d_k.\end{split}\]
We have
\begin{equation}\label{e-gue140731aIII}
\begin{split}
&T^{(q),f,I_0,J_0}_{k}(p,p)\\
&=\Td\alpha_{1,I_0}(p)e^{-k\phi(p)}(\,f\beta_1\,|\,\alpha_1\,)_{k}\ol{\Td\beta_{1,J_0}(p)}e^{-k\phi(p)}\\
&=T^{(q),f,I_0,J_0}_{k,k^{-N}}(p,p)-\Td\alpha_{m_k+1,I_0}(p)e^{-k\phi(p)}
(\,f\beta_1\,|\,\alpha_{m_k+1}\,)_{k}\ol{\Td\beta_{1,J_0}(p)}e^{-k\phi(p)}\\
&\quad-\Td\alpha_{m_k+1,I_0}(p)e^{-k\phi(p)}
(\,f\beta_{m_k+1}\,|\,\alpha_{m_k+1}\,)_{k}\ol{\Td\beta_{m_k+1,J_0}(p)}e^{-k\phi(p)}\\
&\quad-\Td\alpha_{1,I_0}(p)e^{-k\phi(p)}(\,f\beta_{m_k+1}\,|\,\alpha_1\,)_{k}\ol{\Td\beta_{m_k+1,J_0}(p)}e^{-k\phi(p)}.
\end{split}
\end{equation}
From \eqref{e-gue140731aIII}, it is easy to see that 
\[
\abs{T^{(q),f,I_0,J_0}_{k}(p,p)-T^{(q),f,I_0,J_0}_{k,k^{-N}}(p,p)}\leq 
C_1k^{\frac{n}{2}}\sqrt{{\rm Tr\,}P^{(q)}_{k,0<\mu\leq k^{-N}}(p,p)},\]
where $C_1>0$ is a constant independent of $k$ and the point $p$. The lemma follows.
\end{proof}

\begin{proof}[Proof of Theorem~\ref{t-gue140523III}]
Since $M(q-1)=\emptyset$ and $M(q+1)=\emptyset$, it is known \cite[Corollary 1.4]{HM12}, that for every $N>1$, 
\begin{equation}\label{e-gue140731aV}
\dim\cE^{q-1}_{k^{-N}}(M,L^k)=o(k^n),\:\:
\dim\cE^{q+1}_{k^{-N}}(M,L^k)=o(k^n).
\end{equation}
Moreover, it is easy to see that 
\begin{equation}\label{e-gue140731aVI}
\dim \cE^{q}_{0<\mu\leq k^{-N}}(M,L^k)\leq \dim \cE^{q-1}_{k^{-N}}(M,L^k)+\dim \cE^{q+1}_{k^{-N}}(M,L^k).
\end{equation}
From \eqref{e-gue140731aV} and \eqref{e-gue140731aVI}, we have 
\begin{equation}\label{e-gue140802a}
\begin{split}
\int_M\sqrt{{\rm Tr\,}P^{(q)}_{k,0<\mu\leq k^{-N}}(x,x)}dv_M(x)&\leq C_0\sqrt{\int_M{\rm Tr\,}P^{(q)}_{k,0<\mu\leq k^{-N}}(x,x)dv_M(x)}\\
&=C_0\sqrt{\dim \cE^{q}_{0<\mu\leq k^{-N}}(M,L^k)}=o(k^{\frac{n}{2}}),
\end{split}
\end{equation}
where $C_0>0$ is a constant independent of $k$. From \eqref{e-gue140802a} and \eqref{e-gue140731aII}, we conclude that 
\begin{equation}\label{e-gue140802aI}
\mbox{$\lim_{k\To\infty}k^{-n}\abs{T^{(q),f}_k(x,x)-T^{(q),f}_{k,k^{-N}}(x,x)}=0$ in $L^1_{(0,q)}(M)$}.
\end{equation}
In view of Theorem~\ref{localmorse}, we see that 
\begin{equation} \label{e-gue140802aII}
\lim_{k\To\infty}\abs{k^{-n}T^{(q),f}_{k,k^{-N}}(x,x)-
(2\pi)^{-n}\abs{{\rm det\,}\dot{R}^L(x)}f(x)\mathds{1}_{M(q)}(x)I_{\det\ov{W}^{\,*}}(x)}=0
\end{equation} 
in $L^1_{(0,q)}(M)$.
From \eqref{e-gue140802aI} and \eqref{e-gue140802aII}, the theorem follows. 
\end{proof}

\medskip
\noindent
\textbf{\emph{Acknowledgments.}} We are grateful to the referee for several
suggestions which led to the improvement of the presentation.

\end{document}